\documentclass[10pt, letterpaper]{amsart}
\title{Grid homology for singular links in lens space and a resolution cube}
\author{Yonghan Xiao}
\address{Department of Mathematics, School of Mathematical Sciences\\Peking University}
\email{judy\underline{ }xyh0530@stu.pku.edu.cn}
\usepackage[all,pdf]{xy}
\usepackage{geometry}
\usepackage{graphicx} 
\graphicspath{{images/}} 
\usepackage{amsmath}
\usepackage{amssymb}
\usepackage{amsthm}
\usepackage{float}
\usepackage[english]{babel}
\usepackage{tikz}
\usepackage{xcolor}
\usepackage{tkz-euclide}
\usepackage{abstract}
\usepackage[percent]{overpic}
\usepackage{hyperref}
\hypersetup{
	colorlinks,
	citecolor=black,
	filecolor=black,
	linkcolor=black,
	urlcolor=black,
	frenchlinks=true
}

\newtheorem{thm}{Theorem}[section]

\newtheorem{rmk}[thm]{Remark}

\newtheorem{prop}[thm]{Proposition}
\newtheorem{lem}[thm]{Lemma}

\newcommand{\Acknowledgement}{$\mathbf{Acknowledgement}$}

\theoremstyle{definition}
\newtheorem{defi}[thm]{Definition}

\newcommand{\Sbold}{\boldsymbol{S}}

\newcommand{\RR}{\mathbb{R}}

\newcommand{\ZZ}{\mathbb{Z}}
\newcommand{\QQ}{\mathbb{Q}}

\newcommand{\FF}{\mathbb{F}}

\newcommand{\xv}{\boldsymbol{x}}
\newcommand{\yv}{\boldsymbol{y}}
\newcommand{\wv}{\boldsymbol{w}}
\newcommand{\zv}{\boldsymbol{z}}
\newcommand{\AV}{\boldsymbol{\alpha}}
\newcommand{\BV}{\boldsymbol{\beta}}

\newcommand{\TA}{\mathbb{T}_\alpha}
\newcommand{\TB}{\mathbb{T}_\beta}

\newcommand{\sfrak}{\mathfrak{s}}

\newcommand{\DCAL}{\mathcal{D}}

\newcommand{\OV}{\mathbb{O}}
\newcommand{\XXV}{\mathbb{X}}

\newcommand{\BI}{\mathbf{I}}
\newcommand{\BN}{\mathbf{N}}

\begin{document}
	\maketitle
	
	\begin{abstract}
		In this paper, we define grid homologies for singular links in lens spaces and use them to construct a resolution cube for knot Floer homology of regular links in lens spaces. The results will first be proved over $\ZZ/2\ZZ$ and then over $\ZZ$ with the help of sign assignments. We will also identify the signed grid homology and classical knot Floer homology over $\ZZ$ for regular links in lens spaces, illustrating the fact that our resolution cube is genuinely one for knot Floer homology. The main advancement in the paper is that we give a complete description of singular knot theory in lens spaces which was only defined in $S^3$ previously and we construct a signed  combinatorial resolution cube for knot Floer homology in lens spaces which may be powerful in relating $HFK^\circ$ to other link homology theories.
	\end{abstract}
	
	\tableofcontents
	\phantomsection 
	\hypertarget{MyToc}{}  
	
	\section{Introduction}
	
	After Heegaard Floer homology was defined in \cite{MR2113019} and \cite{MR2113020}, knot Floer homology was introduced by Ozsv\'ath and Szab\'o in \cite{MR2065507}, and also independently by Rasmussen in \cite{MR2704683}. In \cite{MR2529302}, Ozsv\'ath, Stipsicz and Szab\'o generalized knot Floer homology to singular knots in $S^3$. Later, Ozsv\'ath and Szab\'o \cite{MR2574747} used this to construct a resolution cube with twisted coefficients for knot Floer homology over $\ZZ/2\ZZ$. Afterward, Manolescu \cite{MR3229041} proved an untwisted version over $\ZZ/2\ZZ$. Recently, Beliakova, Putyra, Robert, and Wagner proved the corresponding result over $\ZZ$ in \cite[Section 4 and 5]{Beliakova2022APO}. With the help of this, Dowlin \cite{MR4777638} constructed a spectral sequence from Khovanov homology to knot Floer homology.
	
	On the other hand, in \cite{46590c91-7ad9-3a60-87b6-4a8ddc2d37fe},  Sarkar and Wang pointed out a criterion that ensures the count of holomorphic disks is combinatorial. A combinatorial version of knot Floer homology (so-called grid homology) first appeared in \cite{MR2480614}. Later, Manolescu, Ozsv\'ath, Szab\'o, and Thurston \cite{MR2372850} generalized the construction to links and extended the base ring to $\ZZ$ via a sign assignment (See also \cite{Gallais2007SignRF}). The book \textit{Grid homology for knots and links} (\cite{MR3381987}) gives a detailed and complete description of this theory. 
	
	Lens spaces are defined as rational surgeries on the unknot in $S^3$, i.e., $L(p,q)=S^3_{-p/q}(U)$. They are the simplest manifolds in terms of Heegaard decomposition other than $S^3$. In \cite{MR2429242}, Baker, Grigsby, and Hedden constructed grid diagrams for knots and links in lens spaces and proved that the combinatorially defined theory coincides with the existing one. In \cite{Celoria_2023}, Celoria constructed a sign assignment for grid diagrams of lens space links, and then the invariance of this oriented theory was proved by Tripp \cite{tripp2021gridhomologylensspace}. In \cite{MR3677933}, Harvey and O'Donnol also considered grid homology for spatial graphs in $S^3$ (Bao also considered Heegaard Floer homology for bipartite graphs in \cite{bao2018floerhomologyembeddedbipartite}). More recently, Kubota further developed properties of grid homology for spatial graphs in \cite{MR4701948}, \cite{Kubota2023SomePO} and \cite{kubota2024gridhomologyspatialgraphs}.
	
	In this paper, we generalize previous constructions to singular links in lens spaces in a combinatorial way. Specifically, we will \begin{itemize}
		\item First, define three versions of grid homology for singular links in lens spaces and prove that they are well-defined invariants;
		\item Then, construct skein exact sequences and a resolution cube for the grid homology of lens space links;
		\item Finally, extend the base ring of our theory from $\ZZ/2\ZZ$ to $\ZZ$ via an appropriate choice of sign assignment.
	\end{itemize}
	
	There were two interpretations of singular links in previous works. Some refer to immersions of a disjoint union of $S^1$'s with transverse double points as singular links (e.g.\cite{MR2491583}), while others refer to embeddings of certain kinds of graphs as singular links. We take the second approach following \cite{MR2529302} and define a singular link in a 3-manifold $Y$ as an embedding of a special oriented trivalent graph (see \text{Definition \ref{singular link1} and Definition \ref{singular link2}}). 
	
	In \cite{MR2574747}, the authors defined Heegaard diagrams for singular links in $S^3$ by introducing $XX$ base points to an usual one. (Also, in \cite{MR3677933}, the authors introduced nonstandard $O$ base points for the same purpose.) Using their strategy, we define a grid diagram for singular links in lens spaces to consist of iterated pairs of $\alpha$ and $\beta$ curves from the standard Heegaard splitting of $L(p,q)$ as background rulings (following \cite{MR2429242}) and $O$, $X$, $XX$ three kinds of base points that record the information of the link. For details, see \text{Definition \ref{grid diagram}}.
	
	After giving a precise definition of a grid diagram for a singular link, we show that each link admits such a diagram by direct construction. In Figure \ref{fig:exofgriddiagramchap1}, we illustrate how to visualize a singular link in lens space on the Heegaard torus and using it to construct a grid diagram. We will come back to this point in Section \ref{Construction of grid diagrams} and give a detailed explanation on this. \begin{figure}
		\centering
		\begin{overpic}[width=0.8\textwidth]{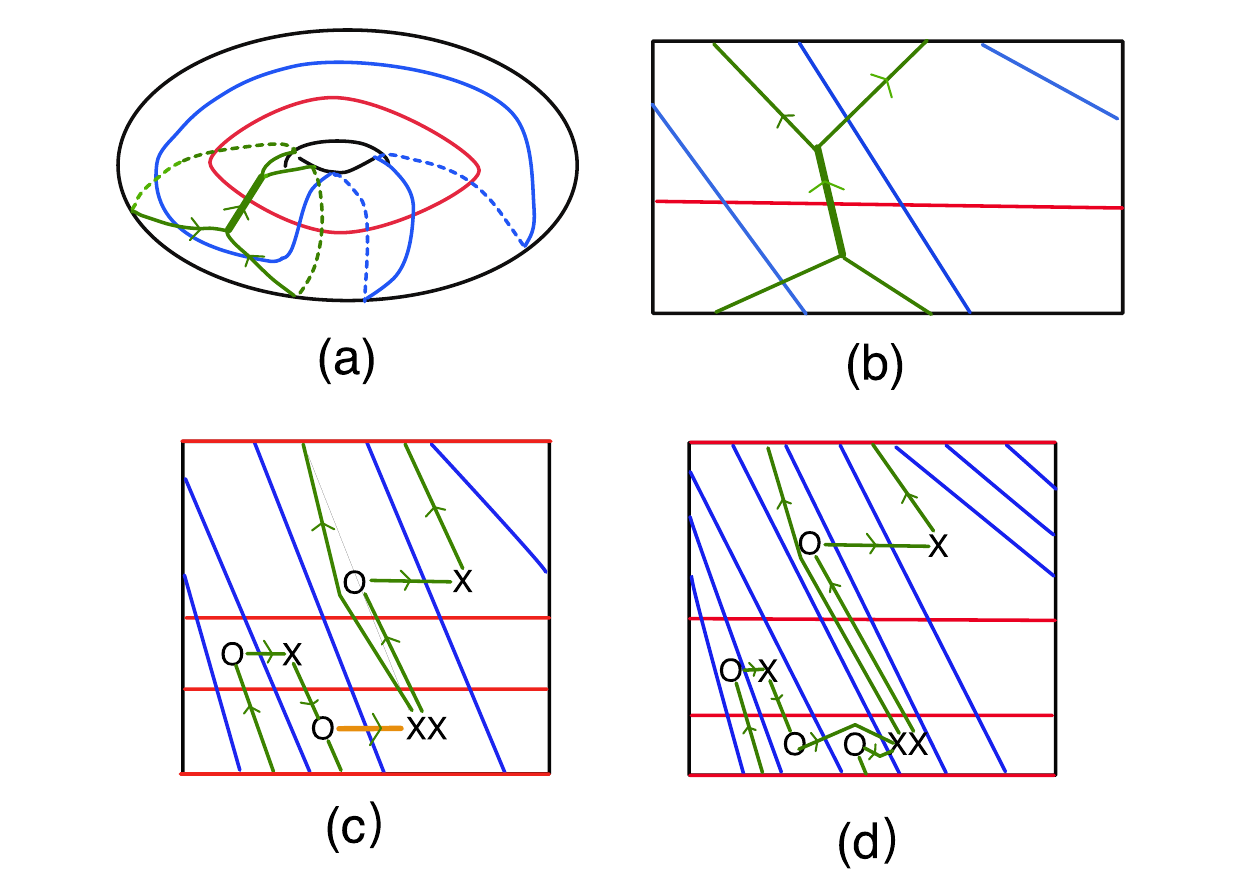}
			\put(33,62) {$\alpha$}
			\put(35,50) {$\beta$}
			
			\put(80,56) {$\alpha$}
			\put(70,60) {$\beta$}
			
		\end{overpic}
		
		\caption{\textbf{An example of grid diagram}}
		\label{fig:exofgriddiagramchap1}
	\end{figure}
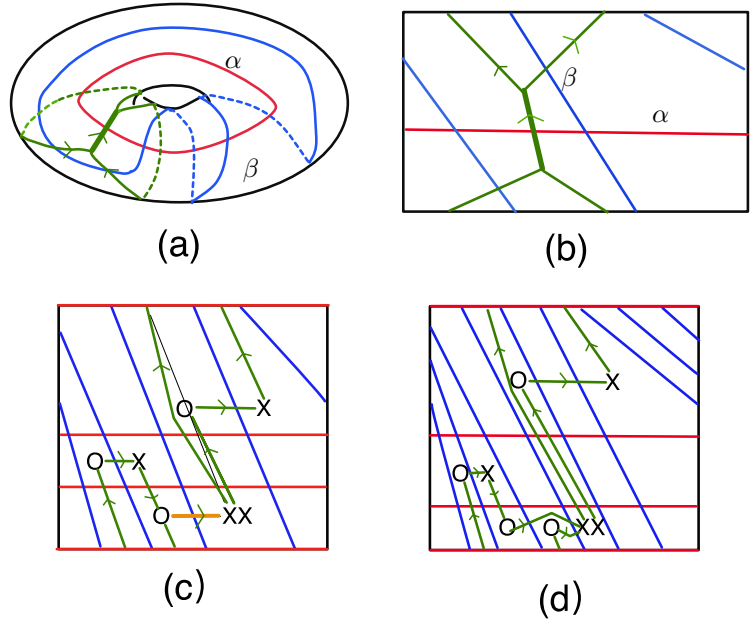
	
	Then, we show Reidemeister's theorem for torus diagrams of singular links (a generic projection to Heegaard torus) and prove that two grid diagrams represent the same singular link if and only if they can be connected by a sequence of grid moves.
	\begin{thm}
		Consider $L$ and $L'$ to be singular links in $L(p,q)$ with the same underlying graph, the following are equivalent:
		\begin{enumerate}
			\item $L$ and $L'$ are equivalent;
			\item $L$ and $L'$ differ by a finite sequence of disk moves (\text{Definition \ref{disk move}});
			\item Regular projections (\text{Definition \ref{regular projection}}) of $L$ and $L'$ to the Heegaard torus differ by a finite sequence of singular Reidemeister moves shown in \text{Figure \ref{fig:rmoves-new}} and disk slides (\text{Definition \ref{disk slide}}). 
		\end{enumerate}
		
	\end{thm}

	\begin{thm}
		If two grid diagrams $g$ and $g'$ represent the same singular link $L$, then they can be connected by a finite sequence of commutation and (de)stabilization.
	\end{thm}
	See \text{Subsection \ref{Moves on torus diagram}, \ref{Moves on grid diagram}} for details on these moves.

	With diagrams in hand, we follow the usual procedure in defining Heegaard Floer homology: form the symmetric product of the Heegaard surface with multiplicity according to its size, use intersection points between real tori formed by $\alpha$ and $\beta$ curves as generators, and define differentials by counting empty rectangles. We will consider three kinds of chain complexes: $CFK^-$, $\widehat{CFK}$, and $\widetilde{CFK}$, and prove the following theorem in \text{Section \ref{Proof of invariance}}.
	
	\begin{thm}
		Fix a singular link $L$ in some lens space $L(p,q)$ with $n$ thick edges and $m$ regular components. If $g_1$ and $g_2$ are  grid diagrams for $(L(p,q),L)$, then we have quasi-isomorphisms \begin{itemize}
			\item $CFK^-(g_1)\simeq CFK^-(g_2)$;
			\item $\widehat{CFK}(g_1)\simeq \widehat{CFK}(g_2)$,
		\end{itemize}
		as relatively bigraded chain complexes of modules over polynomial ring $\FF[U_1,\ldots, U_{2n+m}]$ and over $\FF$ respectively.  
		In particular, we have well-defined homology theories for singular links in lens spaces: \begin{itemize}
			\item Unblocked grid homology: $HFK^-(L)$ as relatively bigraded modules over $\FF[U_1,\ldots, U_{2n+m}]$ and 
			\item Simply blocked grid homology: $\widehat{HFK}(L)$ as relatively bigraded modules over $\FF$.
		\end{itemize}
		Here $\FF=\ZZ/2\ZZ$ or $\ZZ$.
	\end{thm}

	Having defined grid homology for singular links in lens spaces, we consider a resolution cube for grid homology of links in lens spaces, i.e., generalize the result in \cite{MR2574747} via a purely combinatorial approach. We will introduce the notion of a special grid diagram (\text{Definition \ref{special grid diagram}}) in which we fix a way of visualizing a link on a grid diagram such that each crossing looks like one of the standard pictures shown in \text{Figure \ref{fig:standard picture at each crossing}}. Using such a diagram, we perform resolution of crossings (see \text{Figure \ref{fig:resolution of crossings}}) diagrammatically and prove a skein exact sequence.
	
	\begin{thm}
		Consider a special diagram $g$ for a pair $(L(p,q),L)$. Choose a crossing $c$ in $g$. Denote the resulting grid diagram (using the procedure in \text{Figure \ref{fig:grid diagram of resolution and singularization}}) for smoothing and singularization at $c$ by $g_r$, $g_s$ and call the corresponding links $L_r$ and $L_s$, respectively. Let $O_a$, $O_b$ be marks sharing the same column with the newly formed $XX$ base point, and $O_c$, $O_d$ be marks sharing the same row with the newly formed $XX$ base point. Then we have the following:
		\begin{itemize}
			\item When c is positive, there is an exact sequence:
			\begin{equation}
				\ldots\longrightarrow HFK^-(L)\longrightarrow H_*(\frac{CFK^-(L_s)}{U_a+U_b-U_c-U_d}) \longrightarrow HFK^-(L_r)\longrightarrow HFK^-(L)\longrightarrow\ldots
			\end{equation}
			
			\item When $c$ is negative, there is an exact sequence:
			\begin{equation}
				\ldots\longrightarrow HFK^-(L) \longrightarrow HFK^-(L_r)\longrightarrow  H_*(\frac{CFK^-(L_s)}{U_a+U_b-U_c-U_d}) \longrightarrow HFK^-(L)\longrightarrow\ldots
			\end{equation}
			
		\end{itemize}
	\end{thm}
	In Figure \ref{fig:AB diagramchap1}, we show a combined grid diagram for resolutions of crossings. (Left one for negative and right one for positive.) This nice local picture will be used to prove the skein exact sequence in Section \ref{Resolution cube}.
	
	\begin{figure}
		\begin{overpic}[width=0.80\textwidth]{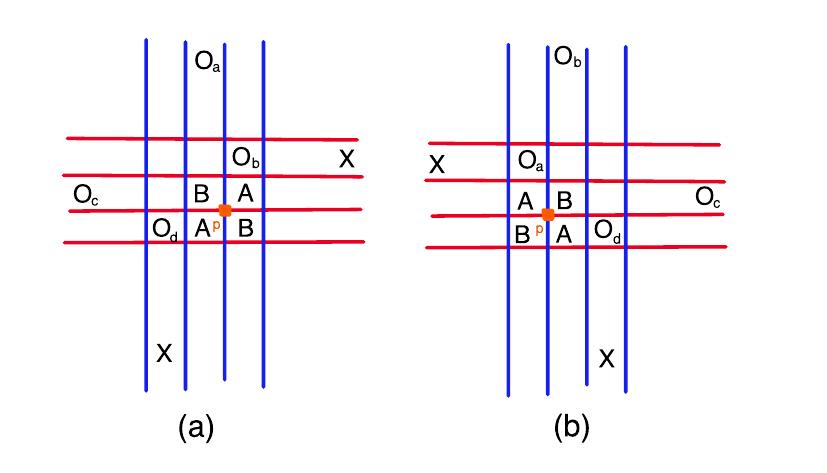}
		\end{overpic}
		
		\caption{\textbf{Combined diagram for resolution of crossing.}}
		\label{fig:AB diagramchap1}

	\end{figure}
	By considering resolutions of all crossings in a diagram, we can form a resolution cube. Using the skein relations and a direct computation of higher differentials, we show the following.

	\begin{thm}
		Let $g$ be a special grid diagram for some possibly singular link $L$ in $L(p,q)$. Assume at each crossing, $g$ looks like the standard picture in \text{Figure \ref{fig:standard picture at each crossing}}. Then there is a spectral sequence converging to $HFK^-(g)$ whose $E_1$ term is the resolution cube with $d_1$ induced by some homomorphism from the skein exact sequence.
	\end{thm}
	
	We will also consider all these theories over $\ZZ$ via a sign assignment for lens space grid diagrams constructed in \cite{Celoria_2023}. After describing how to assign a sign to each rectangle in a grid diagram, we will reprove all the results over $\ZZ$. Using the universal coefficient theorem, this construction allows us to consider grid homology over any base field. In particular, one can use $\QQ$ as the base ring so that we can relate our theory to existing constructions that are only defined over rational numbers.
	
	In \cite{Sarkar2010ANO}, Sarkar introduced the concept of weak equivalence between coherent orientation systems and the concept of weak equivalence between sign assignments. It was demonstrated that for an $l$-component link in $S^3$, both kinds of equivalence classes form a set of size $2^{l-1}$ and there is a natural bijection between these sets. In \cite{eftekhary2018signassignmentslinkfloer}, Eftekhary showed that orientation systems with negative boundary degeneration can be identified with true sign assignments for grid homology in $S^3$, which was introduced in \cite{MR2372850}. Note that our sign assignments coincide with theirs when the background manifold is $S^3$. This contrasts with the positive ones considered in \cite{MR3412088} which will lead to `false' sign assignments according to their notation.  Despite this distinction, one can easily check that true and false assignments are weakly equivalent in the sense of \cite{Sarkar2010ANO}. Similarly, orientation systems with negative and positive boundary degeneration are also weakly equivalent in sense of \cite{Sarkar2010ANO}. Consequently, one can generalize the argument presented in \cite{Sarkar2010ANO} to see that grid homology with true assignment is isomorphic to knot Floer homology with canonical orientation system, as defined in the context of \cite{MR3412088}. In fact, all these concepts could be generalized to lens spaces, and the last isomorphism remains true in this broader context. We will prove the following theorem in the last section as a simple illustration of this and  we will come back to this point in a future work.
	\begin{thm}
		Various versions of grid homologies for knots and links in lens spaces defined using true and false assignments over $\ZZ[U_1,\ldots,U_{2n+m}]$ coincide.
	\end{thm}

	\begin{rmk}(Convention about integer coefficients)
		  Using these facts, our result can be regarded as a resolution cube for the classical knot Floer homology, rather than being limited solely to the combinatorially defined grid homology. Additionally, we would like to highlight that the canonical orientation system from \cite{MR3412088} was used by Dowlin in \cite{MR4777638} to describe the oriented resolution cube of knot Floer homology, so our result, when restricted to $S^3$, gives rise to a combinatorial description of that cube. We hope this may be helpful for further application.
	\end{rmk}

	The paper is organized as follows: In \text{Section \ref{Singular link and grid diagram}}, we first define singular links and construct grid diagrams for them; then we show Reidemeister theorem for torus projections and introduce a complete set of grid moves. In \text{Subsection \ref{Construction of chain complex}-\ref{Definition of homology theories}}, we define the homology theories and introduce two grading systems on them. In \text{Subsection \ref{Computing examples}}, we compute fully blocked and simply blocked grid homology for some singular knots in $L(2,1)$ and $L(3,1)$. In \text{Subsection \ref{Proof of invariance}}, we demonstrate grid move invariance of our theories, which shows that they are well-defined link invariants. In \text{Section \ref{Resolution cube}}, we first introduce resolutions of a crossing and then prove the existence of the skein exact sequence and the resolution cube mentioned above. Finally, we extend all the results to have base ring $\ZZ$ in \text{Section \ref{An oriented version}}. In \text{Subsection \ref{Orientation system versus sign assignment}}, we generalize the argument from \cite{Sarkar2010ANO} and \cite{eftekhary2018signassignmentslinkfloer} to show that when using a grid diagram as a Heegaard diagram, the grid homologies over $\ZZ$ we defined here coincide with the oriented theory defined using orientation system constructed in \cite{MR3412088}. This justifies our notation ``$HFK^\circ$'' and relates our construction to the original knot Floer theory.
	

	\Acknowledgement 
	
	First of all, the author would like to thank her advisor Prof. Jiajun Wang for his detailed instruction and also Prof. Yi Ni for his helpful advice and encouragement. Secondly, the author thanks  Ioannis Diamantis, Hongjian Yang,and Yicheng Yang for their helpful discussion. She would also express sincere gratitude to Prof. Jianfeng Lin and Prof. Yi Xie, who exposed her to grid homology and Dowlin spectral sequence, respectively, both of which motivate this work. Last but not least, the author appreciates Qing Lan,Quan Shi, and Jinyi Wang for their assistance with computer programming, Zhenkun Li and Fan Ye for their instruction in writing and drawing, as well as her family members for their unconditional support.
	
	\section{Singular link and grid diagram}\label{Singular link and grid diagram}
	\subsection{Definitions}
	
	\begin{defi}\label{special oriented trivalent graph}
		An oriented trivalent graph consisting of vertices and thick/thin edges is \emph{special} if 
		\begin{itemize}
			\item at each vertex, there is exactly one thick and two thin edges.
			\item the orientation satisfies that at each vertex, either the thick edge points in and the two thin edges point out, or the thick edge points out while the two thin edges point in.
		\end{itemize}
		Here, we do not require the graph to be connected, and closed components with a single thin circle and no vertex are allowed.(We will call such ``circles'' regular $S^1$ components, and we require these to be oriented if they exist.)
	\end{defi}
	Note that not every trivalent graph with thick edges can be endowed with an orientation so that it becomes special. We only take oriented ones into our consideration. See \text{Figure \ref{fig:singular-links}} for examples and non-examples.
	
	\begin{figure}
		\centering
		\includegraphics[width=0.6\linewidth]{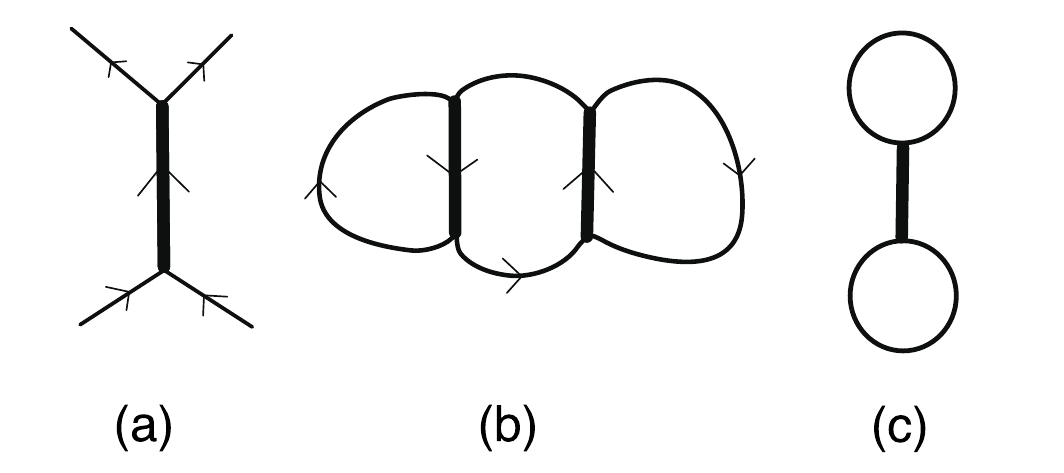}
		\caption{\textbf{Examples and a non-example of special trivalent graph.} (a) orientation at each thick edge; (b) an example of special trivalent graph; (c) non-orientable example}
		\label{fig:singular-links}.
	\end{figure}
	\begin{defi}\label{singular link1}
		A \emph{singular link} $L$ inside a 3-manifold $Y$ is a (piecewise linear) embedding of a special trivalent graph into $Y$. 
	\end{defi}
	
	For further use, we give an alternative definition of singular link as a transverse spatial graph following \cite{MR3677933}. Let $\DCAL$ be the disk endowed with a triangulation as shown in \text{Figure \ref{fig:standard-disk}} which has four vertices, six edges and three 2-simplices. There is a unique vertex $v_0$ lying in the interior of $\DCAL$. In this case, we say that $\DCAL$ is a standard disk and $v_0$ is the vertex associated to $\DCAL$.
	
	\begin{figure}
		\begin{overpic}[width=0.3\textwidth]{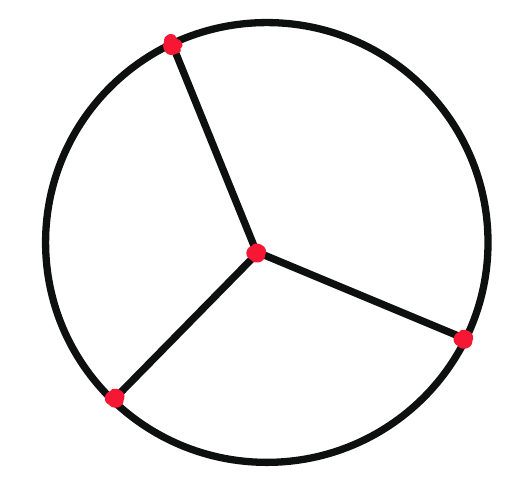}
			\put(15,28) {$v_2$}
			\put(46,12) {$e_6$}
			\put(44,70) {$e_1$}
			\put(82,38) {$v_3$}
			\put(24,88) {$v_1$}
			\put(62,47) {$e_2$}
			\put(30,45) {$e_3$}
			\put(48,40) {$v_0$}
			\put(0,50) {$e_4$}
			\put(86,72) {$e_5$}
			
		\end{overpic}
		
		\centering
		\caption{\textbf{A standard disk with its triangulation.}}
		\label{fig:standard-disk}
	\end{figure}
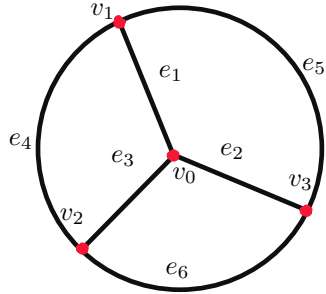
	
	Take any oriented graph $G_0$, we construct an oriented disk graph $G$ associated to $G_0$ as follows. For each vertex $v$ of $G_0$, glue a standard disk $\DCAL$ to $G_0$ by identifying the vertex associated to $\DCAL$ to $v$. In this case, we say $G_0$ is the underlying oriented graph of $G$, since it appears naturally as part of the 1-skeleton of $G$. A vertex (resp. edge) of $G_0$ will be called a graph vertex (resp. edge). Here we allow $G_0$ to have closed components, each consisting of a single circle and no vertex. Again, such circles will be referred to as regular $S^1$ components and no disk is glued to them when forming the associated disk graph.
	
	\begin{defi}\label{singular link2}
		A \emph{singular link} in a 3-manifold $Y$ is a transverse embedding of a 4-valent oriented disk graph with exactly two graph edges going out and two graph edges pointing in at each graph vertex. Here, the transverse embedding means that at each graph vertex, the local picture must look like \text{Figure \ref{fig:transverse-disk-1}}: the vertex has a closed $B^3$ neighborhood in $Y$ such that its standard disk embeds properly in this copy of $B^3$ and separates the incoming and outgoing edges.
	\end{defi}

	\begin{figure}
		\centering
		\includegraphics[width=0.35\linewidth]{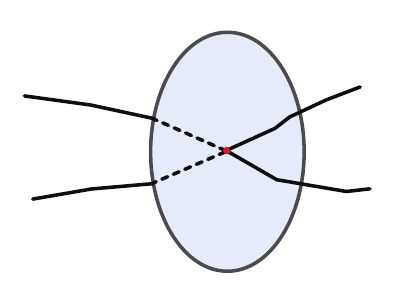}
		\caption{\textbf{A transverse disk that separates incoming and outgoing edges}}
		\label{fig:transverse-disk-1}
	\end{figure}
	
	In this second definition, though we are embedding a 2-dimensional object into $Y$, we will regard the underlying oriented graph as a singular link $L$ and remember the standard disks as auxiliary information that adds restrictions to the isotopies allowed.
	
	To see that the definitions are equivalent, we consider the following operations: First, assume we are in \text{Definition \ref{singular link1}}. For each thick edge, take a small disk centered at the midpoint of the thick edge, ensuring it is perpendicular to the edge and disjoint from all other edges or $S^1$ components. Then, by contracting the thick edges to their midpoints, we obtain the picture in the second definition. Going backward, we replace each graph vertex by a thick edge together with two new vertices, so that the thick edge meets the disk transversely at its center. The thick edge admits a unique orientation that fits the new graph into \text{Definition \ref{special oriented trivalent graph}}. See \text{Figure \ref{fig:equiv-between-2-definitions}} for an illustration. We remark that a special trivalent graph can also be regarded as a bipartite graph defined in \cite{bao2018floerhomologyembeddedbipartite}. 
	\begin{figure}
		
		\centering
		\includegraphics[width=0.6\linewidth]{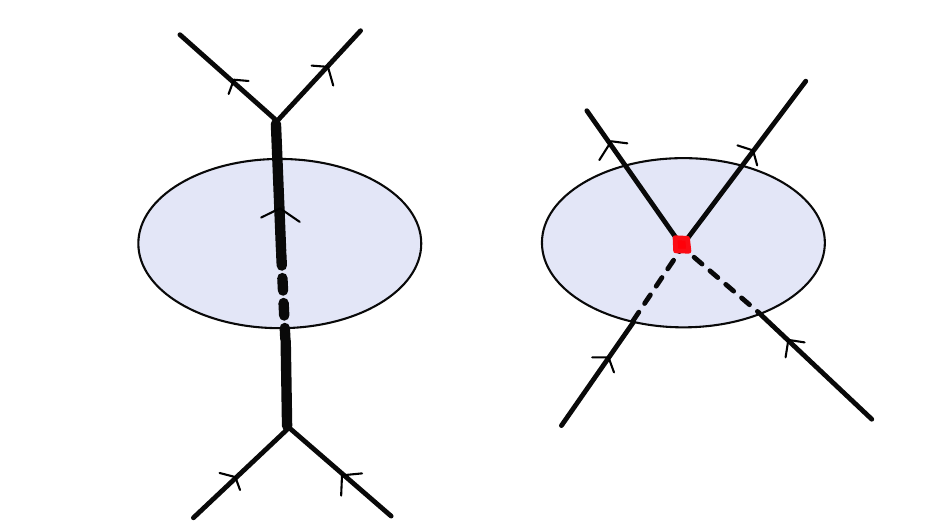}
		\caption{\textbf{Transformation between two definitions of singular links.} By performing local moves at each vertex, we can transition between thick edge convention(left) and transverse graph convention(right).}
		\label{fig:equiv-between-2-definitions}
	\end{figure}
	
	We say a singular link is a singular knot if the underlying graph is connected.

	\begin{defi}
		Fix a special trivalent graph or oriented disk graph $L$. We say two singular links $i_1:L\to Y$ and $i_2:L\to Y$ are \emph{equivalent} if there exists a piecewise linear homeomorphism $f: Y\to Y$ that is isotopic to the identity such that $i_1\circ f=i_2$.
	\end{defi}
	In what follows, we always take $Y=L(p,q)$ for a pair of integer $p,q$ satisfying $gcd(p,q)=1$ and $-p<q<p$.
	
	\subsection{Construction of grid diagrams}\label{Construction of grid diagrams}
	Let $L$ be a (possibly) singular link with $n$ thick edges in some lens space $L(p,q)$. We can assume without loss of generality that $gcd(p,q)=1$ and $-p<q<p$. Our goal in this subsection is to define and construct grid diagrams for such pairs $(L(p,q),L)$.

	Grid diagram currently exists for possibly singular links in $S^3$ and regular links in $L(p,q)$, see \cite{MR2480614},\cite{MR2429242}, and \cite{MR2491583}. In \cite{MR3677933}, they also defined grid diagrams for transverse spatial graphs in $S^3$. We shall show it also exists for singular links in $L(p,q)$. 
	
	Consider the standard Heegaard splitting of $L(p,q)$ into two genus one handlebodies: $L(p,q)=V_\beta\cup_{\Sigma}V_\alpha$. Here $\Sigma=T^2$ is a genus one closed surface. We choose a pair of curves $\mu,\lambda$ on it that are smooth generic representatives of a symplectic basis of its first homology. See \text{Figure \ref{fig:heegaard-torus}}. A curve on $\Sigma$ is said to have slope $\frac{p}{q}$ if it is homologous to $p\mu+q\lambda$. 
	\begin{figure}
		\begin{overpic}[width=0.4\textwidth]{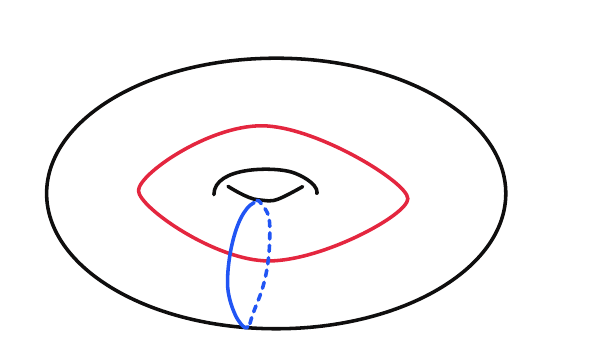}
			\put(60,35) {$\lambda$}
			\put(46,10) {$\mu$}

		\end{overpic}
		\centering
		\caption{\textbf{Meridian and longitude on a Heegaard torus}}
		\label{fig:heegaard-torus}
	\end{figure}
	
	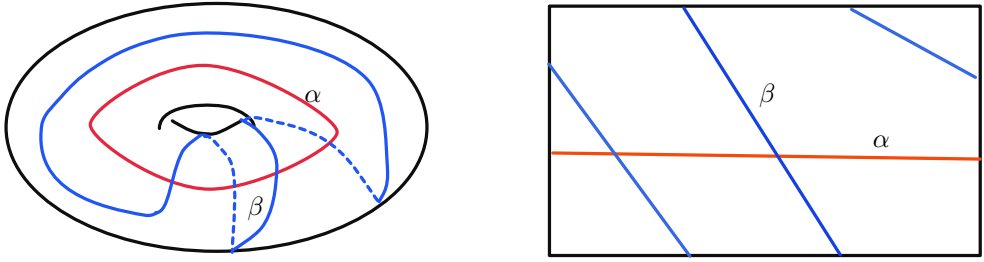
\begin{figure}
		\centering
		\begin{overpic}[width=1\textwidth]{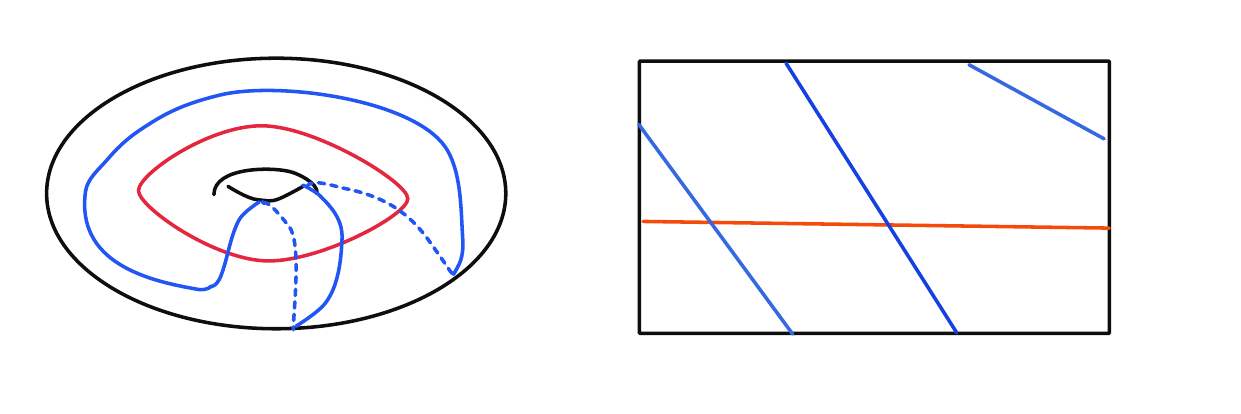}
			\put(30,20) {$\alpha$}
			\put(25,10) {$\beta$}
			
			\put(80,16) {$\alpha$}
			\put(70,20) {$\beta$}
		\end{overpic}
		\caption{\textbf{Standard Heegaard splitting for $RP^3$.} The left is the torus picture,  and the right is the one after cutting open along a pair of meridian and longitude. }
		\label{fig:heegaard-splitting-of-rp3}
	\end{figure}
	
	In this terminology, the standard Heegaard splitting of $L(p,q)$  is given by a $\beta$ curve with slope $-\frac{p}{q}$ and an $\alpha$ curve with slope 0. See \text{Figure \ref{fig:heegaard-splitting-of-rp3}} for $L(2,1)=RP^3$ as an example.
	
	From now on, we fix the convention that in a grid diagram, $\alpha$ curves will be red, and $\beta$ curves will be blue.
	
	\begin{defi}\label{grid diagram}
		A \emph{grid diagram} is a multi-based Heegaard diagram $(T^2, \AV,\BV,\OV,\XXV)$ of $L(p,q)$ as follows:
		\begin{itemize}
			\item The background rulings are given by same number of parallel copies of $\beta$ and $\alpha$ curves. In such a diagram, an annulus cut out by adjacent parallel $\alpha$ (resp. $\beta$) circles is called a row (resp. column).  $\AV$ and $\BV$ will denote the set of $\alpha$ and $\beta$ curves respectively. A connected component in $\Sigma-\AV-\BV$ will be referred to as a cell.
			
			\item The base points are put in cells so that each cell has at most one base point. Each row or column is either singular or regular in the following sense: \begin{itemize}
				\item In a regular row or column, we have one pair of $O$ and $X$ base points;
				\item In a singular row or column, we have one $XX$ base point together with two $O$ base points.
			\end{itemize}  
		\end{itemize}
		
	\end{defi}
	
	It follows from the definition that in a grid diagram, we have the same number of $O$, $X$ base points (with $XX$ base points counting twice) and the same number of rows and columns.
	
	Each grid diagram $g$ can be associated with a unique singular link $L$. That is, we can reconstruct a torus projection (See \text{Definition \ref{grid diagram of a pair}}) of $L$ and recover the embedding of $L$ in $L(p,q)$ from $g$ as follows.
	\begin{enumerate}
		\item Connect each $O$ base point to the $X$ or $XX$ base points in the same row.
		\item Connect each $X$($XX$) base point to the single (or two) $O$ base point in the same column.
	\end{enumerate}
	Here, we follow the rule that horizontal segments always go under those slanted (vertical) ones.
	
	Note that each $XX$ base point plays the role of two regular $X$ base points, i.e., there are two segments pointing to it and two segments going out from it.
	\begin{defi}\label{grid diagram of a pair}
		Now that we have a link diagram living on the Heegaard surface, it can be regarded as the projection of $L$ to $\Sigma$ using the Morse flow of a generic Morse function compatible with the standard Heegaard splitting. We call this picture a \emph{ torus projection} of $L$. 
		
		By further pushing the horizontal segments into $V_\alpha$ and those slanted ones into $V_\beta$, we recover the singular link $L$ in $L(p,q).$ 
		
		In this case, we say $g$ is a \emph{grid diagram for the pair} $(L(p,q),L)$.
	\end{defi}
	
	\begin{prop}
		Every singular link in a lens space $L(p,q)$ admits a grid diagram.
	\end{prop}
	
	The proposition will be proved by direct construction.
	
	In what follows, we fix identifications of $V_\alpha$ and $V_\beta$  with $S^1\times D^2$, so that we can talk about meridian disks without ambiguity. We temporarily focus on \text{Definition \ref{singular link1}} of singular links. 
	
	\begin{defi}
		A singular link $L$ in a lens space $L(p,q)$ is called \emph{in grid position} if
		\begin{enumerate}
			\item All vertices lie on $\Sigma$;
			
			\item Each thin edge or $S^1$ component of $L$ meets  $\Sigma$ transversely in $n\ge 1$ times in its interior so that it is composed of finitely many arcs with ends on $\Sigma$.  Each arc is required to be properly embedded in a meridian disk $\{pt\} \times D^2$ of $V_\beta$ or $V_\alpha$. As we travel along the edge, arcs in $V_\alpha$ and $V_\beta$ appear alternately.
			
			\item The interior of each thick edge lies entirely in a meridian disk of $V_\alpha$ or $V_\beta$.
			
			\item When two thin arcs meet at a vertex, we require them to share a meridian disk which contains no other arcs. Also, if a thick edge has its interior in $V_\alpha$ (resp. $V_\beta$), then the four relevant thin arcs all lie in $V_\beta$(resp. $V_\alpha$).
			
			\item Besides the case in (4), no two arcs share a meridian disk.
			
			\item Each thick edge is short enough so that the two meridian disks containing arcs from its two vertices are adjacent in $V_\beta$ or $V_\alpha$. Here adjacent means that if the disk containing some arcs from $L$ in $V_\gamma$ ($\gamma=\alpha,\beta$) are $\{a_j\}\times D^2$, for $0\le a_1< \ldots< a_k<1$ (parametrize $S^1$ by $[0,1)$), then the two disks are $\{a_i\}\times D^2$ and $\{a_{i+1}\}\times D^2$ for some $i$ or $\{a_1\}\times D^2$ and $\{a_{k}\}\times D^2$.
			
		\end{enumerate}	
		
	\end{defi}
	
	After an isotopy, it is clear that any link in $L(p,q)$ can be put into grid position. So we assume $L$ is already in such a position. Cut $\Sigma$ open along a pair of generic $\lambda$ and $\mu$ curves (generic in the sense that they meet the projection of $L$ to $\Sigma$ transversely). Here. we cut $\Sigma$ open for the convenience of drawing a planar diagram, but readers should keep in mind that the projection of $L$ lives naturally on the Heeagaard torus.
	
	Then we rule the planar diagram with slanted $\beta$ curves and horizontal $\alpha$ curves so that each row contains exactly one meridian disk of $V_\alpha$  which has one or two arcs from $L$ and each column contains exactly one meridian disk of $V_\beta$  which has one or two arcs from $L$. The case of two only happens near a vertex.

	Afterwards, we put base points on the cut-open $\Sigma$. We put an $O$ marks in a cell if there is an arc in $V_\alpha$ leaving it or an arc in $V_\beta$ pointing to it. Similarly, a $X$ mark is put in a cell if there is an arc in $V_\beta$ leaving it or an arc in $V_\alpha$ pointing to it. 
	
	Now, in each row or column, we have two or three marks and at least one of each kind. For the shadow of each thick edge, we would have some consequent columns or rows with three marks and pair of $O$, $X$ in adjacent cells. Add a thick arc connecting this pair of marked points. See (a) in \text{Figure \ref{fig:new-grid-construction}} for example. 
	
	To change this into a desired picture, we introduce $XX$ base points. See \text{Figure \ref{fig:new-grid-construction}}, in which we show how to modify the picture at hand locally into a grid diagram.
	
	\begin{figure}
		\centering
		
		\begin{overpic}[width=0.8\textwidth]{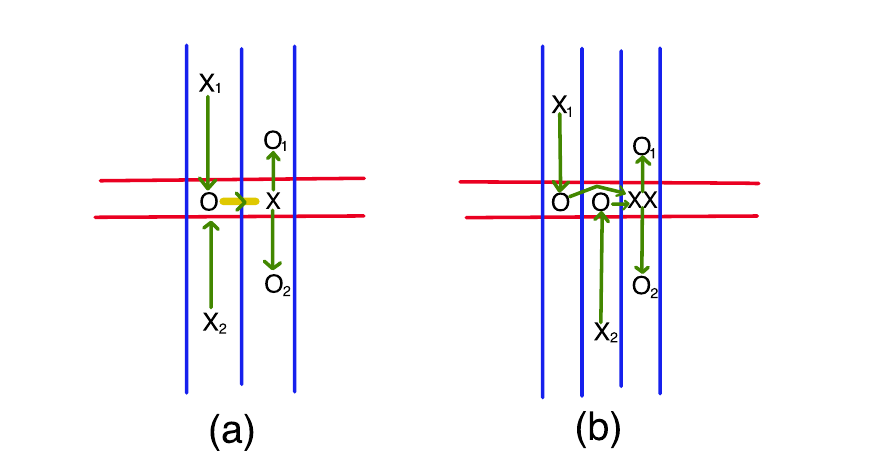}
			
		\end{overpic}
		\caption{\textbf{Making a grid position projection into a grid diagram.} In (a), we show how the projection of a link in grid position looks in a neighborhood of thick edge. Now add one column and one $O$ mark to the diagram and dispense $X_1$ and $X_2$ to the column of two $O$ base points in an arbitrary way. Then, replace the $X$ with an $XX$, we get to (b) which is the desired way of exhibiting a thick edge in grid diagram}
		\label{fig:new-grid-construction}
	\end{figure}
	
	In \text{Figure \ref{fig:ex-of-grid-diagram}}, we give a detailed example illustrating the procedure from a singular link in grid position to a grid diagram using a singular knot with one thick edge in $RP^3$. 
	
	\begin{figure}
		\centering
		\begin{overpic}[width=0.9\textwidth]{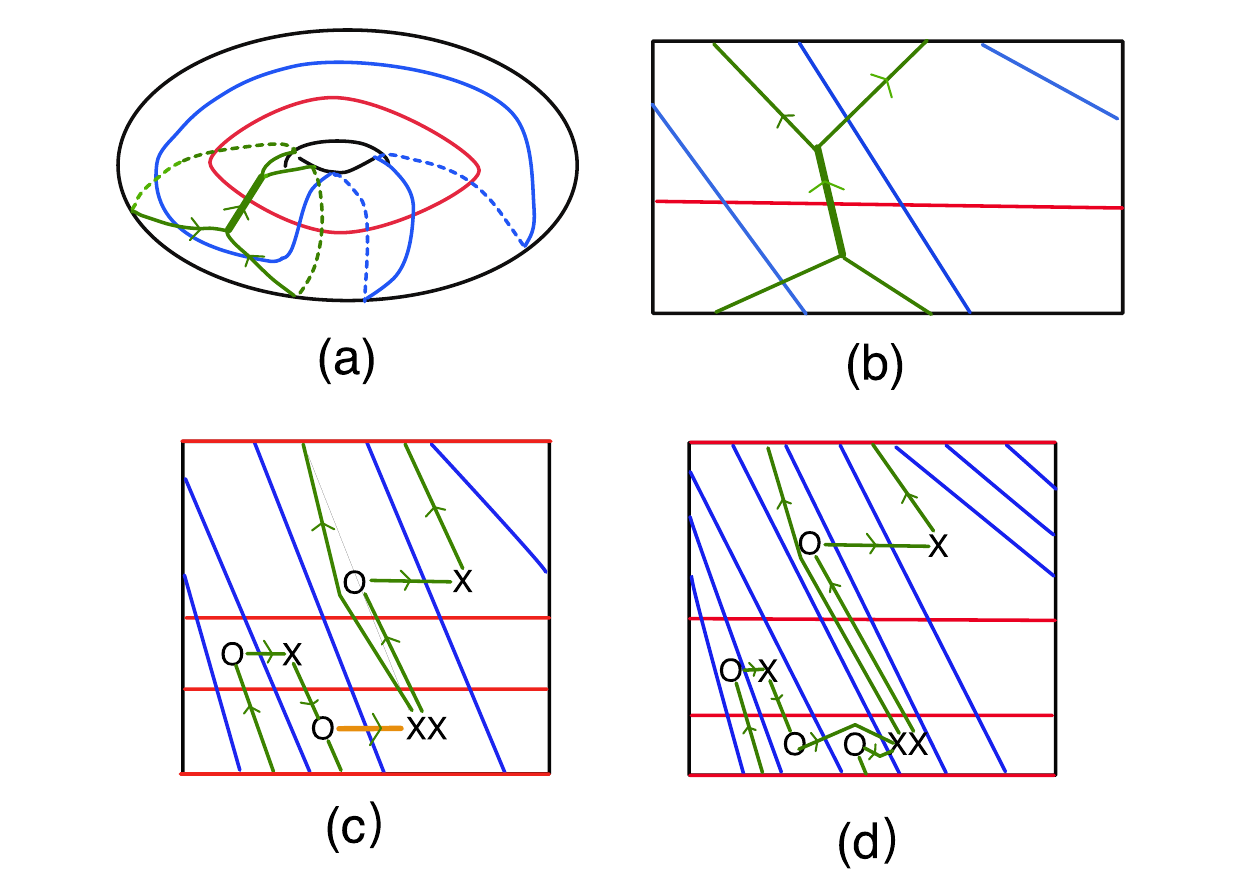}
			\put(33,62) {$\alpha$}
			\put(35,50) {$\beta$}
			
			\put(80,56) {$\alpha$}
			\put(70,60) {$\beta$}
			
		\end{overpic}
		
		\caption{\textbf{An example of a grid diagram} (a) a torus projection of a singular knot with one thick edge in $RP^3$; (b) the planar diagram of the same knot after cutting the torus open along a pair of meridian and longitude; (c) a new projection after putting the knot into grid position; we add new $\alpha$ and $\beta$ curves according to how many meridian disks the knot intersects; (d) modify (c) using the procedure in \text{Figure \ref{fig:new-grid-construction}} to get a grid diagram. }
		\label{fig:ex-of-grid-diagram}
	\end{figure}

	\subsection{Moves on torus diagrams}\label{Moves on torus diagram}
	
	To define an invariant of singular links based on those grid diagrams constructed above. We need to know how different grid diagrams of the same isotopic class of links are related. Since our grid diagram is constructed based on a projection to the Heegaard torus, we first show an analogue of the classical Reidemeister theorem (for planar diagrams of knots in $S^3$). Here, our proof modifies the one of the usual Reidemeister theorem in \cite[Appendix]{Kawauchi1996ASO}. 
	
	Since a singular link is defined as an embedding of a certain graph, it is naturally an object in the piecewise linear category. So, we will live in this category throughout this section and always endow $L$ with a simplicial structure, which is a subdivision of the natural one on a graph. When using \text{Definition \ref{singular link2}}, a singular link $L$ will mean the image of $G_0$, although those standard disks are important for us to consider what kinds of isotopy are allowed.
	
	We first introduce some concepts that are necessary for describing a Reidemeister theorem.
	
	\begin{defi}\label{regular projection}
		Recall that we are considering the projection of a singular link $L$ onto Heegaard torus $T^2$ induced by the flow of a generic Morse function which is compatible with the standard Heegaard splitting of lens spaces. By perturbing it a little, we may assume the projection map is piecewise linear. Such a projection is called \emph{regular} if \begin{itemize}
			\item The set of multiple points in the image $p(L)$ consists of  finitely many double points. 
			\item No points in the preimage $p^{-1}(c)\cap L$ of any double point $c\in p(L)$ are vertices of $L$.
		\end{itemize} 
	\end{defi}
	\begin{defi}\label{disk move}
		For a link $L$ and an embedded disk $D$ in $L(p,q)$ satisfying $L\cap D=L\cap\partial D$, which is an arc, we can form a new link $L'=cl(L-L\cap D)\cup (\partial D-L\cap D)$. In this case, we say $L'$ is obtained from $L$ by performing a \emph{disk move}. It is clear that such two links are equivalent.
		
		By subdividing $L$ and $D$, we can assume the arc $\partial D\cap L $ is a union of 1-simplices while $D$ is a union of 2-simplices. Thus, disk moves can be done in the piecewise linear category.
		
	\end{defi}

	The set of Reidermeister moves for torus diagrams of singular links in lens spaces is shown in \text{Figure \ref{fig:rmoves-new}} following those defined for spatial graphs in \cite{MR946218}. 
	\begin{figure}
		\centering
		\includegraphics[width=0.85\linewidth]{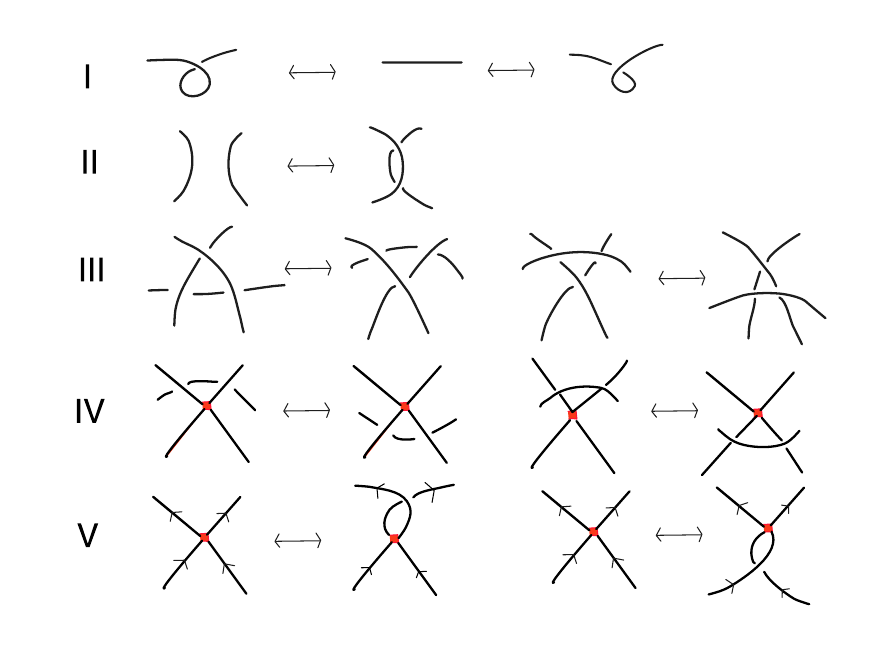}
		\caption{\textbf{Singular Reidemeister moves I-V}}
		\label{fig:rmoves-new}
	\end{figure}
	The first three moves are the usual ones, while IV and V account for newly introduced vertices. Due to definition difference compared with Kauffman's setup, here Reidemeister move V can only happen between edges with the same orientation at a vertex.
	
	\begin{defi} \label{disk slide}
		Let $D$ be a generic projection of some links $L$ in $L(p,q)$ to the Heegaard diagram $T^2$ via the flow of a Morse function compatible with the Heegaard splitting. Equip the torus with two generic representatives $\alpha_0$ and $\beta_0$ of attaching curves of 1-handle and 2-handle, respectively. We can construct a new torus diagram $D'$ for $L$ by replacing an arc in $D$ that contains only over crossings (resp. under crossings) with its complement in a parallel copy of $\beta_0$(resp.$\alpha_0$). In this case, we say $L'$ differ from $L$ by a \emph{$\beta$-disk slide} (\emph{$\alpha$-disk slide}). See Figure \ref{fig:disk slide} for an illustration.
	\end{defi}
	\begin{figure}
		\centering
		\includegraphics[width=0.85\linewidth]{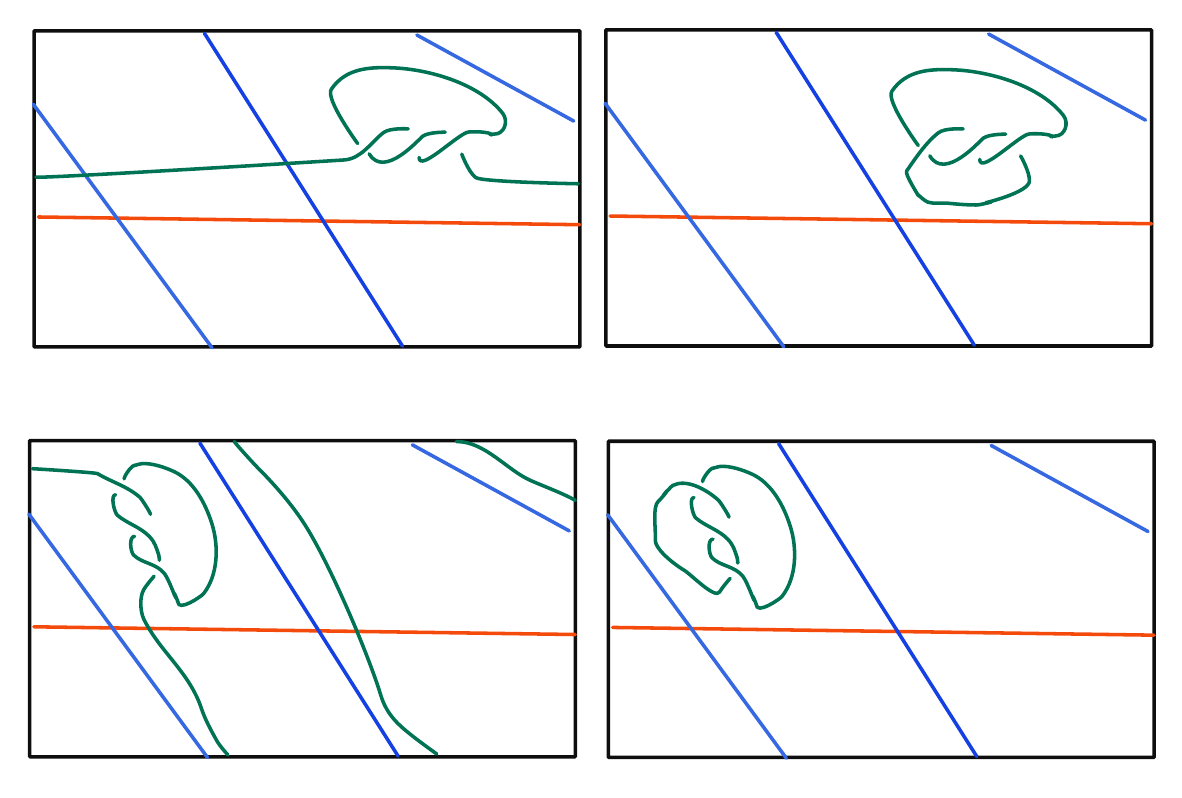}
		\caption{\textbf{Disk slide moves} In the first row, we show an example of an $\alpha$-disk slide; in the second row, we show an example of a $\beta$-disk slide.}
		\label{fig:disk slide}
	\end{figure}
	
	\begin{thm}\label{singular Reidemeister theorem}
		Consider $L$ and $L'$ to be singular links in $L(p,q)$ with the same underlying graph, the following are equivalent:
		\begin{enumerate}
			\item $L$ and $L'$ are equivalent;
			\item $L$ and $L'$ differ by a finite sequence of disk moves;
			\item regular projections of $L$ and $L'$ to the Heegaard torus differ by a finite sequence of singular Reidemeister moves and disk slides. 
		\end{enumerate}
		
	\end{thm}
	\begin{proof}
		(3) implies (1) is clear. 
		
		We now show (1) implies (2). Take $h$ to be the piecewise linear homeomorphism in the definition of equivalence. Note that any automorphism of $L(p,q)$ that is isotopic to the identity fixes the Heegaard torus $T$ and two handlebodies up to isotopy. See \cite{MR710104} for detailed knowledge on the mapping class group of lens spaces. So without loss of generality, we assume that $h$ sends $T^2$, $ V_\beta$, $V_\alpha$ to themselves, respectively. Observe that $L$ intersects $V_\alpha$ in finitely many arcs, so by pushing $T$ a little bit down into $V_\alpha$ using its collar neighborhood in $V_\alpha$, we may assume that $L$ lies entirely in $\mathrm{int}(V_\beta)$. Since $h(L)=L'$ and $h(V_\beta)=V_\beta$, $L'$ also lies in $\mathrm{int}(V_\beta)$.
		
		Let $V_\beta$ be parametrized as $S^1\times D^2$ in the usual way. Let $C=S^1\times\{0\}$ be the core curve, $D_\epsilon$ be the subdisk of $D^2$ with radius $\epsilon$, and $V_{\epsilon}=S^1\times D_{\epsilon}\subset V_\beta$. Further denote by $S=\partial V_\beta\cong T^2$.
		
		We can further isotope $L$, $L'$, and $h$ so that \[L\cap C=\emptyset=L'\cap C,\] and $h$ fixes some $V_{\epsilon'}$ disjoint from $L$ and $L'$ for $\epsilon'>0$ sufficiently small.
		
		Before starting the proof, we make an important claim: If $V_{\epsilon'}$ is a solid torus disjoint from $L$, then there exists a singular link $L^*$ in $V_{\epsilon'}$ which can be connected to $L$ within a finite sequence of disk moves.
		
		Choose $\epsilon_0$ and $\epsilon$ with \[0<\epsilon_0<\epsilon'<\epsilon<1\] such that \[V_{\epsilon_0}\subset \mathrm{int}(V_{\epsilon'})\subset L\cup V_{\epsilon'} \subset int (V_{\epsilon}) .\] Then, we have an obvious homeomorphism \[d:V_{\epsilon}-\mathrm{int}(V_{\epsilon_0})\to S\times[1,3],\] with $d(V_{\epsilon'}-\mathrm{int}(V_{\epsilon_0}))=S\times [1,2]$.
		
		Denote by $p_S:S\times[1,3] \to S$ and $p_I:S\times[1,3] \to [1,3]$ the projections. Further introduce the compositions $e_S=p_S\circ d:V_{\epsilon}-\mathrm{int}(V_{\epsilon_0}) \to S$, $e_I=p_I\circ d:V_{\epsilon}-\mathrm{int}(V_{\epsilon_0}) \to [1,3]$.
		
		By further composing with an automorphism $f$ of $S\times[1,3]$ satisfying $f(S\times[1,2])=S\times[1,2]$ if necessary, we may assume $e_S|_L$ is an immersion. Observe that $e_I(L)\subset (2,3)$, so we have a well-defined map $F:L\times [0,1]\to S\times [1,3]$ defined by \[F(x,t)=(e_S(x),e_I(x)-t).\]
		This $F$ is not piecewise linear, but it sends piecewise linear subspaces to piecewise linear subspaces. In particular, the restriction to each $L\times\{t\}$ is piecewise linear.
		
		For each $t\in [0,1]$, there is a closed neighborhood $N(t)\subset [0,1]$ such that $F$ is an embedding when restrict to $L\times N(t)$, thus $F(L\times N(t))$ is a piecewise linear submanifold of $S\times[1,3]$. Using this and the compactness of $[0,1]$, we get a finite sequence of ``time levels'' $0=t_0<t_1\ldots <t_n=1$ such that each $F(L\times\{ t_i\})$ can be connected to $F(L\times \{t_{i+1}\})$ using finitely many disk moves. By taking $L^*=F(L\times \{1\})$, we have shown our claim.
		
		Now we go back and pick up our homeomorphism $h$. We have assumed that $h$ fix some $V_{\epsilon'}$  disjoint from $L$ and $L'$,and we can take $L^*\subset V_{\epsilon'}$ obtained from the claim above.  Now $L^*$ can be related to $L$ by a finite sequence of disk moves. Applying $h$, we see that $h(L^*)=L^*$  can be connected to  $h(L)=L'$ by a finite sequence of disk moves. Consequently, $L$ and $L'$ can be connected by a finite sequence of disk moves.

		Finally, we show that (2) implies (3). It suffices to consider a pair of links $L$ and $L'$ that can be related by a single disk move. By subdividing the disks into triangles and refining the triangulation of $T^2$ and $L$, we only need to consider small triangles for which exactly one of the following is true. \begin{itemize}
			\item The core curve of $V_\alpha$ or $V_\beta$ passes through the disk once.
			\item It contains at most one crossing in its interior.
			\item There is at most one graphical vertex (an original vertex in the transverse graph definition or equivalently a thick edge in the first definition) on its boundary or in its interior.
		\end{itemize} 
		The first case corresponds to a $\beta$ (resp. an $\alpha$)-disk slide if the core curve belongs to $V_\beta$ (resp. $V_\alpha$).
		
		The second case has been dealt with in \cite{MR1414898}. We only need to deal with the third one:
		\begin{itemize}
			\item See \text{Figure \ref{fig:disk-move1}} for cases in which there is exactly one graphical vertex in the interior of the triangle. Our convention is that disk moves happen between the orange edges and the blue edge that together form the ``disk''-triangle. These can be realized by a single Reidemeister move IV or a composition of one or two Reidemeister move II together with a Reidemeister move IV. See \text{Figure \ref{fig:disk-move-detailed}} for an illustration.
			
			\begin{figure}
				\centering
				\includegraphics[width=0.9\linewidth]{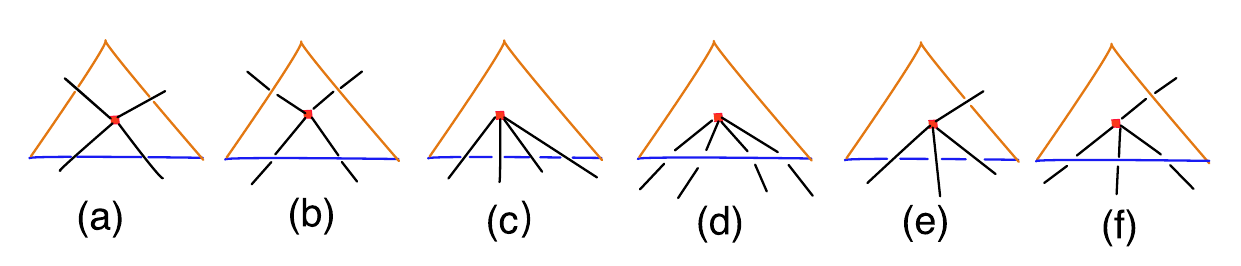}
				\caption{\textbf{Disk moves with a vertex in the disk interior.} In (a)-(f), we list all possible cases of a triangle of disk moves containing exactly one vertex in its interior. In each case, the disk move transforms between the blue edge and the orange broken line. (a) and (b) can be realized by a Reidemeister move IV directly. In \text{Figure \ref{fig:disk-move-detailed}}, we show how to realize (c) and (e) by Reidemeister moves. (d) and (f) can be realized similarly.}
				\label{fig:disk-move1}
			\end{figure}
			\begin{figure}
				\centering
				\includegraphics[width=0.8\linewidth]{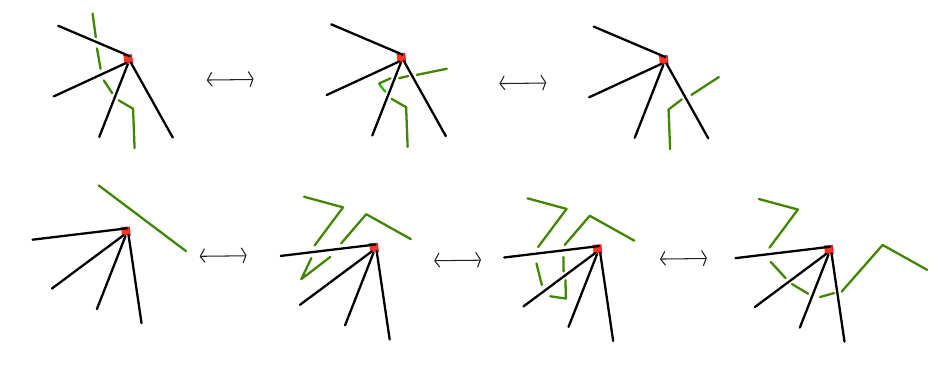}
				\caption{\textbf{Realize disk moves by Reidemeister moves.} In the first row, we show how to realize \text{Figure \ref{fig:disk-move1}}(e) by Reidemeister moves II and IV. In the second row, we show how to realize \text{Figure \ref{fig:disk-move1}}(c) by Reidemeister moves II and IV. We omit orientations in this list, since orientations are irrelevant for these moves.}
				\label{fig:disk-move-detailed}
			\end{figure}

			\item For the case of a graphical vertex appearing as a vertex of the triangle, see \text{Figure \ref{fig:disk-move2}}. Each of these is just a singular Reidemeister move V.
		\end{itemize}
		
		\begin{figure}
			\centering
			\includegraphics[width=0.75\linewidth]{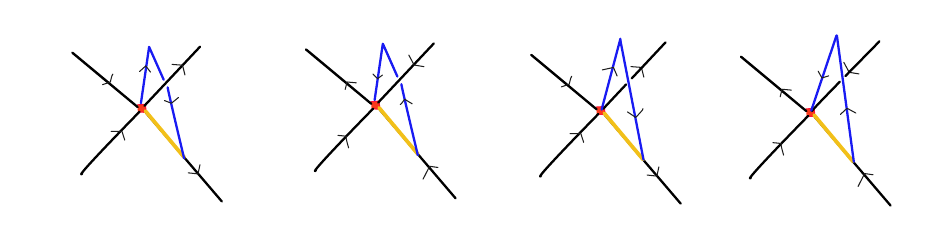}
			\caption{\textbf{Disk moves with a graphical vertex as a vertex of the triangle.} We list here all possible cases for disk moves to have a graphical vertex as a vertex of the triangle (disk). Each of these can be realized by a Reidemeister move V.}
			\label{fig:disk-move2}
		\end{figure}

		After dealing with all these cases, we know (2) implies (3), so the three items are equivalent.
		
	\end{proof}
	\begin{rmk}
		\begin{enumerate}
			\item The $\alpha$(resp. $\beta$)-disk slides can be realized as changing the way we connecting $O$ to $X$ (resp. $X$ to $O$) in the horizontal (resp. vertical) annulus. Therefore, these moves are ``invisible'' on a grid diagram if only consider the information we listed in Definition \ref{grid diagram}, so we will only consider Reidemeister moves in next subsection.
			\item The proof of Theorem \ref{singular Reidemeister theorem} follows the trend of the proof of classical Reidemeister Theorem in \cite[Appendix]{Kawauchi1996ASO}, in which they considered $S^3=B^3\cup_{S^2} B^3$ which is replaced by $L(p,q)=V_{\alpha}\cup_{T^2} V_{\beta}$ in our case. The extra disk slides account for the handles since now the genus is no longer zero. One can see from the proof that it holds for knots in more general three manifolds which has a standard Heegaard splitting and a simple enough mapping class group.  
		\end{enumerate}
	\end{rmk}

	\subsection{Moves on grid diagrams}\label{Moves on grid diagram}
	
	Now we are ready to introduce grid moves and show that they are enough for realizing all Reidemeister moves defined above. Once this is done, we can introduce invariants of singular links by considering invariants of grid diagrams that are kept unchanged under these moves. 
	
	For classical grid moves, there is a detailed description in \cite[Chapter 3]{MR3381987}. Here we almost follow \cite{MR3677933} except that our background grid diagrams are no longer formed by squares, and we are using a different labeling convention: our $XX$ corresponds to their nonstandard $O$, our $X$ corresponds to their regular $O$, while our $O$ corresponds to their $X$. 
	
	To make the definitions below clear, we recall the convention stated in \text{Definition \ref{grid diagram}}: a column (resp. row) is an annulus cut out by adjacent $\beta$ (resp. $\alpha$) curves. We introduce three kinds of \emph{grid moves}, naming them after the classical case:
	\begin{itemize}
		\item Cyclic permutation: Permuting the columns or rows cyclically. This is the same as the classical case. 
		
		\item Stabilization and destabilization: Assume $g$ is a grid diagram of size $n$. We stabilize $g$ to obtain a diagram of size $n+1$. Choose any marked cell of $g$. Erase the mark in this cell as well as marks in the same row or column. Split the empty row and column into two rows and two columns. If the original mark is $X$ or $O$, then there are exactly four ways to fill in the new rows and columns to get a grid diagram (See \cite{MR3381987}, in which they described them in terms of directions: South-East, South-West, North-East, or North-West). If the original mark is $XX$, then we must fill in the new rows and columns with one $XX$, one $X$ and six $O$s. There are sixteen ways to do this to fit the new diagram into \text{Definition \ref{grid diagram}} without changing the link that it represents. Destabilization is defined as the inverse of a stabilization.
		
		\item Commutation: we describe using columns; the same applies to rows. A pair of adjacent columns can be exchanged if the following conditions are satisfied: there are slanted line segments $LS_1$ and $LS_2$ on the torus such that 
		\begin{enumerate}
			\item $LS_1\cup LS_2$ contains all the $O$, $X$, and $XX$ base points in these two columns;
			\item the projection of $LS_1\cup LS_2$ to a single curve $\beta_i$ is $\beta_i$;
			\item the projection of $\partial LS_1\cup \partial LS_2$ consists of exactly two points.
		\end{enumerate}
		
	\end{itemize}
	
	In \text{Figure \ref{fig:grid-moves-new}}, we show examples of grid moves. 
	
	One can see from these examples that grid moves are just reformulations of the Reidemeister moves we considered in the previous section; in particular, they keep the equivalence class of the link unchanged.
	
	\begin{figure}
		\centering
		\includegraphics[width=0.95\linewidth]{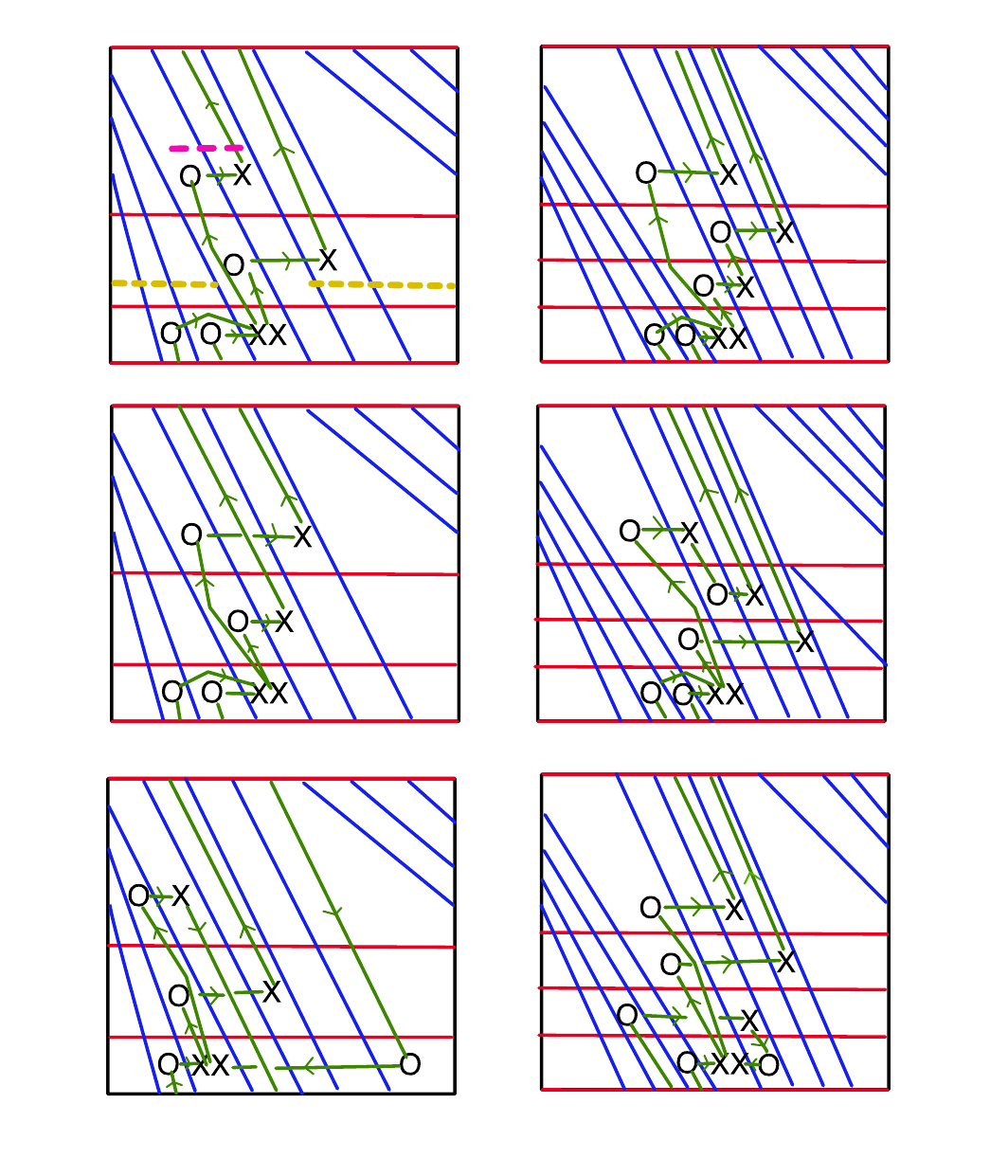}
		\caption{\textbf{Example of grid moves.} The top left picture is an alternative grid diagram for the singular knot we have seen in \text{Figure \ref{fig:ex-of-grid-diagram}}. We shall apply various grid moves to this diagram. The pink and yellow segments are $LS_1$ and $LS_2$, respectively, showing that these two columns are available to be exchanged. In the middle left picture, we show the grid diagram after commutation. The lower left picture is the result of performing a cyclic permutation; the top right picture results from a stabilization $O:NW$; the middle right picture results from a stabilization $X:NE$; the lower right picture results from a stabilization $XX:NW$.}
		\label{fig:grid-moves-new}
	\end{figure}
	
	\begin{rmk}
		There are extra moves introduced in \cite{MR3381987}, the so-called switch and cross commutation. In our definition, switches are contained in the set of commutations. As mentioned there, cross commutation leads to crossing changes on planar knot diagrams, so it does not preserve the isotopy type in general.
	\end{rmk}
	
	Before proving the completeness of these moves, we introduce the concepts of L-formation and preferred diagram mimicking \cite{MR3677933}.
	
	\begin{defi}
		For an $XX$ mark, we call the set of $O$'s sharing the same row or column its \emph{flock}. We say the flock of an $XX$ is \emph{in L-formation} if it is immediately to the right or above that $XX$. A preferred grid diagram is a grid diagram in which all $XX$'s have their flocks in L-formation. 
		
	\end{defi}
	An example of such a diagram is shown in \text{Figure \ref{fig:l-formation}}. This is again a grid diagram of the knot we have considered in \text{Figure \ref{fig:ex-of-grid-diagram}}.
	\begin{figure}
		\centering
		\includegraphics[width=0.4\linewidth]{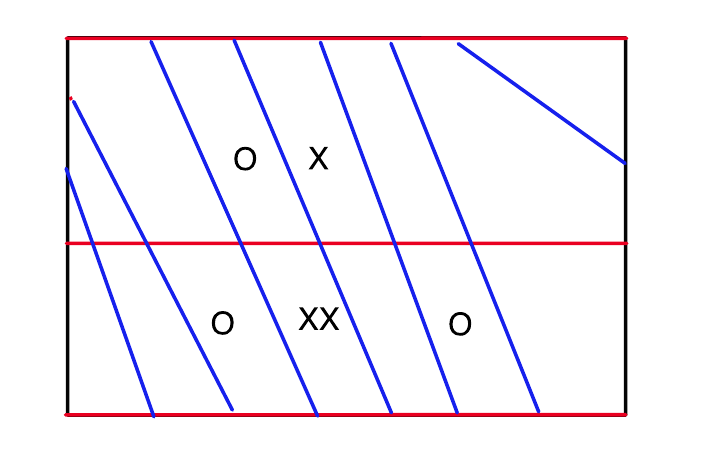}
		\caption{\textbf{Example of a grid diagram in L-formation.}}
		\label{fig:l-formation}
	\end{figure}
	
	The proof of the following in \cite{MR3677933} directly extends to our case:
	\begin{lem}
		Every grid diagram of a singular link can be made into a preferred grid diagram using a finite sequence of grid moves.
	\end{lem}
	The proof of this lemma from \cite{MR3677933} has been done in two steps:\begin{itemize}
		\item By using stabilization, we separate the flocks so that no $O$ is shared by a singular row and a singular column. 
		\item By doing stabilization and commutation, we can put the flock of each $XX$ into L-formation. With the help of the first step, this can be done without violating the L-formation of other $XX$'s. Thus, we can achieve our goal in finite steps.
	\end{itemize}
	
	\begin{thm}\label{grid moves is complete}
		If two grid diagrams $g$ and $g'$ represent the same singular link $L$, then they can be connected by a finite sequence of grid moves.
	\end{thm}
	The main idea of this proof is to verify that any planar isotopy and all Reidemeister moves can be realized by a finite sequence of grid moves. With the help of the previous lemma, we only need to consider the case that both $g$ and $g'$ are preferred. We find that this verification has been done by previous works:
	\begin{itemize}
		\item For regular Reidemeister moves and isotopy of diagrams that do not involve any singular point, the proof has been done by \cite{CROMWELL199537} and \cite{article}. For a detailed version, one can refer to \cite[Appendix B]{MR3381987}.
		\item For the singular Reidemeister moves and isotopy that involve vertices, a detailed proof is given in \cite{MR3677933}. The only difference is that our diagram is not a square but a parallelogram. Actually, we only need part of their proof: Since all of our vertices are 4-valent, only two kinds of isotopies need to be taken care of: a vertex passing through an arc and two vertices passing each other. See \cite[Figure 26,27,28 and 30,31]{MR3677933}.
	\end{itemize}

	Before ending this section, we quote a lemma from \cite{MR3381987} that implies that it suffices to consider (de)stabilization and commutation invariance for the theory that we are going to define.
	\begin{lem}(\cite[Lemma 3.2.4]{MR3381987})
		A cyclic permutation can be realized by a finite sequence of commutation, stabilization, and destabilization.
	\end{lem}
	Although they proved this for grid diagrams of knots and links in $S^3$, their proof applies to our case without change.

	\section{Grid homology for singular links in lens spaces}\label{Grid homology for singular links in lens spaces}
	In this section, we always consider a grid diagram as follows unless otherwise stated. $g$ is a grid diagram of size $n+k$ representing a singular link $L$ in $L(p,q)$ with $m$ regular components and $n$ thick edges. That is, we have a diagram ruled by $n+k$ $\beta$ and $\alpha$ circles with $n$ $XX$ marks, $k$ $X$ marks, and $2n+k$ $O$ marks. By renumbering the $O$-marks, we assume that $O_1,\ldots, O_{2n+m}$ lie on distinct thin edges or $S^1$ components. Note that $n$ and $m$ are intrinsic to the underlying graph while $k$ depends on the embedding and the specified grid diagram. 
	
	For convenience, we introduce some new notations:\begin{itemize}
		\item $\OV$: the set of all $O$ marks;
		\item $\XXV\XXV$: the set of all $XX$ marks;
		\item $\XXV_r $: the set of all $X$ marks;
		\item $\XXV=\XXV\XXV\cup \XXV_r$.
	\end{itemize} 
	
	\subsection{Construction of chain complexes}\label{Construction of chain complex}
	
	Form the symmetric product $Sym^{n+k}(T^2)$, which contains tori $\TA=\alpha_1\times \ldots \times \alpha_{n+k}$ and $\TB=\beta_1\times \ldots \times \beta_{n+k}$. The set of intersection points $\boldsymbol{S}=\TA\cap \TB$ will be the generating set of our chain complexes. For $\xv\in \boldsymbol{S}$, we shall regard it as a point in $Sym^{n+k}(T^2)$, as well as a finite set of points in $g$. In our grid diagram, each pair of $\beta$ and $\alpha$ curves has $p$ intersection points between them. Consequently, we have $(n+k)!\cdot p^{n+k}$ generators in total.

	Note that the grid diagram endows the torus with a natural cell-complex structure: It has $p\cdot(n+k)^2$ 0-cells, $2p\cdot(n+k)^2$ 1-cells and $p\cdot(n+k)^2$ 2-cells. 
	
	A \emph{domain} from $\xv$ to $\yv$ is a 2-chain $D$ with $\partial D \cap \AV$ a path from $\xv$ to $\yv$ and $\partial D \cap \BV$ a path from $\yv$ to $\xv$.  We will use $\pi_2(\xv,\yv)$ to denote the set of domains from $\xv$ to $\yv$. When a domain from $\xv$ to $\yv$ is an embedded rectangle, it will be called a \emph{rectangle } from $\xv$ to $\yv$. A rectangle $r$ from $\xv$ to $\yv$ is \emph{empty} if $int (r)\cap \xv=\emptyset=\mathrm{int}(r)\cap \yv$. Denote by $\mathrm{Rect}(\xv,\yv)$ the set of rectangles from $\xv$ to $\yv$ and by $\mathrm{Rect}^\circ(\xv,\yv)$ the set of empty rectangles from $\xv$ to $\yv$. For a point $p$ living in the interior of some 2-cell, denote by $n_{p}(D)$ the multiplicity of $D$ at $p$. We abuse the notation $n_{P}(D)$ for the sum of multiplicities of all points in $P$, when $P$ is a finite set of points, each living in the interior of some 2-cell.
	
	\begin{defi}\label{definition of chain complex}
		We define three versions of chain groups and differentials: Here $\FF=\ZZ/2\ZZ$.
		\begin{itemize}
			\item Unblocked version: the chain group $CFK^-(g)$ is generated by $\boldsymbol{S}$ over $\FF[U_1,\ldots,U_{2n+k}]$, the differential counts empty rectangles:
			\[\partial^-\xv=\sum_{\yv\in \boldsymbol{S}} \sum_{\substack{r\in \mathrm{Rect}^\circ(\xv,\yv)\\ n_{\XXV}(r)=0}} \prod_{i=1}^{2n+k} U_{i}^{n_{O_i}(r)} \yv;\]
			
			\item Simply blocked version: $\widehat{CFK}(g)=CFK^-(g)/(U_1=\ldots U_{m+2n}=0)$;

			\item Fully blocked version: $\widetilde{CFK}(g)=CFK^-(g)/(U_1=\ldots U_{k+2n}=0)$. 
			
			Alternatively, take $\FF$ as the base ring, so that $\widetilde{CFK}(g)$ is a finitely generated vector space over $\FF$ with $\boldsymbol{S}$ as a basis. The differential counts empty rectangles without any marked points:
			\[\widetilde{\partial} \xv=\sum_{\yv\in \boldsymbol{S}} \sum_{\substack{r\in \mathrm{Rect}^\circ(\xv,\yv)\\n_{\XXV}(r)=0, n_{\OV}(r)=0}} \yv.\]
			
		\end{itemize}
	\end{defi}

	\begin{rmk}
		Although we use the notations ``CFK'' and ``HFK'' for our chain complexes and homology groups instead of ``GC'' and ``GH'', we only take grid diagrams into consideration in our paper.
	\end{rmk}
	Using the standard complex structure on $T^2$ and with the help of Lipshitz's formula (\cite{MR2240908}), one sees that one is exactly the Maslov index of an empty rectangle, and empty rectangles are the only possible ``disks'' in a grid diagram with such Maslov index. So our grid homology can be regarded as a combinatorial generalization of Heegaard Floer theory to singular links in lens spaces.
	
	To verify they are chain complexes, one can resort to analytic tools and consider broken flow lines. However, we can actually do this combinatorially. We illustrate this using the minus version. Observe that for any $\xv\in \boldsymbol{S}$, 
	\[\partial^- \circ \partial^-(\xv)= \sum_{\zv\in \boldsymbol{S}} \sum_{D \in\pi_2(\xv,\zv)} N(D) \prod_{i=1}^{2n+k} U_i^{n_{O_i}(D)}\zv,\]
	where $N(D)$ is the number of ways one can decompose $D$ as a union of empty rectangles that intersect $\XXV$ emptily. That is how many ways we can express $D$ as $D=r_1\ast r_2$ for some $\yv\in\boldsymbol{S}$, $r_1 \in \mathrm{Rect}^\circ (\xv,\yv)$, $r_2\in \mathrm{Rect}^\circ (\yv,\zv)$, and $r_1\cap \XXV= r_2\cap \XXV=\emptyset$. Here we use $\ast$ to denote juxtaposition of domains on a grid diagram $g$. Note that if $\prod_{i=1}^{2n+k} U_{i}^{n_i}\zv$ appears in $(\partial^-)^2 \xv$, $\xv -\xv\cap \zv$ can only consist of 0, 2, 3, or 4 points. The case of 1 can be ruled out by a direct geometric argument. Following \cite{Celoria_2023}, we denote the cardinality of this set by $M$. The case $M=0$ is easy: when $\xv=\zv$, $D$ can only be a linear combination of thin annuli, which is not allowed since it always contains an $X$ or $XX$ mark. Here and in the following, an annulus or a rectangle is called thin if it is of width or height one. So terms like $\prod_{i=1}^{2n+k} U_{i}^{n_i}\xv$ never appear in $(\partial^-)^{2} \xv$. For $M=3$ or $4$, the argument in \cite[Lemma 4.6.7]{MR3381987} shows that $N(D)=2$ whenever it is not zero. For $M=2$, it only appears when the lens space is not $S^3$, as now each pair of $\alpha$ and $\beta$ intersect more than once. A detailed analysis has been done in \cite[Proposition 2.11]{Celoria_2023}, which shows that $N(D)$ is again zero or two in this case. Thus we can conclude that $(\partial^-)^2=0$. 
	
	\begin{rmk}\label{rmk on N(D)}
		The proof of \cite[Proposition 2.11]{Celoria_2023} (or \cite[Proposition 3.6]{tripp2021gridhomologylensspace}), together with that of \cite[Lemma 4.6.7]{MR3381987} give a complete analysis of $N(D)$ when $D$ can be written as $r_1\ast r_2$, a juxtaposition of empty rectangles. An important observation is that the number of decompositions has nothing to do with what marks the domain contains, so their result can be used to analyze the composition of any pair of rectangle counting maps. This will play an important role in the following sections. 
	\end{rmk}

	\subsection{Gradings}\label{gradings}
	
	We will define two kinds of grading combinatorially on the chain complex. Instead of using flowlines and $\mathrm{Spin^c}$ structures, this will be done by fitting the grid diagram into a suitable coordinate system and defining a comparison function between tuples of points in the diagram. The idea of this construction comes from \cite{MR2429242}. Modification has been done due to ``singularities'' according to \cite{MR2529302}.
	
	We embed our grid diagram into $\RR^2$ by declaring the $\alpha$ circles $\{\alpha_i\}_{i=0}^{n+k-1}$ to be line $y=\frac{i}{n+k}, 0\le i\le n+k-1$ and the $\beta$ circles $\{\beta_i\}_{i=0}^{n+k-1}$ to be $y=-\frac{p}{q}(x-\frac{i}{p(n+k)})$. We shall take a preferred fundamental domain $0\le y< 1$, $-\frac{q}{p}y\le x<-\frac{q}{p}+1 $. As above, we shall denote our grid diagram by $g$. 
	
	We first introduce some auxiliary functions: consider a set-valued function \[W:\{\text{finite sets of points in $g$}\}\to \{\text{finite sets of pairs } (a,b) \text{ with } a\in [0,p(n+k)),b\in[0,n+k)\},\] which assigns to each set of points in $g$ the tuple of its $\RR^2$ coordinates written with respect to the basis \[\{\overrightarrow{v_1}=(\frac{1}{p(n+k)},0), \overrightarrow{v_2}=(-\frac{q}{p},\frac{1}{n+k})\}.\]

	We always assume that $X$, $O$, or $XX$ marks live in the center of their cells, respectively. For $\xv\in \boldsymbol{S}$, $W(\xv)$ is a set of integers while $W(\OV)$, $W(\XXV_r)$, and $W(\XXV\XXV)$ are sets of half-integers.
	
	We also need another coordinate transformation function \[C_{p,q}:\{\text{finite sets of pairs } (a,b)\text{ with }a\in [0,p(n+k)),b\in[0,n+k)\}\]
	\[\to \{\text{finite sets of pairs } (a,b) \text{ with } a,b\in[0,(n+k)p)\},\]
	sending a N-tuple of coordinates \[\{(a_i,b_i)\}_{i=0}^{N-1}\] to a pN-tuple of coordinates\[\{((a_i+(n+k)ql) \text{ mod } p(n+k),(b_i+(n+k)l) \text{ mod }p(n+k))\}_{i=0,l=0}^{i=N-1,l=p-1}.\]
	
	Now, let $\widetilde{W}$ be the composition $C_{p,q} \circ W$. 
	
	Further, we introduce a function $\mathcal{I}$ (originally from \cite[Section 4.3]{MR3381987}). It has as input an ordered pair $(A,B)$, each of which is a finite set of coordinate pairs. The output $\mathcal{I}(A,B)$ is the cardinality of pairs $(a_1,a_2)\in A$ and $(b_1,b_2)\in B$ for which $a_i<b_i$, $i=1,2$. Also define $\mathcal{J}(A,B)=\frac{1}{2}(\mathcal{I}(A,B)+ \mathcal{I}(B,A))$, which is the symmetrization of $\mathcal{I}$.
	
	Now we define relative gradings between $\xv$ and $\yv$ when they can be connected by a sequence of not necessarily empty rectangles. (For readers familiar with original Heegaard Floer theory, this is equivalent to that $\xv$ and $\yv$ come from same $\mathrm{Spin^c}$ structure.)
	
	The \emph{relative Maslov grading} is given by
	\begin{align}
		M(\xv,\yv) & = M_{\OV}(\xv,\yv)\\
		& =\frac{1}{p}(\mathcal{J}(\widetilde{W}(\xv),\widetilde{W}(\xv)-\mathcal{J}(\widetilde{W}(\yv),\widetilde{W}(\yv))-2\mathcal{J}(\widetilde{W}(\OV-\XXV\XXV),\widetilde{W}(\xv-\yv))).
	\end{align}
	
	Similarly, we define 
	\begin{align}
		M_{\XXV}(\xv,\yv) 
		& =\frac{1}{p}(\mathcal{J}(\widetilde{W}(\xv),\widetilde{W}(\xv)-\mathcal{J}(\widetilde{W}(\yv),\widetilde{W}(\yv))-2\mathcal{J}(\widetilde{W}(\XXV_r+\XXV\XXV),\widetilde{W}(\xv-\yv))).
	\end{align}

	Then the \emph{relative Alexander grading} is defined as \[A(\xv,\yv)=\frac{1}{2}(M_{\OV}(\xv,\yv)-M_{\XXV}(\xv,\yv)).\]
	
	We extend these gradings to the unblocked theory by letting \[A(U_i\xv,\xv)=-1,\]
	\[M(U_i\xv,\xv)=-2\]
	for any $\xv\in \boldsymbol{S}$ and $1\le i\le 2n+k$.

	We define a lift of these gradings to $\QQ$ following the convention in \cite{MR2429242}, in the sense that when our link has no singular point, our definition goes back to theirs. This is needed for us to compare gradings coming from different grid diagrams.
	
	\begin{align}\label{canonical lift of M}
		M(\xv)& = M_{\OV}(\xv)\\
		& =\frac{1}{p}(\mathcal{J}(\widetilde{W}(\OV-\XXV\XXV)-\widetilde{W}(\xv),\widetilde{W}(\OV-\XXV\XXV)-\widetilde{W}(\xv)))\\
		&+d(p,q,q-1)+\frac{p-1}{p},
	\end{align}
	\begin{align}
		M_{\XXV}(\xv)& =\frac{1}{p}(\mathcal{J}(\widetilde{W}(\XXV_r+\XXV\XXV)-\widetilde{W}(\xv),\widetilde{W}(\XXV_r+\XXV\XXV)-\widetilde{W}(\xv)))\\
		&+d(p,q,q-1)+\frac{p-1}{p},
	\end{align}
	
	\[A(\xv)=\frac{1}{2}(M_{\OV}(\xv)-M_{\XXV}(\xv)-(n+k-1)),\]
	in which $d(p,q,i)$ is the correction term of $L(p,q)$ defined as in \cite{MR1957829}, using a canonical identification of $\mathrm{Spin^c}(L(p,q))$ with $\ZZ/p\ZZ$.
	
	\begin{prop}
		In each of the three chain complexes $CFK^-(g)$, $\widehat{CFK}(g)$ and $\widetilde{CFK}(g)$, the differential preserves $A$ and lowers $M$ by 1.
	\end{prop}
	
	This can be shown exactly in the same way as in \cite[Proposition 2.11]{Celoria_2023}. This is an argument of lifting the diagram to one for a singular link in $S^3$, which is what the function $\widetilde{W}$ does. We will give a detailed description of how to use lift to compute gradings in \text{Subsection \ref{Proof of invariance}}.

	\subsection{Definition of homology theories}\label{Definition of homology theories}
	We define three versions of homology for singular links using the three chain complexes defined in the previous two sections. We shall prove later that the first two are invariants for singular links in $L(p,q)$ while the third one is closely related to the second one in the sense we shall make precise in \text{Proposition \ref{tilde vs hat}}. Here we fix a singular link with $m$ $S^1$ components and $n$ thick edges as above and choose a grid diagram $g$ for $L$ of size $n+k$.
	
	\begin{defi}
		\begin{itemize}
			\item The unblocked grid homology is defined as \[HFK^-(L)=H_*((CFK^-(g),\partial^-)),\] viewed as an $\FF[U_1,\ldots,U_{m+2n}]$ module.
			
			\item The simply blocked grid homology is defined as \[\widehat{HFK}(L)=H_{*}((\widehat{CFK}(g),\widehat{\partial})),\] viewed as an $\FF$ vector space, where $\widehat{\partial}$ is the differential inherited from $\partial^-$.
			
			\item The fully blocked grid homology is defined as \[\widetilde{HFK}(L)=H_{*}((\widetilde{CFK}(g),\widetilde{\partial})),\]viewed as an $\FF$ vector space. 
		\end{itemize}
		
		Each homology group is relatively bigraded by $M$ and $A$ defined above. 
	\end{defi}
	
	To show that $HFK^-(L)$ is well-defined as a $\FF[U_1,\ldots,U_{m+2n}]$ module and $\widehat{HFK}(L)$ is well-defined as a $\FF$ vector space, we need the following propositions.
	
	\begin{prop}\label{homotopic U_i action}
		Consider any grid diagram $g$,
		\begin{enumerate}
			\item If $O_i$ and $O_j$ belong to the same thin edge or the same $S^1$ component of the underlying graph, then the multiplication by $U_i$ is chain homotopic to the multiplication by $U_j$.
			\item For any singular point $XX$, let $O_i$, $O_j$ be $O$-marks sharing the same row, and $O_s$, $O_l$ be marks sharing the same column, then the multiplication by $U_iU_j$ is chain homotopic to the multiplication by $U_sU_l$.
		\end{enumerate}
	\end{prop}
	\begin{proof}
		These statements can be proved in the same vein as \cite[ Proposition 4.21]{MR3677933}. One can also refer to   \cite[Lemma 4.6.9]{MR3381987}. For readers' convenience, we also provide a self-contained proof. The detail is as follows:
		
		For (1), we may suppose that $O_i$ and $O_j$ are connected to the same $X$ marked point, say $X_l$. Define a map $H_{X_l}:CFK^-(g) \to CFK^-(g)$ by \[H_{X_l}(\xv)=\sum_{\yv\in \boldsymbol{S}} \sum_{\substack{r\in \mathrm{Rect}^\circ(\xv,\yv)\\Int(r)\cap \XXV=X_l}} \prod_{t=1}^{2n+k} U_{t}^{n_{O_t}(r)} \yv.\] 
		
		Now we want to analyze how $\partial^-\circ H_{X_l}+H_{X_l} \circ\partial^-$ acts on generators. Note that this is a sum of compositions of rectangle counting maps, so \text{Remark \ref{rmk on N(D)}} applies. For each pair of $\xv\ne\zv$ differing by at most four points and each positive domain $D$ connecting them that intersects $\XXV$ in $X_l$ once, there can only be zero or two ways to decompose $D$ into two rectangles. In case of two, say $D=r_1\ast r_2= r'_1\ast r'_2$, then we must have one of $r_1$, $r_2$ containing $X_l$, the other does not, and same hold for $r'_1$ and $r'_2$. Thus, the two decompositions both contribute to $\partial^-\circ H_{X_l}+H_{X_l} \circ\partial^-$ and get canceled mod 2. When $\xv=\zv$, only two thin annuli through $X_l$ contribute to the map, which gives rise to the term $(U_i-U_j) \xv$.
		Thus \[\partial^-\circ H_{X_l}+H_{X_l} \circ\partial^-=U_i-U_j,\] showing the desired homotopy relationship.
		
		For (2), assume we are considering $XX_1$. We only need to modify the definition of $H$ into \[H_{XX_1}(\xv)=\sum_{\yv\in \boldsymbol{S}} \sum_{\substack{r\in \mathrm{Rect}^\circ(\xv,\yv)\\Int(r)\cap \XXV=XX_1}} \prod_{t=1}^{2n+k} U_{t}^{n_{O_t}(r)} \yv.\]
		
		Again, all the other terms in $\partial^-\circ H_{XX_1}+H_{XX_1} \circ\partial^-$ cancel in pairs; only the two thin annuli through $XX_1$ contribute multiplication by $U_iU_j-U_sU_l$. That is \[\partial^-\circ H_{XX_1}+H_{XX_1} \circ\partial^-=U_iU_j-U_sU_l,\]showing the desired homotopy relationship.
		
	\end{proof}

	\begin{prop}\label{tilde vs hat}
		For $g$ as described above, we have an isomorphism of relatively graded vector spaces \[\widetilde{HFK}(L)=\widehat{HFK}(L)\otimes W^{\otimes k-m},\] in which $W$ is a 2-dimensional bigraded vector space  spanned by generators in bigrading $(0,0)$ and $(-1,-1)$.
		
	\end{prop}
	
	\begin{proof}
		A basic but important observation is that when multiplication by $U_i$ and $U_j$ are homotopic maps, then when $U_i$ is set to be zero, $U_j$ induces a zero map on the homology of the new chain complex.
		
		This observation, together with the exact sequence of chain complexes (we abbreviate the chain group as $C$), \[0\rightarrow \frac{C}{U_1=\ldots =U_{2n+m}=0} \xrightarrow{U_{2n+m+1}} \frac{C}{U_1=\ldots =U_{2n+m}=0} \rightarrow \frac{C}{U_1=\ldots =U_{2n+m+1}=0} \rightarrow 0\] leads to a short exact sequence between different homology groups: 
		\[0\rightarrow H_*(\frac{C}{U_1=\ldots =U_{2n+m}=0}) \rightarrow H_*(\frac{C}{U_1=\ldots =U_{2n+m}=U_{2n+m+1}=0}) \rightarrow H_*(\frac{C}{U_1=\ldots =U_{2n+m}=0}) \rightarrow 0,\]
		in which the second arrow is homogeneous and the third one is of bigrading $(1,1)$.
		
		Then the proof is finished by repeating the process for each $O$ with subscript greater than $2n+m$.
		
	\end{proof}

	\subsection{Computing examples}\label{Computing examples}
	In this subsection, we compute $\widetilde{HFK}(K)$ for some singular knots in $L(2,1)$ and $L(3,1)$.

	In \text{Figure \ref{fig:example of singular knot in L(2,1)}}, we show two grid diagrams for a singular knot $K_1$ with one thick edge in $L(2,1)$, of grid number three and two, respectively. We shall denote the diagrams by $g_a$ and $g_b$ as we labeled in the figure. In $g_a$, we have forty-eight generators and fifty-six rectangles without any marked points in their interiors. In $g_b$, we have eight generators and three rectangles without any marked points in their interiors. Using some computer program, we compute that 
	\[\widetilde{HFK}(g_a)\cong \FF^8\] with four generators in each $\mathrm{Spin^c}$ structure and 
	\[\widetilde{HFK}(g_b)\cong \FF^4\]with two generators in each $\mathrm{Spin^c}$ structure.
	
	This demonstrates the computable aspect of our theory and also verifies the rank formula in \text{Proposition \ref{tilde vs hat}}.

	\begin{figure}
		\begin{overpic}[width=0.9\textwidth]{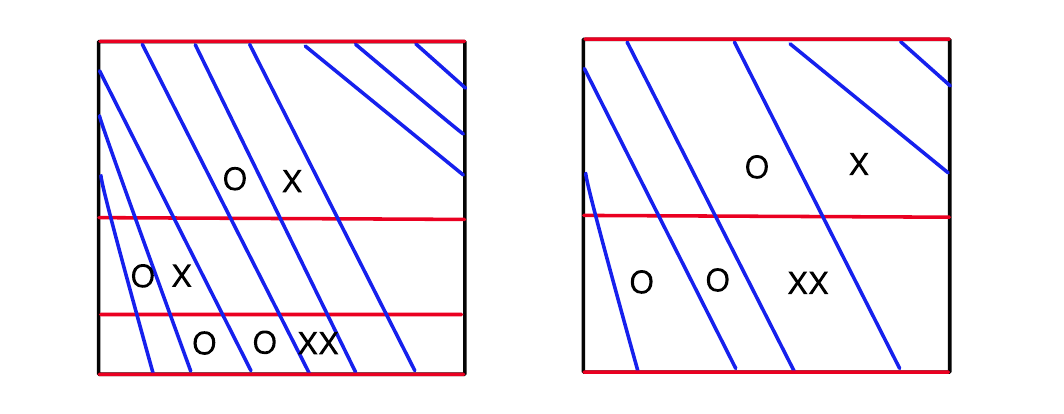}
			\put(25,-2){$g_a$}
			\put(72,-2){$g_b$}
			
		\end{overpic}
		
		\caption{\textbf{Two grid diagrams for a singular knot with one thick edge in $L(2,1)$} }
		\label{fig:example of singular knot in L(2,1)}
	\end{figure}
	
	In \text{Figure \ref{fig:example of singular knot in L(3,1)}}, we show a grid diagram $g$ for a singular knot $K_2$ with one thick edge in $L(3,1)$ of grid number two. In $g$, we have eighteen generators and ten rectangles without any marked points in their interiors. With the help of some computer program, we compute that 
	\[\widetilde{HFK}(g)\cong \FF^6,\] with two generators in each $\mathrm{Spin^c}$ structure.
	
	\begin{figure}
		\begin{overpic}[width=0.5\textwidth]{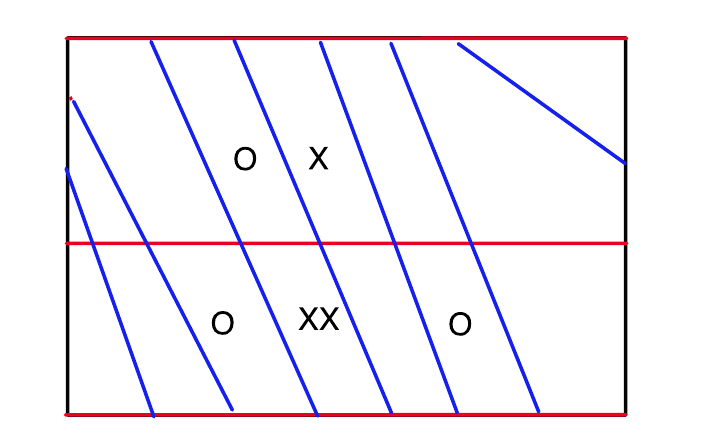}
			\put(45,-2){$g$}
		\end{overpic}
		\caption{\textbf{A grid diagram for a singular knot with a single thick edge in $L(3,1)$}}
		\label{fig:example of singular knot in L(3,1)}
	\end{figure}
	
	We can appeal to \text{Proposition \ref{tilde vs hat}} and conclude that \[\widehat{HFK}(L(2,1),K_1)=\FF\oplus \FF\] with one $\FF$ from each $\mathrm{Spin^c}$ structure and \[\widehat{HFK}(L(3,1),K_2)=\FF\oplus \FF\oplus \FF,\] with one $\FF$ from each $\mathrm{Spin^c}$ structure; 
	
	\subsection{Proof of invariance}\label{Proof of invariance}
	Using \text{Theorem \ref{grid moves is complete}} in \text{Subsection \ref{Moves on grid diagram}}, we only need to show that if $g$ and $g'$ are two grid diagrams that differ by a commutation or stabilization, then $CFK^\circ(g)$ and $CFK^\circ(g')$ are quasi-isomorphic bigraded chain complexes, for $\circ=\widehat{  }$ or $-$. That is
	\begin{thm}
		Fix a singular link $L$ in some lens space $L(p,q)$ as above. If $g_1$ and $g_2$ are grid diagrams for $(L(p,q),L)$, then we have quasi-isomorphisms \begin{itemize}
			\item $CFK^-(g_1)\simeq CFK^-(g_2)$;
			\item $\widehat{CFK}(g_1)\simeq \widehat{CFK}(g_2)$,
		\end{itemize}
		as relatively bigraded chain complexes of modules over $\FF[U_1,\ldots, U_{2n+m}]$ and over $\FF$, respectively.  
		In particular, we have well-defined homology theories for singular links in lens spaces: \begin{itemize}
			\item Unblocked grid homology: $HFK^-(L)$ as relatively bigraded modules over $\FF[U_1,\ldots, U_{2n+m}]$ and 
			\item Simply block grid homology: $\widehat{HFK}(L)$ as relatively bigraded modules over $\FF$.
		\end{itemize}
		
	\end{thm}

	Since $\widehat{CFK}$ appears as a quotient of $CFK^-$, the proof is easier and similar. We will focus on the second one.

	\subsubsection{Commutation invariance}\label{Commutation invariance}
	We will use an important observation in \cite{MR3677933}: the existence of line segments $LS_1$ and $LS_2$ leads to a well-defined ``combined'' grid diagram. That is, when $g$ and $g'$ differ by a single commutation move, we can combine  the information of $g$ and $g'$ into a single diagram. See \text{Figure \ref{fig:commutation}} for examples. The left two pictures show $g$ and $g'$, and the right pictures show the combined grid diagram and those polygons we shall use later.
	
	Here we illustrate how to fit the proof in \cite{MR3381987} into our case for column commutation, the row case is exactly the same. Here we use $\gamma$ to denote those $\beta$-curves after commutation.

	\begin{figure}
		\centering
		\begin{overpic}[width=1.0\textwidth]{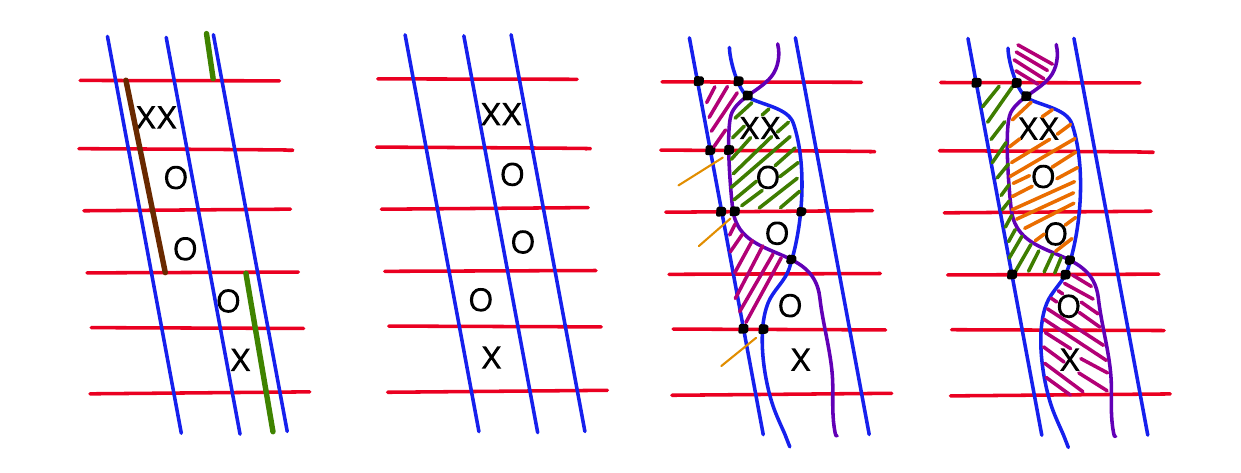}
			\put(5,29) {$LS_1$}
			\put(22,13) {$LS_2$}
			
			\put(60,33) {$x_1$}
			\put(53,33) {$y'_1$}
			\put(62,30) {$a$}
			\put(54,26.5) {$x_2$}
			\put(52.5,22.5) {$y'_2$}
			\put(65,16.5) {$b$}
			\put(57,9.6) {$x'_4$}
			\put(56,8) {$y_4$}
			\put(54,17) {$x'_3$}
			\put(64.7,19) {$x_3$}
			\put(55.5,19.8) {$y_3$}
			
			\put(85,30) {$a$}
			\put(79,32.5) {$x_1$}
			\put(76,32.5) {$y_1$}
			\put(78,14.5) {$x_2$}
			\put(82.5,14.5) {$y_2$}
			\put(88,16.5) {$b$}
			
			\put(58,35) {$\beta_i$}
			\put(61.5,35) {$\gamma_i$}
			\put(80.5,35) {$\beta_i$}
			\put(84,35) {$\gamma_i$}

		\end{overpic}
		
		\caption{\textbf{Commutation move and combined grid diagram.} In all pictures here, the $\alpha$-curves are in red and the $\beta$ curves are in blue. In the combined diagrams, we use $\gamma$ to denote the $\beta$ curves after commutation. We have $\gamma_j=\beta_j$ for $j\ne i$, and $\gamma_i$ is drawn in purple which intersect $\beta$ transversely in two points, $a$ and $b$. In (a) and (b), we draw $g$ and $g'$, respectively. The brown and green segments in (a) are $LS_1$ and $LS_2$, respectively, showing the commutation is available between these two columns. (c) and (d) are two copies of the combined grid diagram for this commutation. In (c), we mark two pentagons in pink: the upper one with a corner at $a$ contributes to $P$, the lower one with a corner at $b$ contributes to $P'$. Also in (c), we show a small triangle from $\xv$ to the nearest $\xv'$ in green. In (d), we show a hexagon in green that contributes to the map $H$, and also two bigons with corners $a$, $b$ in orange and red.}
		\label{fig:commutation}
	\end{figure}
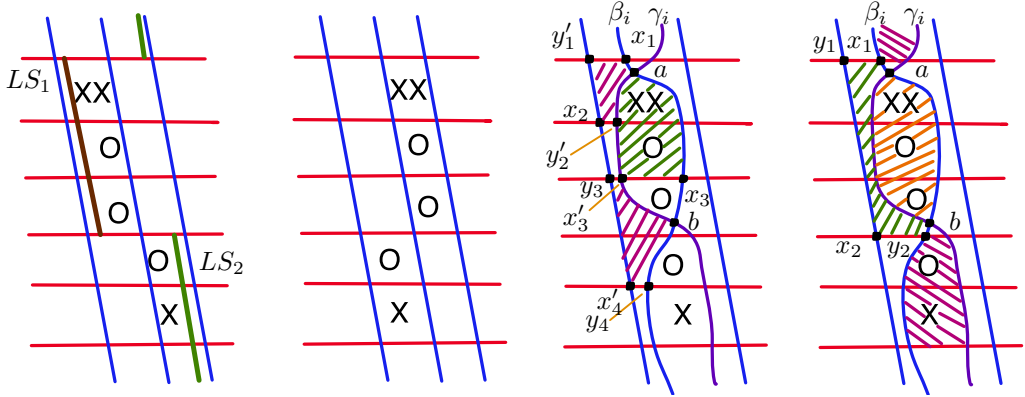
	
	Now we introduce the notion of a pentagon on a combined grid diagram. For $\xv\in \boldsymbol{S}(g)$ and $\yv' \in \boldsymbol{S}(g')$ satisfying $|\xv\cap \yv'|= n+k-2$, say $x_i=y'_i$ for $i\ge 3$, a \emph{pentagon} $p$ connecting $\xv$ to $\yv'$ is one embedded in the combined diagram with 
	\begin{itemize}
		\item Four corners of $p$ are at $\{x_1,x_2,y'_1,y'_2\}$, respectively;
		\item Each corner point $x$ of $p$ is an intersection between two curves from $\{\beta_t,\alpha_t,\gamma_{t}\}_{t=1}^{n+k}$. A small disk centered at $x$ is divided into four quadrants by these two curves, and $p$ contains exactly one of them.
		\item $\partial_{\alpha} p=\partial p\cap\AV=\yv'-\xv$. 
	\end{itemize}
	
	$\mathrm{Pent}(\xv,\yv')$ will denote the set of pentagons from $\xv$ to $\yv'$. A pentagon $p$ is \emph{empty} if $\mathrm{int}(p)\cap \xv=\emptyset$. The set of empty pentagons from $\xv$ to $\yv'$ will be denoted by $\mathrm{Pent}^\circ(\xv,\yv')$. Set $\mathrm{Pent}(\xv,\yv')=\emptyset$ if $|\xv\cap \yv'|\ne n+k-2$. Similarly, we can define a pentagon from $\yv'$ to $\xv$. One can see directly from the definition that in the first case, $p$ must have a corner point at $a$, while in the second case $p$ must have a corner point at $b$. See \text{Figure \ref{fig:commutation}} (c) for examples.
	
	In \cite{MR3677933}, the authors used the condition ``having all corners less than $90^\circ$'' to characterize the convexity of a pentagon. This is equivalent to the second condition in our definition.
	
	Define a $\FF[U_1,\ldots,U_{2n+k}]$ module map $P:CFK^-(g)\to CFK^-(g')$ by \[P(\xv)=\sum_{\yv'\in \boldsymbol{S}(g')}\sum_{\substack{p\in \mathrm{Pent}^\circ(\xv,\yv')\\p\cap \XXV=\emptyset}} \prod_{i=1}^{2n+k}U_i^{n_{O_i}(p)}\yv'.\]
	
	Similarly, one defines $P':CFK^-(g')\to CFK^-(g)$. 
	
	To show that $CFK^-(g)$ and $CFK^-(g')$ are quasi-isomorphic, we check that $P$ is a chain homotopy equivalence with $P'$ as its homotopy inverse. The routine check of $P$ and $P'$ being chain maps has been done in \cite{MR3381987} when the background manifold is $S^3$ and in \cite{MR3677933} when the link is regular in lens spaces; we fit their proof into our setup. For any pair of generators $\xv \in\boldsymbol{S}(g)$ and $\zv' \in\boldsymbol{S}(g')$, we consider for each domain $D$ that intersect $\XXV$ emptily from $\xv$ to $\zv'$, how many ways it can be decomposed into an empty rectangle union an empty pentagon. That is, we need to consider $N'(D)$ in \[\partial^- \circ P(\xv)+P\circ\partial^-(\xv)=\sum_{\zv'\in \mathcal{S}(g')} \sum_{\substack{D\in \pi_2(\xv,\zv') \\ D\cap \XXV=\emptyset}} N'(D) \prod_{i=1}^{2n+k} U^{O_i(r)} \zv'.\]
	
	Now, $M=\vert \xv-\xv\cap \zv'\vert$ may be 1, 2, 3 or 4. An argument similar to $(\partial^-)^2=0$ shows that $N'(D)$ will always be even, thus $\partial^- \circ P+P\circ\partial^-=0$. In the same vein, one can also show that $P'$ is a chain map. See \cite[Section 5.1]{MR3381987} for diagrams illustrating this fact.
	
	
	We also introduce a hexagon-counting map $H$ as follows:  
	For $\xv,\yv\in \boldsymbol{S}(g)$, an embedded disk $h$ in the torus with boundary on $\{\beta_t,\alpha_t,\gamma_{t}\}_{t=1}^{n+k}$  is called a \emph{hexagon} from $\xv$ to $\yv$ if 
	
	\begin{itemize}
		\item Four corners of $h$ are in $\xv\cup\yv$, and two other corners are at $a$, $b$ respectively;
		\item Each corner point x of $h$ is an intersection between two curves from $\{\beta_t,\alpha_t,\gamma_{t}\}_{t=1}^{n+k}$. A small disk centered at $x$ is divided into four quadrants by these two curves, and $h$ contains exactly one of them.
		\item $\partial_{\alpha} h=\partial h\cap\AV=\yv-\xv$.
	\end{itemize} 
	
	As for pentagons, we introduce $\mathrm{Hex}(\xv,\yv)$ and $\mathrm{Hex}^\circ(\xv,\yv)$ to denote the set of hexagons and empty hexagons from $\xv$ to $\yv$. We define $H:CFK^-(g)\to CFK^-(g)$ by
	\[H(\xv)= \sum_{\yv\in \boldsymbol{S}(g)} \sum_{\substack{h\in \mathrm{Hex}^\circ(\xv,\yv)\\h\cap \XXV=\emptyset}} \prod_{i=1}^{2n+k} U_{i}^{n_{O_i}(h)} \yv.\]

	This appears naturally as a chain homotopy equivalence between $P'\circ P$ and $\mathrm{Id}_{CFK^-(g)}$, that is \[\partial^- \circ H+H\circ \partial^-=P'\circ P -\mathrm{Id}_{CFK^-(g)}.\] 
	
	This identity is shown by considering decompositions of domains into a rectangle and a hexagon or two pentagons. For any $\xv \in \mathcal{S}(g)$, we have \[\partial^- \circ H (\xv)+H\circ \partial^-(\xv) -P'\circ P(\xv) =\sum_{\zv\in \mathcal{S}(g)} \sum_{D\in \pi_2(\xv,\zv)} N''(D)\prod_{i=1}^{2n+k} U_{i}^{n_{O_i}(D)} \zv,\] where $N''(D)$ is number of decomposition of $D$ into the form $r\ast h$, $p'\ast p$ or $h \ast r'$. 
	
	This is again an argument similar to $(\partial^-)^2=0$, which follows from existing proof in \cite[Section 5.1]{MR3381987} with an extra care in the case  $M=\vert \xv-\xv\cap \zv'\vert=2$. For an intuition, see \cite[Figure 5.9]{MR3381987}. 

	It remains to be seen that the relative bigradings are also link invariants, that is, $P$, $P'$, and $H$ are bigraded and of certain bi-grading, respectively. This is not immediate from the existing argument since our grading convention is a little bit complicated. We remedy this by lifting to the universal cover and then applying the existing result. This idea comes from the identification between combinatorially defined Maslov grading and original Maslov grading in \cite{MR2429242}.
	
	It is well-known that $S^3$ is the universal cover of $L(p,q)$ with $\ZZ/p\ZZ$ as deck transformation group. Denote the covering map by $\pi$. $L$ can be lifted to $\widetilde{L}=\pi^{-1}(L)$, which is a singular link in $S^3$, with possibly more components than $L$. Note that our function $\widetilde{W}$ takes the grid diagram $g$ of $L$ to a grid diagram $\widetilde{g}$ of $\widetilde{L}$. \cite[Figure 8]{MR2429242} gives a good illustration of this fact. 
	
	For any $\xv\in \boldsymbol{S}(g)$, there is a well-defined lift $\widetilde{\xv}\in \boldsymbol{S}(\widetilde{g})$. By definition of $\widetilde{W}$, we have
	\[M(\widetilde{\xv})-M(\widetilde{\yv})=p(M(\xv)-M(\yv)), \]
	whenever $\xv$ and $\yv$ belong to the same $\mathrm{Spin^c}$ structure. 
	
	For $\xv\in \boldsymbol{S}(g)$, there is a unique $\xv'\in \boldsymbol{S}(g')$ sharing $n+k-1$ points with $\xv$. There is a small triangle $t_{\xv}$ connecting them on the combined Heegaard diagram. See \text{Figure \ref{fig:commutation}}. Consider any $\yv'\in\boldsymbol{S}(g')$ and $p\in \mathrm{Pent}^\circ (\xv,\yv')$, which intersects $\XXV$ in an empty set. Lifting all these to $\widetilde{g}$($\widetilde{g'}$), the argument of \cite[Lemma 5.1.3]{MR3381987} shows that 
	\begin{equation}\label{Maslov x, x'}
		M(\widetilde{\xv})-M(\widetilde{\xv}')=-p+ 2\vert \widetilde{t_{\xv}}\cap \widetilde{\OV}\vert-2\vert \widetilde{t_{\xv} }\cap \widetilde{\XXV\XXV}\vert,
	\end{equation}
	\begin{equation}\label{Maslov x', y'}
		M(\widetilde{\xv}')-M(\widetilde{\yv}')=p- 2\vert \widetilde{p\ast t_{\xv} } \cap \widetilde{\OV}\vert+2\vert \widetilde{p\ast t_{\xv} }\cap \widetilde{\XXV\XXV}\vert.
	\end{equation}
	Here we add $\widetilde{ }$ to each set of points to denote its lift to $S^3$.	  
	
	More precisely, $p$ and $t_{\xv}$ each lift to a disjoint union of $p$ identical copies of pentagons/ triangles in the lift of the combined diagram to $S^3$. We can compute directly that
	\[\mathcal{J}(\widetilde{\xv},\widetilde{\xv})-\mathcal{J}(\widetilde{\xv}',\widetilde{\xv}')=0\] 
	\[\mathcal{J}(\widetilde{\OV}-\widetilde{\XXV\XXV},\widetilde{\OV}-\widetilde{\XXV\XXV})-\mathcal{J}(\widetilde{\OV}'-\widetilde{\XXV\XXV}',\widetilde{\OV}'-\widetilde{\XXV\XXV}')=-p,\] 
	while \[\mathcal{J}(\widetilde{\OV}-\widetilde{\XXV\XXV},\widetilde{\xv})-\mathcal{J}(\widetilde{\OV}'-\widetilde{\XXV\XXV},\widetilde{\xv}')=0 \text{ or }-p,\]
	the value in the last equation changes accordingly with  
	$\vert t_{\xv}\cap (\OV\cup -\XXV\XXV)\vert$. From this, \text{Equation \ref{Maslov x, x'}} follows. 
	
	Next, observe that $p\ast t_{\xv}$ is a rectangle from $\xv'$ to $\yv'$, which lifts to a union of $p$ rectangles connecting $\widetilde{\xv}'$ to $\widetilde{\yv}'$. Then \text{Equation \ref{Maslov x', y'}} follows from the definition of $M$.
	
	Now one sees \[M(\widetilde{\xv})-M(\widetilde{\yv}')=- 2\vert \widetilde{p} \cap \widetilde{\OV}\vert,\] since $p\cap \XXV\XXV=\widetilde{p}\cap \widetilde{\XXV\XXV}=\emptyset$.
	
	This, together with the covering formula of $M$ shows that $P$ preserves $M$. A similar argument using $M_{\XXV}$ in place of $M_{\OV}$ shows that $P$ preserves $M_{\XXV}$ also, thus it  preserves $A$. Now we can conclude $P$ is birgraded. With the roles of $g$ and $g'$ exchanged, the argument above shows also $P'$ is bigraded.
	
	Using the same idea, we can argue that $H$ is homogeneous of degree $(1,0)$. Again, we illustrate the idea using Maslov grading. If $h$ is a hexagon in $\mathrm{Hex}^\circ(\xv,\yv)$, then $h \ast B$ is a rectangle from $\xv$ to $\yv$ in which $B$ is one of the bigons with corners at $a$ and $b$ (See \text{Figure \ref{fig:commutation}} (d)). This rectangle lifts to disjoint union of rectangles in $\widetilde{g}$m while the bigon $B$ contains a pair of $O$, $X$ or a triple $\{O,O,XX\}$. Thus, \[M(\widetilde{\xv})-M(\widetilde{\yv})=-p-2\vert (\widetilde{h\ast B}\cap (\widetilde{\OV}-\widetilde{\XXV\XXV}))\vert=p(-1-2\vert h \ast B\cap (\OV-\XXV\XXV)\vert),\] showing that $H$ increases $M$ by 1.
	
	Now we have verified the grading problem, so we can conclude that when $g$ and $g'$ are grid diagrams that differ by a commutation, $CFK^-(g) \simeq CFK^-(g')$ are quasi-isomorphic relatively bigraded chain complexes  while $HFK^-(g) \cong HFK^-(g')$ are isomorphic relatively bigraded modules over $\FF[U_1,\ldots,U_{2n+m}]$. Thus, we can conclude commutation invariance of our homology theories. 
	
	In fact, our proof of the grading shift property can be used to show the following lemma. 
	\begin{lem}\label{relative grading through funtion n} (\cite[Proposition 2.11]{Celoria_2023}, see also \cite[Proposition 4.4]{tripp2021gridhomologylensspace})
		If $\xv,\yv \in \mathcal{S}$ can be connected by $r\in \mathrm{Rect}^\circ (\xv,\yv)$, then we have \[M(\xv,\yv)=1+2\sum_{i=1}^n{n_{XX_i}}(r)-2\sum_{i=1}^{2n+k} n_{O_i}(r),\]
		\[A(\xv,\yv)=2\sum_{i=1}^n{n_{XX_i}}(r)+\sum_{i=1}^k {n_{X_i}}(r)-\sum_{i=1}^{2n+k} n_{O_i}(r).\]
		
	\end{lem}  
	With the help of this, we can avoid detailed computation, as discussed above, in the following sections.

	\subsubsection{Stabilization invariance}\label{Stabilization invariance}
 	\cite[Corollary 3.2.3]{MR3381987} can be reinterpreted in our terminology as follows: \begin{lem}
		Any stabilization can be realized as a composition of a stabilization of type $XX:SW$ or $X:SW$ with a sequence of commutations.
	\end{lem}
	
	With the help of this lemma, it suffices for us to show that when diagrams $g$ and $g'$ can be related by one such stabilization, $CFK^-(g)$ and $CFK^-(g')$ are quasi-isomorphic as chain complexes of $\FF[U_{1},\ldots,U_{m+2n}]$ modules.
	
	In \text{Figure \ref{fig:stabilization}}, we show a pair of grid diagrams $g$ and $g'$ which differ by a stabilization of type $XX:SW$. For further reference, we call the newly introduced marked points $O_{\text{new}}$ and $X_{\text{new}}$. When we perform an $X:SW$ stabilization, we cannot tell which $X$ near $O_{\text{new}}$ is the new one, we choose the one sharing the same column with $O_{\text{new}}$ to be $X_{\text{new}}$. We will need some algebraic terminology in our argument. For detail on this, one can refer to \cite[Section 5.2.2]{MR3381987}. 
	
	\begin{figure}
		\centering
		\includegraphics[width=0.6\linewidth]{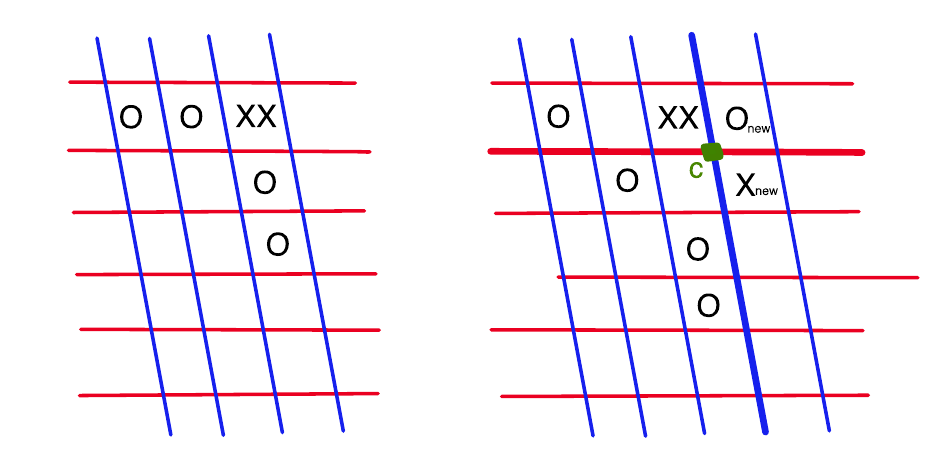}
		\caption{\textbf{Stabilization.} On the left, we show part of a grid diagram; on the right, we show the result after a $XX:SW$ stabilization. The newly introduced curves are bolder than the old ones, and the special intersection point $c$ is marked in green.}
		\label{fig:stabilization}
	\end{figure}

	The intersection point $c$ of the newly introduced $\beta$ and $\alpha$ curve, as shown in \text{Figure \ref{fig:stabilization}}, will lie at the heart of our argument. Among $p$ intersection points, we specify $c$ to be the corner shared by the new $X$ mark and the $XX$ (or $X$) at which we perform stabilization at.
	
	To achieve our goal, we first deal with the base ring change, as there is a new variable $U_{\text{new}}$ in the base ring of $CFK^-(g')$ accounting for $O_{\text{new}}$ .
	
	\begin{prop}\label{stabilization: cone isom to g'}
		Suppose that $g$ and $g'$ are given as above, of size $n+k$ and $n+k+1$, respectively. Take $U_j$ to be the variable corresponding to any $O_j$ originally in $g$ and lying in the same thin edge or $S^1$ component as $O_{\text{new}}$. Then there is a quasi-isomorphism of bigraded chain complexes over $\FF[U_{\text{new}},U_1,\ldots,U_{2n+k}]$ from $CFK^-(g')$ to $\mathrm{Cone}(U_{\text{new}}-U_j)$, where $\mathrm{Cone} (U_{\text{new}}-U_j)$ is the mapping cone of multiplication by $U_{\text{new}}-U_j$ on $CFK^-(g)[U_{\text{new}}]$. 
	\end{prop}
	
	This proposition together with an algebraic lemma will complete our proof.
	\begin{lem}
		Let $C$ be a bigraded chain complex over $\FF[U_1,\ldots,U_{2n+k}]$. Then there is an isomorphism of bigraded $\FF[U_1,\ldots,U_{2n+k}]$-modules \[H(\mathrm{Cone}(U_{\text{new}}-U_{j}:C[U_{\text{new}}]\to C[U_{\text{new}}]))\cong H(C).\]
	\end{lem}
	Here $U_j$ is chosen as in \text{Proposition \ref{stabilization: cone isom to g'}}. This lemma follows from \cite[Lemma 5.2.16]{MR3381987} with only some notation changes.
	
	It follows from \text{Proposition \ref{homotopic U_i action}} that the choice of $U_j$ here is irrelevant to the quasi-isomorphism result.
	
	Using the special intersection point $c$, we split $\boldsymbol{S}(g')$ into $\BI\cup \BN$, where $\BI$ is the set of generators containing $c$ and $\BN$ is its complement. Let $\BI^-$ and $\BN^-$ denote the submodule of $CFK^-(g')$ generated by $\BI$ and $\BN$, respectively. Since we block all $X$-marks in our theory, $\BN^-$ is actually a subcomplex. Now the differential on $CFK^-(g')$ can be written in matrix form\[\partial^-=\begin{pmatrix}
		\partial^{\BI}_{\BI} & 0\\
		\partial^{\BN}_{\BI} & \partial^{\BN}_{\BN}
	\end{pmatrix},\]
	so $CFK^-(g')$ is the mapping cone of $\partial^{\BN}_{\BI}:(\BI^-,\partial^{\BI}_{\BI})\to (\BN^-,\partial^{\BN}_{\BN}) $.
	
	Note that there is a natural bijection $c:\boldsymbol{S}(g)\to \BI$, $\xv\mapsto \xv'=\xv\cup \{c\}$. This leads to an isomorphism of chain complexes $e:(\BI^-,\partial^{\BI}_{\BI})\to CFK^-(g)[U_{\text{new}}]\left[1,1\right]$. Here and below, we use $C[l,k]$ to denote the complex $C$ with bigrading $(M,A)$ shifted up by $(l,k)$. The proof of this fact follows immediately from the  natural bijection on the sets of generators and rectangles. 
	
	Again, we need to deal with the grading problem. On one hand, we can lift the singular link $L$ as well as diagrams $g$ and $g'$ via $\pi$ to $\widetilde{L}$, $\widetilde{g}$ and $\widetilde{g'}$. Using the fixed absolute lift, one can compute directly  $M(\widetilde{\xv'})-M(\widetilde{\xv})=-p$,  $A(\widetilde{\xv'})-A(\widetilde{\xv})=-p$ using the combinatorial formula, which is straightforward in the $S^3$ case. On the other hand, we can appeal to \text{Lemma \ref{relative grading through funtion n}} and see the grading shift in a direct way.
	
	Now, we introduce rectangle-counting maps to relate $(\BI^-,\partial^{\BI}_{\BI})$ and $(\BN^-,\partial^{\BN}_{\BN})$.
	Define \begin{itemize}
		\item $H_{X_{\text{new}}}:\BN^-\to \BI^-$ by 
		\[H_{X_{\text{new}}}(\xv) =\sum_{\yv\in \BI}
		\sum_{\substack{r\in \mathrm{Rect}^\circ(\xv,\yv)\\\mathrm{int}(r)\cap\XXV=X_{\text{new}}}}
		\prod_{i=1}^{2n+k} U_{i}^{n_{O_i}(r)}\cdot U_{\text{new}}^ {n_{O_{\text{new}}}(r)} \yv.\]
		\item $H_{O_{\text{new}}}:\BI^-\to \BN^-$ by 
		\[H_{O_{\text{new}}}(\xv) =\sum_{\yv\in \BN}
		\sum_{\substack{r\in \mathrm{Rect}^\circ(\xv,\yv)\\\mathrm{int}(r)\cap\XXV=\emptyset, O_{\text{new}}\in int (r)}}
		\prod_{i=1}^{2n+k} U_{i}^{n_{O_i}(r)}\yv.\]
		\item $H_{X_{\text{new}},O_{\text{new}}}:\BN^-\to \BN^-$ by 
		\[H_{X_{\text{new}},O_{\text{new}}}(\xv) =\sum_{\yv\in \BN}
		\sum_{\substack{r\in \mathrm{Rect}^\circ(\xv,\yv)\\\mathrm{int}(r)\cap\XXV=X_{\text{new}}, O_{\text{new}}\in int (r)}}
		\prod_{i=1}^{2n+k}U_{i}^{n_{O_i}(r)} \yv.\]
	\end{itemize}
	
	Since all three maps concern generators in the grid diagram $g'$, we can verify directly using the lifted diagram $\widetilde{g'}$ that $H_{X_{\text{new}}}$, $H_{O_{\text{new}}}$ and $H_{X_{\text{new}},O_{\text{new}}}$ are of bi-grading $(-1,-1)$, $(1,1)$ and $(1,0)$, respectively. We can also appeal to \text{Lemma \ref{relative grading through funtion n}} and prove the grading shift formula by counting marked points in rectangles that are taken into account in each map.
	
	It can be shown as in \cite{tripp2021gridhomologylensspace} that \begin{lem}
		$H_{X_{\text{new}}}$ and $H_{O_{\text{new}}}$ are chain maps: \[\partial^{\BI}_{\BI}\circ H_{X_{\text{new}}}= H_{X_{\text{new}}} \circ \partial^{\BN}_{\BN},\]
		\[\partial^{\BN}_{\BN}\circ H_{O_{\text{new}}}= H_{O_{\text{new}}} \circ \partial^{\BI}_{\BI}.\]
	\end{lem}
	
	As we discussed during the proof of $\partial^2=0$, we consider, for each pair of generators $(\xv,\zv)$ in $\BN$ or $\BI$ and each domain $D$ from $\xv$ to $\zv$, how many times it contributes to $H_{X_{\text{new}}}\circ H_{O_{\text{new}}}$ and $H_{O_{\text{new}}}\circ H_{X_{\text{new}}}+ \partial^{\BN}_{\BN}\circ H_{X_{\text{new}},O_{\text{new}}}+ H_{X_{\text{new}},O_{\text{new}}}\circ \partial^{\BN}_{\BN}$. For $H_{X_{\text{new}}}\circ H_{O_{\text{new}}}$, our life is easy since only thin annuli are involved, which contribute $\xv\mapsto \xv$ for each $\xv\in \BI$. In case of $\BN$, note that all maps in these terms are rectangle counting maps, so \text{Remark \ref{rmk on N(D)}} applies. In each decomposition $D=r_1\ast r_2$, it contributes to one of the four terms depending on how $O_{\text{new}}$ and $X_{\text{new}}$ distribute in the two rectangles. So $N(D)=2$ except in the case that $D$ is the thin annulus through $X_{\text{new}}$ and $O_{\text{new}}$. Such an annulus contributes $\xv \mapsto \xv$  for every $\xv \in \BN$. These together show the identities:
	\[H_{X_{\text{new}}}\circ H_{O_{\text{new}}}=\mathrm{Id}_{\BI^-};\]
	\[H_{O_{\text{new}}}\circ H_{X_{\text{new}}}+ \partial^{\BN}_{\BN}\circ H_{X_{\text{new}},O_{\text{new}}}+ H_{X_{\text{new}},O_{\text{new}}}\circ \partial^{\BN}_{\BN}=\mathrm{Id}_{\BN^-}.\]
	
	Combining the computations above, we have a commutative square 
	\[\xymatrix{(\BI^-,\partial_{\BI}^{\BI}) \ar[d]^{e} \ar[r]^{\partial^{\BN}_{\BI}} &(\BN^-,\partial_{\BN}^{\BN}) \ar[d]^{e\circ H_{X_{\text{new}}}}\\
		(CFK^-(g)[U_{\text{new}}][1,1],\partial^-_{g}) \ar[r]^{U_{\text{new}}-U_j} &(CFK^-(g)[U_{\text{new}}][1,1],\partial^-_{g})\\}. \]
	
	This, together with \cite[Lemma 5.2.12]{MR3381987} (property of mapping cone) leads to \text{Proposition \ref{stabilization: cone isom to g'}}.
	
	Thus, we can conclude that when $g$ and $g'$ are grid diagrams that differ by a single stabilization, we have a quasi-isomorphism of relatively bigraded chain complexes $CFK^-(g)\simeq CFK^-(g')$ and an isomorphism of relatively bigraded $\FF[U_{1},\ldots,U_{2n+m}]$-modules $HFK^-(g)\cong HFK^-(g')$. 
	
	In summary, for a singular link $L$ embedded in $L(p,q)$ and any grid diagram $g$ representing it, $HFK^-(L)=: HFK^-(g)$ and $\widehat{HFK}(L)=:\widehat{HFK}(g)$ are well-defined invariants for the isotopy class of $L$.
	
	\begin{rmk}
		In \cite[Section 5]{tripp2021gridhomologylensspace}, they showed invariance of grid homology for regular lens space links in great detail. Although we did not provide a detailed analysis on decompositions of domains, \text{Remark \ref{rmk on N(D)}} together with their investigation on higher polygons forms a complete proof.
	\end{rmk}

	\section{Resolution cube}\label{Resolution cube}
	Having defined grid homology for singular links in lens spaces, we now try to relate the grid homologies of links that differ by a local change and provide a way to do computations through grid homology of those links that admit a torus diagram without any crossings (defined as ``totally singular'' in \cite{MR2574747}). More concretely, our main tasks in this section are \begin{enumerate}
		\item Proving a skein exact sequence, which acts as the  foundation of a resolution cube;
		\item Constructing a spectral sequence starting from the resolution cube and converging to the grid homology of links in lens spaces.
	\end{enumerate}
	
	For grid homology of regular links in lens spaces, one can regard regular knots and links as singular ones without vertices or one can refer to \cite{MR2429242}. In \cite{MR2429242}, the authors also identified the grid homology for regular lens space links with $HFL^\circ$ defined in \cite{MR2443092}. So, we are actually proving a resolution cube of $HFL^\circ$ specialized to grid diagrams. The construction in this section mostly follows \cite[Section 4, 6]{MR2574747}. Although all the existing resolution cubes used braid diagrams of knots and links, as remarked in \cite[Section 4]{MR2574747}, this construction has no dependence on the property of braids.
	
	\subsection{Resolution of crossings}
	
	For knots in $S^3$, it is natural to consider resolution on a planar diagram. In our case, resolving crossings on a torus projection appears as a natural choice. Our singular links appear with orientation, so we can assign a sign to each crossing. See \text{Figure \ref{fig:resolution of crossings}}. For a diagram with oriented crossings, a usual way to form a resolution cube is to consider oriented resolution and disoriented resolution at each crossing. This was used to construct Khovanov homology. However, we will use singularization in place of disoriented resolution and that's where singular links appear.
	
	\begin{defi}
		At each crossing in a torus diagram for a possibly singular link in lens space $L(p,q)$, we define its \emph{smoothing} and \emph{singularization} as shown in \text{Figure \ref{fig:resolution of crossings}}. For a positive crossing, we call its smoothing the \emph{1-resolution} and its singularization the \emph{0-resolution}. For a negative crossing, the notions of 0 and 1 exchange.
		
	\end{defi}
	
	\begin{figure}
		\begin{overpic}[width=0.7\textwidth]{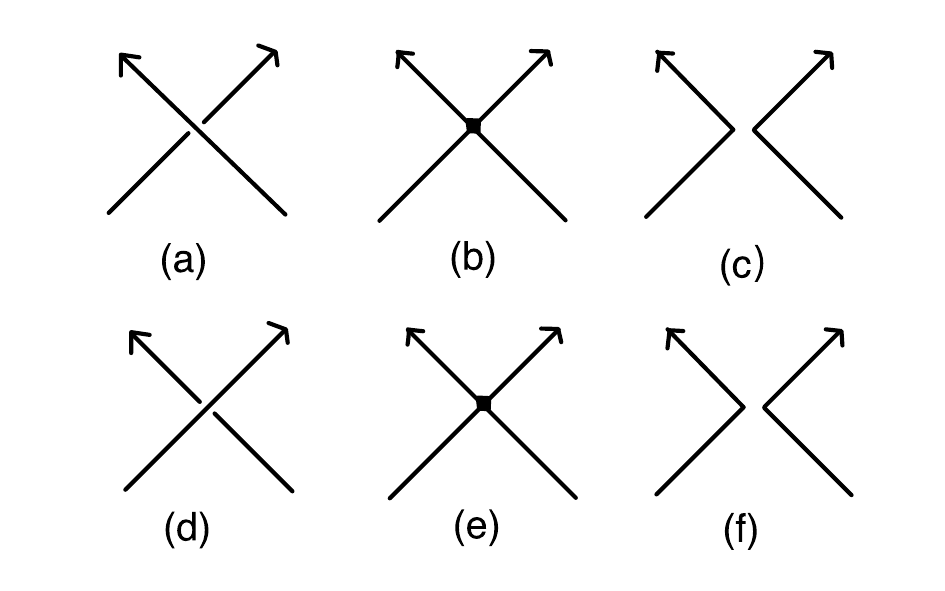}
			
		\end{overpic}
		
		\caption{\textbf{Resolution of crossings.} In (a), we show a negative crossing. In (b) and (c) we show its corresponding singularization and smoothing. Similarly, in the second row, we show a positive crossing together with its singularization and smoothing in (d), (e), and (f), respectively.}
		\label{fig:resolution of crossings}

	\end{figure}

	Note that for each pair of $X$, $O$ marks sharing a row or a column, there are two ways to connect them in the complement of $\alpha$ or $\beta$ curves. Recall that we cut $\Sigma$ open along a pair of generic representatives for a symplectic basis of $H_1(\Sigma;\ZZ)$ to obtain a planar diagram. We fix the convention that on this diagram, those vertical segments emanating from $X$ always go upward diagrammatically. From now on, we mean by $g$ a grid diagram together with a fixed way of connecting base points to obtain the link projection subject to the above convention. Thus, each positive or negative crossing locally looks like \text{Figure \ref{fig:standard picture at each crossing}} (d) or (a). There are certain cells containing crossings in the grid diagram, which will be called \emph{crossing squares}. To make resolving a crossing available on a grid diagram, we need the assumption of special on diagram in the following sense:
	
	\begin{defi}\label{special grid diagram}(\cite[Definition 6.1]{MR2574747})
		A grid diagram is called \emph{special} if \begin{enumerate}
			\item each vertical arc crosses at most one horizontal arc, each horizontal arc goes under at most one vertical arc, i.e each column or row has at most one crossing square;
			\item no two crossing squares share a corner;
			\item each crossing square shares two sides with squares marked by $X$; the corner shared by the two $X$-labeled cells will be called the two $X$s' \emph{crossing corner};
			
			\item when a rectangle has two of its corners as crossing corners, it contains an $X$ in the interior.
		\end{enumerate}
	\end{defi}
	
	Using grid moves, one can show that every link in a lens space possesses a special diagram. Indeed, the first two requirements can be achieved by sufficiently many stabilization moves. The third one can be achieved by stabilization and commutations. In \text{Figure \ref{fig:standard picture at each crossing}}, we define the \emph{standard picture at crossings} that meets the need of (3) and show how to achieve it for each positive or negative crossing. Our resolution cube will be constructed based on grid diagrams that are special and standard at each crossing. The fourth one is guaranteed once we have (1) and the standard picture of (3) in hand.
	
	\begin{figure}
		
		\begin{overpic}[width=0.8\textwidth]{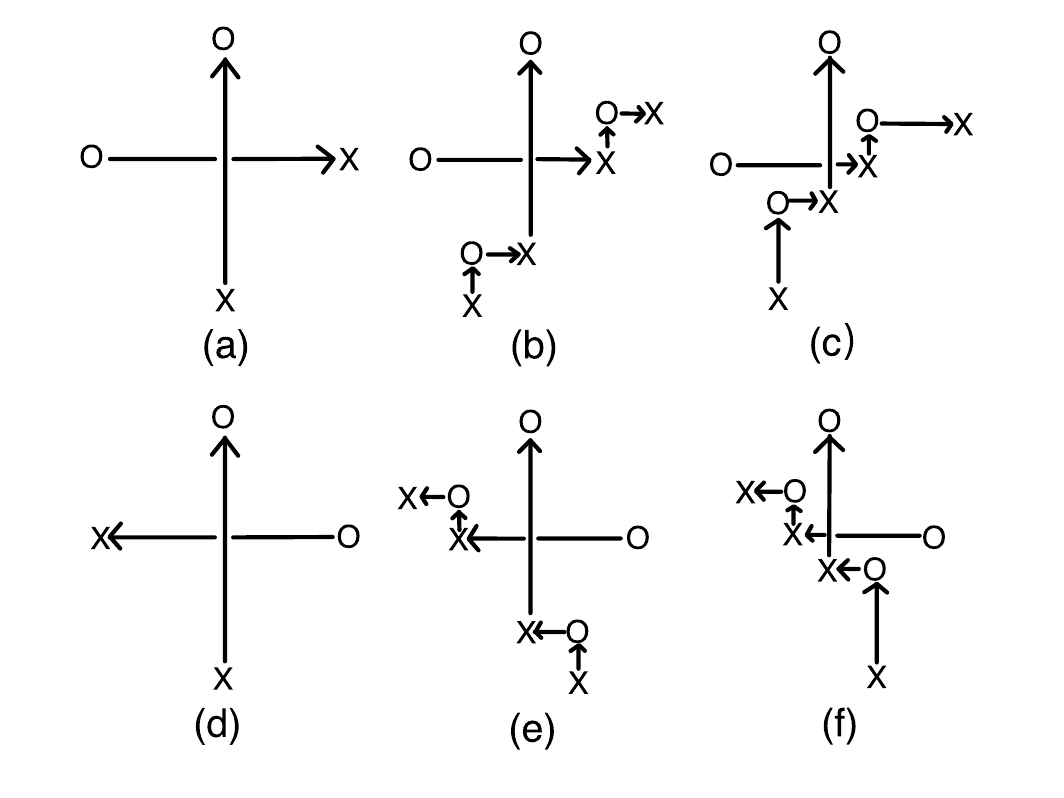}
			
		\end{overpic}
		
		\caption{\textbf{Standard picture at each crossing.} In (a) and (d), we show how a negative and positive crossing look like after fixing our convention. In (b) and (c), we show how to achieve (3) in the definition of a special grid diagram  starting from (a) via a sequence of stabilization and commutation. Similarly, we do so for (d) in (e) and (f). The patterns in (c) and (f) are defined to be the standard picture for a negative or positive crossing respectively.}
		\label{fig:standard picture at each crossing}
		
	\end{figure}

	Using the standard picture, we can resolve a crossing on a diagram. That is, we can change a diagram locally to get one for the singularized or smoothed link at each crossing. See \text{Figure \ref{fig:grid diagram of resolution and singularization}}.
	\begin{figure}
		\begin{overpic}[width=0.85\textwidth]{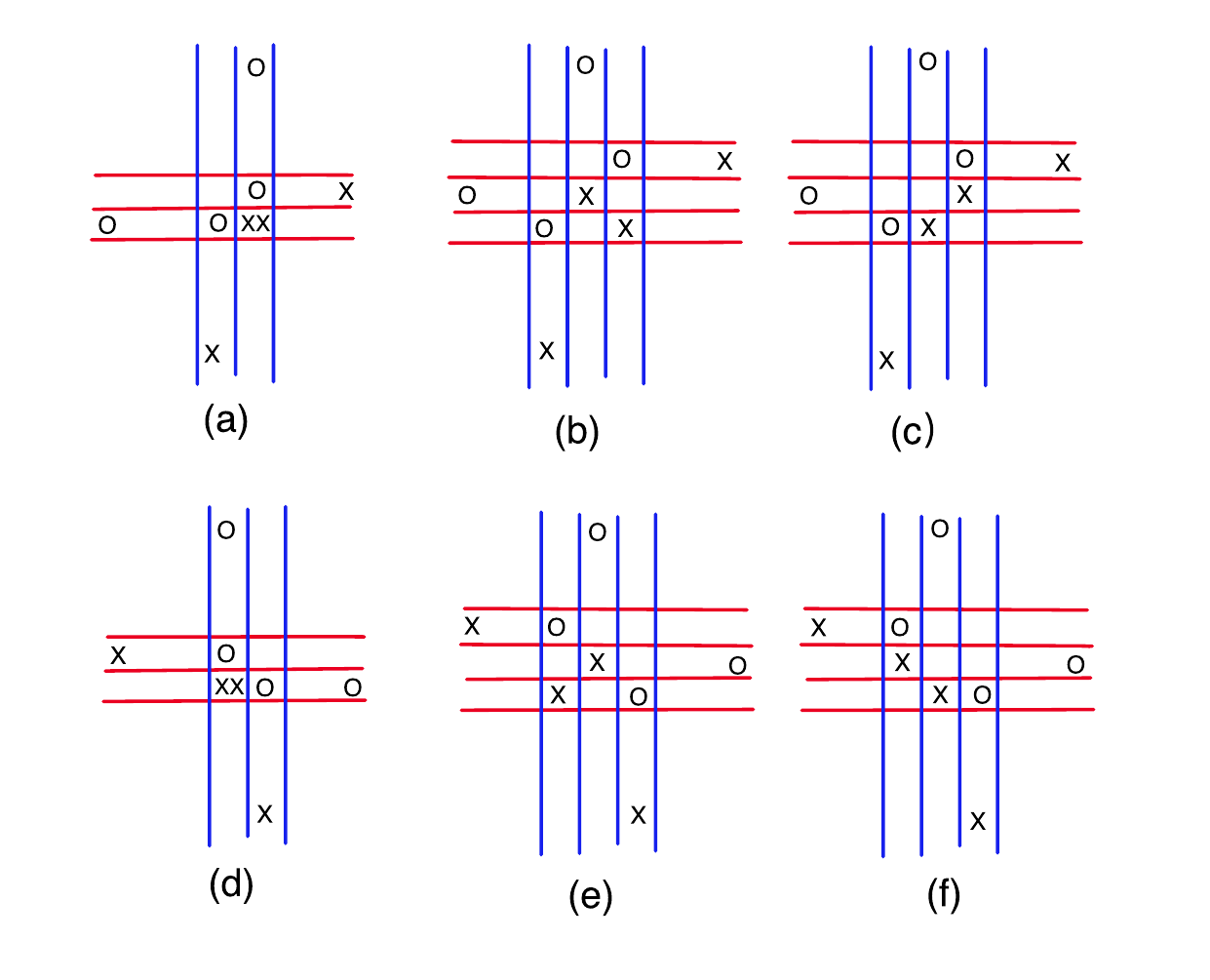}

		\end{overpic}
		
		\caption{\textbf{Realizing resolution on grid diagram.} (c) and (f) here are the same as in \text{Figure \ref{fig:standard picture at each crossing}}. In (b)(resp. (e)), we show a grid diagram for the smoothing of the crossing in (c) (resp. (f)). In (a)(resp. (d)), we merge the two rows as well as the two columns containing crossing $X$s into one singular row and one singular column. This gives rise to a grid diagram for the singular knot or link after singularizing at this crossing. }
		\label{fig:grid diagram of resolution and singularization}

	\end{figure}

	From now on, we fix a (possibly) singular link $L$ in $L(p,q)$ and choose a grid diagram $g_0$ of it. We apply the procedure above to get a special grid diagram $g$. Now, take any of its crossings $c$. As we did when proving commutation invariance, we can form a combined diagram that contains the information of the original link $L$ and the resolved link $L_r$. With a little effort, the singularized link $L_s$ can also be reconstructed from this diagram. See \text{Figure \ref{fig:AB diagram}}.

	For further reference, we fix some notation here. For a grid diagram $g$, we denote by \begin{itemize}
		\item $X(-)$, the set of singular points;
		\item $C(-)$, the set of crossings.
	\end{itemize} 
	Note that $X(g)$ depends only on the underlying oriented graph, while $C(g)$ depends on the diagram $g$ itself.

	\begin{figure}
		\begin{overpic}[width=0.85\textwidth]{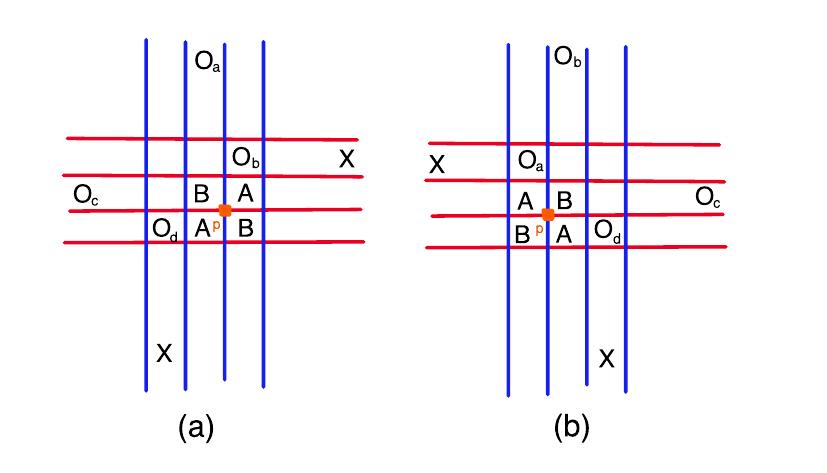}

		\end{overpic}
		
		\caption{\textbf{Combined diagram for resolution of crossing.} (a) is a combined diagram for $L_r$ and $L$ when $c$ is a negative crossing. When placing $X$ marks at $B$s, we get the grid diagram $g_r$, while placing $X$ marks at $A$s, we get $g$. By compressing two rows and columns containing $A$, $B$ into single ones, we get $g_s$. (b) shows one when $c$ is positive. In both cases, $O$-base points are marked as in the description of \text{Theorem \ref{skein exact sequence}}, and $p$ is the special intersection point we shall use in the proof.}
		\label{fig:AB diagram}

	\end{figure}

	\begin{defi}
		For any grid diagram $g$, a \emph{complete resolution} is an assignment $I:C(g)\to \{0,1\}$ that assigns to each crossing its 0 or 1 resolution. When $g$ is a special grid diagram with the standard local picture as in \text{Figure \ref{fig:standard picture at each crossing}} at each crossing, we can apply the procedure in \text{Figure \ref{fig:grid diagram of resolution and singularization}} to get a canonical diagram for the resolved link. This will be denoted by $g_I$ for each complete resolution.
	\end{defi}
	\subsection{Skein exact sequence}\label{Skein exact sequence}
	
	To form a resolution cube, we first prove the existence of a skein exact sequence.
	
	\begin{thm}\label{skein exact sequence}(\cite[Theorem 4.1]{MR2574747})
		Let $L$, $c$, and $g$ be fixed as above. Denote the resulting grid diagrams for smoothing and singularization at $c$ by $g_r$ and $g_s$, and call the corresponding links $L_r$ and $L_s$, respectively. Let $O_a$ and $O_b$ be marks sharing the same column with the newly formed $XX$ base point, and $O_c$ and $O_d$ be marks sharing the same row with the newly formed $XX$ base point. Then we have the following:
		\begin{itemize}
			\item When $c$ is positive, there is an exact sequence:
			\begin{equation}\label{positive skein}
				\ldots\longrightarrow HFK^-(L)\longrightarrow H_*(\frac{CFK^-(L_s)}{U_a+U_b-U_c-U_d}) \longrightarrow HFK^-(L_r)\longrightarrow HFK^-(L)\longrightarrow\ldots
			\end{equation}
			
			\item When $c$ is negative, there is an exact sequence:
			\begin{equation}\label{negative skein}
				\ldots\longrightarrow HFK^-(L) \longrightarrow HFK^-(L_r)\longrightarrow  H_*(\frac{CFK^-(L_s)}{U_a+U_b-U_c-U_d}) \longrightarrow HFK^-(L)\longrightarrow\ldots
			\end{equation}
			
		\end{itemize}

	\end{thm}
	
	We will relate the complexes defined by $g$, $g_r$ and $g_s$ given their close relationship. For further use, we set up more notations: \begin{itemize}
		\item $\XXV_0$ will denote the set of all $X$ or $XX$ base points except the two adjacent to the crossing square of $c$, which is shared by $g$, $g_r$ and $g_s$. $\XXV(-)$ denotes the set of all $X$ and $XX$ base points in a grid diagram.
		\item $\OV$ will denote the set of $O$ base points, shared by $g$, $g_r$ and $g_s$.
		\item $\mathbb{B}$ and $\mathbb{A}$ will denote the sets of $B$ and $A$ marks, respectively.
		
	\end{itemize}  
	
	We first assume $c$ is positive. As in the proof of stabilization invariance, the intersection point $p$ at the shared corner of the two $X$-labeled squares will be of importance. Consider a decomposition $\mathcal{S}=\mathcal{P}\cup \mathcal{C}$, where $\mathcal{P}$ is the set of generators containing $p$, and  $\mathcal{C}$ is the complement of $\mathcal{P}$ in $\mathcal{S}$. Let $Z$ be the submodule of $CFK^-(g_r)$ generated by $\mathcal{P}$. $Z$ is actually a subcomplex of $CFK^-(g_r)$, since we only allow rectangles without $X$ marks in our differential. Let $Y$ be the quotient complex $CFK^-(g_r)/Z$. Then $Y$ has $\mathcal{C}$ as a generating set, and the induced differential on $Y$ counts rectangles without any $A$, $B$ or $X$ in their interior.
	
	\begin{lem}\label{g_s isom to Z}
		The chain complex $Z$ is isomorphic to $CFK^-(g_s)$.
	\end{lem}
	\begin{proof}
		We shall construct a chain homotopy equivalence $Q:CFK^-(g_s)\to Z$ as follows. Observe that except for the $\alpha$ and $\beta$ curve meeting at $p$, and the splitting of a pair of singular column and row into two regular columns and rows, $g_r$ is exactly the same as $g_s$. Thus, we have a bijection $Q:\mathcal{S}(g_s)\to \mathcal{P}$ given by $\xv\mapsto \xv'=\xv\cup\{p\}$. This leads to a map $Q:CFK^-(g_s)\to Z$. We claim that it is a chain map. This follows from the observation that we also have a bijection \[\{\text{rectangles counted by } \partial^-_{g_s}\} \xrightarrow{ 1:1 } \{\text{rectangles counted by } \partial^-_{Z}\}.\]

		Indeed, consider a pair $\xv,\yv\in \mathcal{S}(g_s)$, for which there is a rectangle $r$ of size $a\times b$ in $\mathrm{Rect}^\circ (\xv,\yv) $ that intersects $\XXV(g_s)$ in an empty set. Denote the column and row the newly introduced XX belongs to by $C_s$ and $R_s$. If $\mathrm{int}(r)\cap C_s=\mathrm{int}(r)\cap R_s=\emptyset$, then the image $r'$  of $r$ naturally lives in $\mathrm{Rect}^\circ(\xv',\yv')$. Note that since $\mathrm{int}(r)\cap \XXV=\emptyset$, $r$ cannot intersect both $C_s$ and $R_s$. If $\mathrm{int}(r)\cap C_s\ne\emptyset$, then by augmenting the part of size $1\times b$ in the intersection into a sub-rectangle of size $2\times b$, we get $r'$ embedding in $g_r$ which belongs to $\mathrm{Rect}^\circ(\xv',\yv')$. A similar operation can be done if $\mathrm{int}(r)\cap R_s\ne\emptyset$. Now we have a well-defined map from the left-hand side to the right-hand side. On the other hand, for any pair $\xv',\yv'\in \mathcal{P}$ and $r'\in \mathrm{Rect}^\circ(\xv',\yv')$, $r'$ does not have any corner at $p$ and must have empty intersection with the four squares that have $p$ as a corner. Thus, any such $r'$ gives rise to a unique $r$ in $\mathrm{Rect}^\circ(\xv,\yv)$ during the merge of rows and columns. One can see directly that this procedure is inverse to the one described above. This shows that $\partial^-_{Z}\circ Q=Q\circ\partial^-_{g_s}$, i.e., $Q$ is a chain map.
		
		In summary, for $Z$ and $CFK^-(g_s)$, we have a chain map $Q$ induced by one-to-one correspondence on generators. Thus, they are isomorphic chain complexes.
		
	\end{proof}
	
	On the other hand, $Y$ appears naturally as a subcomplex of $CFK^-(g)$, with quotient complex $CFK^-(g)/Y$ isomorphic to $Z$. Define maps $\Phi_A:Y\to Z$
	\[\Phi_A(\xv)=\sum_{\yv\in \mathcal{P}} \sum_{\substack{r\in \mathrm{Rect}^\circ(\xv,\yv)\\n_{\XXV_0}(r)=0, n_{\mathbb{A}}(r)=1}} \prod_{O_i\in \OV} U_i^{n_{O_i}(r)}\yv;\]
	
	and $\Phi_B: Z\to Y$,
	\[\Phi_B(\xv)=\sum_{\yv\in \mathcal{C}} \sum_{\substack{r\in \mathrm{Rect}^\circ(\xv,\yv)\\n_{\XXV_0}(r)=0, n_{\mathbb{B}}(r)=1}} \prod_{O_i\in \OV} U_i^{n_{O_i}(r)}\yv.\]
	
	It again follows from \text{Remark \ref{rmk on N(D)}} that these are chain maps, since $\Phi_A$, $\Phi_B$ and $\partial^-_{Y}$, $\partial^-_{Z}$ are all rectangle counting maps.
	
	One can see directly that $CFK^-(g_r)$ appears as the mapping cone of $\Phi_A$, while $CFK^-(g)$ appears as the mapping cone of $\Phi_B$.
	\begin{lem}\label{composition to U multiplication}(\cite[Lemma 4.3]{MR2574747})
		The compositions $\Phi_A\circ \Phi_B$ and $\Phi_B\circ\Phi_A$ are both equal to multiplication by $U_c+U_d-U_a-U_b$.
	\end{lem}
	\begin{proof}	
		Note that both maps are compositions of rectangle counting maps, so \text{Remark \ref{rmk on N(D)}} and an argument similar to the one for Proposition \ref{homotopic U_i action} leads to the desired equality.
		
	\end{proof}
	
	When $c$ is a negative crossing, we still have a special point p in our combined diagram. But now $Z$ is a subcomplex of $CFK^-(g)$ with quotient complex $Y$, while $Y$ is a subcomplex of $CFK^-(g_r)$, with quotient complex $Z$. For $\Phi_A: Z \to Y$ and $\Phi_B: Y\to Z$ defined using similar expressions (except the direction of arrows changes), we have $CFK^-(g)$ being the mapping cone of $\Phi_B$, and $CFK^-(g_r)$ being the mapping cone of $\Phi_A$. \text{Lemma \ref{g_s isom to Z}, \ref{composition to U multiplication}} and their proofs work without change.

	We now have commutative squares for $c$, a positive or a negative crossing, respectively:
	\begin{equation}\label{filtration diagram}
		\xymatrix{Z \ar[d]_{U_c+U_d-U_a-U_b} \ar[r]^{\Phi_B} &Y \ar[d]^{\Phi_A}\\
			Z \ar[r]^{=} &Z\\}
		\xymatrix{Z \ar[d]_{\Phi_A} \ar[r]^{=} &Z \ar[d]^{U_c+U_d-U_a-U_b}\\
			Y \ar[r]^{\Phi_B} &Z\\}			
	\end{equation}

	\begin{proof}(of \text{Theorem \ref{skein exact sequence}}) 
		When $c$ is a positive crossing, we have the first diagram above in hand, which is quasi-isomorphic to $\mathrm{Cone}(\Phi_B)$, thus, also to $CFK^-(g)$. This can be seen by considering the filtration by the vertical coordinate, where the complex is an extension of a subcomplex with trivial homology by a quotient complex quasi-isomorphic to $\mathrm{Cone}(\Phi_B)$. On the other hand, using the horizontal filtration, we can regard it as an extension of a subcomplex isomorphic to $\mathrm{Cone}(\Phi_A)$ by $\mathrm{Cone} (U_a+U_b-U_c-U_d)$. Recall that \text{Lemma \ref{g_s isom to Z}} said that $Z$ is isomorphic to the chain complex $CFK^-(g_s)$. Putting things together, we have a long exact sequence :
		\[\ldots\longrightarrow HFK^-(g)\longrightarrow H_*(\frac{CFK^-(g_s)}{U_a+U_b-U_c-U_d}) \longrightarrow HFK^-(g_r)\longrightarrow HFK^-(g)\longrightarrow \ldots,\]
		when $c$ is positive.
		
		Arguing using the second diagram, we see that  \[\ldots\longrightarrow HFK^-(g)\longrightarrow HFK^-(g_r) \longrightarrow H_*(\frac{CFK^-(g_s)}{U_a+U_b-U_c-U_d})\longrightarrow HFK^-(L)\longrightarrow\ldots\]
		is exact when $c$ is negative.
		
		Since $g_r$, $g$, and$g_s$ are diagrams for $L_r$, $L$, and $L_s$ respectively, we have the desired skein exact sequences.
	\end{proof}
	\begin{rmk}
		Our proof mimics the one in \cite[Section 6]{MR2574747}, whose argument in turn follows from the proof using planar diagrams given in Section 4 of the same article. However, under the combinatorial setup, things are actually easier since now $\Phi_A\circ \Phi_B=U_c+U_d-U_a-U_b$, that is, there is no need for a chain homotopy to relate them. 
	\end{rmk}
	
	\subsection{Remarks on gradings}\label{Remarks on gradings}
	We want our skein relation to preserve the relative bigrading on homology groups. Since our definition of grading is somewhat complicated, we need a generalization of \text{Lemma \ref{relative grading through funtion n}}.
	
	To show that the skein exact sequence we introduced in the previous section is bigrading-preserving, we need to find out how gradings change under the four maps we considered above: $Q$, $\Phi_A$, $\Phi_B$, and multiplication by $U_c+U_d-U_a-U_b$. We will fix the convention that when we say a map is of bi-degree $(m,a)$, then it increases Maslov grading by $m$ and increases Alexander grading by $a$. The last one is clearly of bi-degree $(-2,-1)$, so it suffices to consider the first three. We assume here that c is positive so that $Z$ is a subcomplex of $CFK^-(g_r)$, while $Y$ is a subcomplex of $CFK^-(g)$. Negative case can be treated in exactly the same way. In this case, we have \[CFK^-(g_r)\cong \mathrm{Cone}(\Phi_A),\text{ } CFK^-(g)\cong \mathrm{Cone}(\Phi_B). \]
	
	Recall that in the construction of $Q$, we have an explicit correspondence between generators and rectangles. The proof of \text{Lemma \ref{g_s isom to Z}} shows that if $\xv,\yv$ in $\mathcal{S}(g_s)$ can be connected by a rectangle without the distinguished XX marked point, then $\xv', \yv'$ in $\mathcal{P}$ can be connected by a rectangle that changes the bigrading in exactly the same way. (This is obvious when the background manifold is $S^3$. When it is some $L(p,q)$, the rectangle lifts to $p$ copies on $S^3$. Using the additivity of all terms in the definition of gradings and the covering formula, one sees that the result holds in general.) This shows that $Q$ is a bigraded chain map between relatively bigraded groups.
	
	The identification of grid homology with the usual Heegaard Floer homology has been done in \cite{MR2372850}. One can also observe that the grid version of singular knot Floer defined in \cite{MR2574747} is the same as ours. These, together with Lipshitz's formula for the Maslov index and \text{Lemma \ref{relative grading through funtion n}} show that for $\xv,\yv \in \mathcal{S}$ and a rectangle $r$ connecting them, we have \[M(\xv,\yv)=1+2n_{\XXV\XXV}(r)-2n_{\OV}(r)+ n_{\xv}(\mathrm{int}(r))+ n_{\yv}(\mathrm{int}(r)),\]
	\[A(\xv,\yv)= 2n_{\XXV\XXV}(r)+n_{\XXV_r}(r)-n_{\OV}(r).\]
	
	The notation here follows \text{Subsection \ref{gradings}}.
	
	Also note that for $g$, a grid diagram for a link in $L(p,q)$, if $\widetilde{g}$ is its lift to $S^3$, then a rectangle embedded in $g$ gives rise to $p$ copies of rectangles in $\widetilde{g}$, each containing the same number of each kind of marks. Thus, regarding $Z$, $Y$ as sub/quotient complexes of $CFK^-(g_r)$, $\Phi_A$ has bi-degree $(-1,0)$, while $\Phi_B$ has bi-degree $(-1,-1)$. On the other hand, when regarding $Z$, $Y$ as quotient/sub complexes of $CFK^-(g)$, $\Phi_A$ has bi-degree $(-1,-1)$, while $\Phi_B$ has bi-degree $(-1,0)$. The differences come from the position change of $X$ base points.  
	
	Now we have shown that all the maps we used in the construction above are of a certain bidegree. Assume $c$ is positive. A closer look at the exact sequence shows that in \[\longrightarrow HFK^-(g)\longrightarrow H_*(\frac{CFK^-(g_s)}{U_c+U_d-U_a-U_b}) \longrightarrow HFK^-(g_r)\longrightarrow HFK^-(g)\longrightarrow,\]
	the second arrow is induced by a quotient of chain complex, and the third arrow is induced by the composition of an inclusion and a quotient, so both of them are grading preserving. Moreover, the fourth arrow comes from $\Phi_A:Y\to Z$ connecting a quotient complex to a subcomplex that is of bidegree $(-1,0)$ as we saw above.
	
	A similar analysis shows that when $c$ is negative, the only grading shift $(-1,0)$ happens at the arrow $HFK^-(g) \to HFK^-(g_r)$. 
	
	There is also a more formal way of seeing the change in Maslov grading, using the fact that $M$ acts as the homological grading in our theories. For any pair of graded chain maps $f:A\to B$, $g:B\to C$, we have an exact triangle \[\mathrm{Cone}(f)\longrightarrow \mathrm{Cone}(g\circ f) \longrightarrow \mathrm{Cone}(g).\] In the induced long exact sequence, a grading change happens at the connecting morphism $H_*(\mathrm{Cone}(g))\to H_*(\mathrm{Cone}(f))$. Apply to our case, we have $\frac{CFK^-(g_s)}{U_c+U_d-U_a-U_b}$ being the mapping cone of $U_c+U_d-U_a-U_b=\Phi_A \circ \Phi_B: Z\to Z$, thus, we have the degree shifting property as above.

	\begin{rmk}
		Recall that	on a grid diagram, two generators $\xv$, $\yv$ belong to the same $\mathrm{Spin^c}$ class iff they can be connected by a sequence of rectangles in $g$.
		Since $Q$ is defined via a direct identification of generators and differential, while $\Phi$-maps are defined by counting rectangles, they all preserve $\mathrm{Spin^c}$ classes of generators. Therefore, our exact sequence actually splits into $p$ distinct exact sequences, one for each $\mathrm{Spin^c}$ class. For more details on $\mathrm{Spin^c}$ structure one can refer to \cite{MR2113019}. 
	\end{rmk}
	
	\subsection{Resolution cube and spectral sequence}\label{Resolution cube and spectral sequence}
	
	We shall prove an analogue of \cite[Theorem 4.4]{MR2574747} or more precisely, \cite[Theorem 1.2]{MR3229041}, since we will consider an untwisted resolution cube over the usual base ring $\FF[U_1,\ldots,U_{2n+m}]$. As remarked at the beginning of \cite[Section 2]{MR3229041}, they used the twisted coefficient originally in \cite{MR2574747} to obtain finiteness on induced maps and to deal with extra components appearing in the smoothings. In detail,
	\begin{enumerate}
		\item we need all the rectangle counting maps occurring in higher differentials to have a finite coefficient for each generator;
		\item we need to deal with the change of base ring when one component of $L$ splits into two during the smoothing of a crossing.
	\end{enumerate}

	For (1), they used the admissibility of diagrams to solve the problem in \cite{MR3229041}. On a grid diagram, there is no non-trivial periodic domain, i.e. any nonempty domain with boundary a sum of $\alpha$ and $\beta$ curves must contain some $O$ or $X$($XX$) marked point, so admissibility holds naturally. 
	
	They solved (2) by an operation called insertion, i.e., introducing a new pair of $\alpha$ and $\beta$ curves together with a new pair of $O$, $X$ base points. Using this, they showed how to deal with the base ring change by the lemma below. Its proof turned out to rely on \cite[Lemma 6.1]{MR2443092}, which does not require the background manifold to be the three-sphere, thus, it holds in our case. Note that on a special grid diagram with the standard picture at each crossing, we already have sufficiently many base points and variables, so no insertion is needed. One can compare our \text{Figure \ref{fig:standard picture at each crossing}, \ref{fig:grid diagram of resolution and singularization}} with \cite[Figure 6]{MR2443092}.

	\begin{lem}(\cite[Lemma 2.2]{MR3229041})
		Let $L$ be a possibly singular link in lens space $L(p,q)$ with components $L_1,\ldots, L_l$. For each $L_i$, we have a chain complex $CFK^-(L_i)$ over some polynomial ring $R_i$. Some suitable choices of a special grid diagram for each of them can be fit into one of $L$, so we have $CFK^-(L)$ over \[R= R_1\otimes_{\FF}\ldots \otimes_{\FF} R_l.\] Then we have a quasi-isomorphism of R-modules : \[CFK^-(L)\simeq CFK^-(L_1)\otimes_{\FF}\cdot\cdot\cdot \otimes_{\FF} CFK^-(L_l)\otimes_{\FF} H_*(T^{l-1}).\]
	\end{lem}
	
	With these two problems solved, we can use the idea from  \cite{MR2574747} to construct an untwisted resolution cube for links in lens spaces.
	
	To make a precise statement, we need some terminology:
	\begin{itemize}
		\item $R$ will denote the ring $\FF[U_1,\ldots, U_{2n+k}]$, with each $U_i$ corresponding to an $O$-mark $O_i$. (Here, $n$ is the number of $XX$ marks, $k$ is the number of $X$ marks, thus $g$ is of size $n+k$.)
		\item To each crossing $c\in C(g)$, we denote by $O_a^{(c)}$,  $O_b^{(c)}$ the $O$ marks that is pointed to by the two corner $X$s, while denote by $O_c^{(c)}$, $O_d^{(c)}$ the two $O$ marks that point to the pair of corner $X$'s. We introduce similar notation for $XX\in X(g)$. See \text{Figure \ref{fig:AB diagram}} for an illustration. Further denote by $u^{(c)}$ the term $U^{(c)}_a+U^{(c)}_b-U^{(c)}_c-U^{(c)}_d\in R$.
		
		\item For a crossing $c\in C(g)$, let $g_{c,s}$ or $g_{c,r}$ denote the singularized or smoothed diagram at this single crossing $c$. 
		
		\item We define \[\mathcal{U}_c:CFK^-(g_{c,s})\otimes_{R} R/(u^{(c)})\to CFK^-(g_{c,r}), \]
		\[\mathcal{Z}_c:CFK^-(g_{c,r})\to CFK^-(g_{c,s}) \otimes_{R} R/(u^{(c)}) \] to be the unzip or zip homomorphism induced by the horizontal arrows in the left and right diagrams in \text{Equation \ref{filtration diagram}}. 
	\end{itemize}
	
	The resolution cube of $g$ is the union $\bigcup_{I:C(g)\to \{0,1\}} g_I$. Using this, we form a relatively bi-graded R module $V(g)=\bigoplus_{I} V(g_I)$, where
	\[ V(g_I)=H_*(CFK^-(g_I)\otimes_{R}(\bigotimes_{XX\in X(g_I)} R/(u^{(XX)}))).\] 
	
	\begin{thm}\label{spectral sequence}(\cite[Theorem 4.4]{MR2574747})
		Let $g$ be a special grid diagram for some possibly singular link $L$ in $L(p,q)$. Assume at each crossing, $g$ looks like the standard picture in \text{Figure \ref{fig:standard picture at each crossing}}. Then there is a spectral sequence converging to $HFK^-(g)$ whose $E_1$ term is $V(g)$ with the differential $d_1$ induced by the zip and unzip homomorphisms.
	\end{thm}
	
	The proof of this theorem follows from the discussion in \cite[Section 4]{MR2574747}. There are two important observations: \begin{enumerate}
		\item Their existing proof relies only on the well-behaved diagram and chain complexes but not on the background manifold;
		\item As they mentioned in Section 6, when using a special grid diagram, all higher homotopies vanish. This has been seen during the proof of \text{Lemma \ref{composition to U multiplication}}: Our $\Phi_A$ and $\Phi_B$ compose directly to multiplication by $-u^{(c)}$, with no need of the map $\Phi_{AB}$ in their terminology. This will also simplify higher differentials dramatically.
	\end{enumerate}
	
	For each $c\in C(g)$, we denote by $p_c$, $A_c$ and $B_c$ the distinguished intersection point and marked points as shown in \text{Figure \ref{fig:AB diagram}}. We say $\xv\in \mathcal{S}(g)$ is \begin{itemize}
		\item of type $Z_c$ if \begin{itemize}
			\item $c$ is positive and $p_c\in \xv$;
			\item $c$ is negative and $p_c\notin\xv$.
		\end{itemize}
		\item  of type $Y_c$ if \begin{itemize}
			\item $c$ is positive and $p_c\notin \xv$;
			\item $c$ is negative and $p_c\in\xv$.
		\end{itemize}
	\end{itemize}
	
	We say a rectangle $r$ is \begin{itemize}
		\item of type $A_c$ if  \begin{itemize}
			\item$ r\in \mathrm{Rect}^\circ (\xv,\yv)$, one of $\xv$, $\yv$ is of type $Z_c$, the other is of type $Y_c$;
			\item $n_{\mathbb{A}_c}(r)=1$ and $n_{\mathbb{B}_c}(r)=0$.
		\end{itemize}
		\item of type $B_c$ if  \begin{itemize}
			\item$ r\in \mathrm{Rect}^\circ (\xv,\yv)$, one of $\xv$, $\yv$ is of type $Z_c$, the other is of type $Y_c$;
			\item $n_{\mathbb{B}_c}(r)=1$ and $n_{\mathbb{A}_c}(r)=0$.
		\end{itemize}
		\item of type $AB_c$ if  \begin{itemize}
			\item$ r\in \mathrm{Rect}^\circ (\xv,\yv)$, and both $\xv$, $\yv$ are of type $Z_c$;
			\item $n_{\mathbb{A}_c}(r)=1$ and $n_{\mathbb{B}_c}(r)=1$.
		\end{itemize}
	\end{itemize}
	
	Let $C_p$ denote the set of positive crossings $g$, and let $C_n$ denote the set of negative ones. For disjoint sets $I$, $J$, and $K\subset C(g)$, let $Z_I=\cap_{c\in I} Z_c$, $Y_I=\cap_{c\in J} Y_c$, $Z_K=\cap_{c\in K} Z_c$. 
	\begin{defi}\label{Phi IJK}
		
		Define $\Phi_{I,J,K}: Z_I\cap Y_J\cap Z_K\to Y_I\cap Z_J\cap Z_K$ by counting empty rectangles that are \begin{itemize}
			\item of type $A_c$ for $c\in (I\cap C_p) \cup(J\cap C_n)$;
			\item of type $B_c$ for $c\in (I\cap C_n)\cup (J\cap C_p)$;
			\item of type $AB_c$ for $c\in K$;
			\item multiplicities at all other $X$ or $XX$ base points are zero.
		\end{itemize}
		
	\end{defi} 
	
	\begin{rmk}\label{rmk on phi ijk}
		From items in \text{Definition \ref{special grid diagram}}, it follows immediately that $\Phi_{I,J,K}$ is zero whenever $K\ne \emptyset$ or $\vert I\cup J\vert>1$. When $I=J=K=\emptyset$, $\Phi_{I,J,K}$ is just the usual differential.
	\end{rmk}

	Thus, in our case, \cite[Lemma 4.5]{MR2574747} holds without the need to prove it. \begin{lem}\label{composition of Phi}(\cite[Lemma 4.5]{MR2574747})
		Fix $I$, $J$, and $K$ disjoint as above. Then, \[\sum_{\substack{\{I_1,I_2,J_1,J_2,K_1,K_2| I_1\sqcup I_2\sqcup K_1\sqcup K_2 =I\cup K, \\ J_1\sqcup J_2 \sqcup K_1\sqcup K_2=J\cup K\}}} \Phi_{I_1,J_1,K_1}\circ\Phi_{I_2,J_2,K_2}=\]
		\[\begin{cases}
			-u^{(c)}, &\text{if K=\{c\} and }I=J=\emptyset \\
			0, &\text{otherwise}
		\end{cases}.\]
		
	\end{lem}
	
	Now consider the graded group \[C=\bigoplus_{I\sqcup J\sqcup K\sqcup L= C(g)} Z_I\cap Y_J\cap Z_K\cap Z_L, \] 
	the sum ranges over all partitions of $C(g)$ into disjoint subsets. Define a relation on the set of partitions by \[(I_1,J_1,K_1,L_1)<(I_2,J_2,K_2,L_2)\] if $I_2\subset I_1$, $J_2\subset J_1\cup I_1$, $K_2\subset I_1\cup K_1$, and $\vert K_1 \Delta K_2\vert \le 1$. Here, we use $\Delta$ to denote the difference between two sets, i.e., $K_1 \Delta K_2=(K_2-K_1)\cup (K_1-K_2)$.
	\begin{defi}\label{higher differential}
		For such a pair of partitions, define a map \[D_{(I_1,J_1,K_1,L_1)<(I_2,J_2,K_2,L_2)}: Z_{I_1}\cap Y_{J_1}\cap Z_{K_1}\cap Z_{L_1}\to Z_{I_2}\cap Y_{J_2}\cap Z_{K_2}\cap Z_{L_2}\]
		\[\begin{cases}
			\Phi_{I_1\cap J_2, L_2\cap J_1, L_2\cap I_1}, & \text{if } K_1=K_2;\\
			1 ,& \text{if } K_2=K_1\cup\{n\}, J_1=J_2, L_2=L_1, \text{ $n$ is a negative crossing} ;\\
			u^{(n)} ,& \text{if } K_2=K_1-\{n\}, J_1=J_2, I_1=I_2, \text{ $n$ is a negative crossing};\\
			u^{(p)} ,& \text{if } K_2=K_1\cup\{p\}, J_1=J_2, L_2=L_1, \text{ $p$ is a positive crossing}; \\
			1, & \text{if } K_2=K_1-\{p\}, J_1=J_2, I_1=I_2, \text{ $p$ is a positive crossing}. 
		\end{cases}\]
	\end{defi}
	
	Endow $C$ with the endomorphism \[D=\sum_{(I_1,J_1,K_1,L_1)<(I_2,J_2,K_2,L_2)} D_{(I_1,J_1,K_1,L_1)<(I_2,J_2,K_2,L_2)}.\]
	
	\begin{lem}\label{identify C with CFK}(\cite[Lemma 4.6]{MR2574747})
		$(C,D)$ is a chain complex quasi-isomorphic to $CFK^-(g)$.
	\end{lem}
	\begin{proof}
		The fact that $D^2=0$ follows immediately from \text{Remark \ref{rmk on phi ijk}} and \text{Lemma \ref{composition of Phi}}.  For each $(I_1,J_1,K_1,L_1)$, we consider the $(I_2,J_2,K_2,L_2)$ component of $D^2|_{(I_1,J_1,K_1,L_1)}$.  When $K_1=K_2$, we have a square exactly the same as in \text{Equation \ref{filtration diagram}}, so \text{Lemma \ref{composition of Phi}} leads to the desired result. When $\vert K_1\Delta K_2\vert=1$ or $2$, we have a square with the same map on the left and right-hand sides, and also the same map on the top and bottom edges. Since we are working over $\FF=\ZZ/2\ZZ$, the result follows. 
		
		The quasi-isomorphism can be seen by a sub-quotient argument. In $C$, there is a subcomplex with trivial homology formed by those partitions with some positive crossings in $K$ or $L$. The resulting quotient has a subcomplex consisting of those partitions with each negative crossing in $J\cup L$. This subcomplex can be identified with $CFK^-(g)$ while the resulting quotient has trivial homology. The identification can be done as follows: This sub-quotient complex has representative components of 
		$C$ with all positive crossings in $I\cup J$ and all negative crossings in $J\cup L$. In all these partitions, $K=\emptyset$, so each generator in $\mathcal{S}(g)$ appears exactly once in it. For the identification of differential, it follows from \text{Remark \ref{rmk on phi ijk}} that $\Phi_{I,J,K}$ goes back to differential when all the three sets are empty, and the observation \begin{itemize}
			\item when $I=\{p\}$ consists of s single crossing, a rectangle from  $\xv\in Z_{p}$ to $\yv\in Y_{p}$ must contain one of $B_p$ instead of $A_p$;
			\item while, when $J=\{n\}$ consists of a single crossing, a rectangle from  $\xv\in Y_{p}$ to $\yv\in Z_{p}$ must contain one of $B_p$ instead of $A_p$.
		\end{itemize} These together show that all those nonzero components of $D$ come from $\partial^-_{CFK^-(g)}$. Thus, we have proved our claim. 
		
	\end{proof}
	
	\begin{proof}(of \text{Theorem \ref{spectral sequence}})
		$(C,D)$ is actually a ``partition-filtered'' description of the $\vert C(g)\vert $-dimensional resolution cube, with all complete resolutions of $C(g)$ as an index set. Collapsing all vertical filtrations, we obtain \[\bigoplus_{I:C(g)\to \{0,1\}} CFK^-(g_I)\otimes_{R}(\bigotimes_{XX\in X(g_I)} R/(u^{(XX)})), \] with $\mathcal{U}_c$ and $\mathcal{Z}_c$ as edge homomorphisms. This and the quasi-isomorphism provided by \text{Lemma \ref{identify C with CFK}} give rise to the wanted spectral sequence.

	\end{proof}
	
	\section{An oriented version}\label{An oriented version}
	This section aims at extending the base ring of our grid homology from $\ZZ/2\ZZ$ to $\ZZ$. Our construction follows \cite[Chapter 15]{MR3381987}, \cite[Section 3]{Celoria_2023} and \cite[Section 7]{tripp2021gridhomologylensspace}. The oriented resolution cube using braid diagram has also been used in \cite{MR4777638} and \cite{Beliakova2022APO}. The strategy will be assigning a sign to each rectangle embedded in the grid diagram, then considering the sign-twisted differential and homology.
	
	\subsection{Sign assignment}
	Fix a grid diagram $g$ for some link $L\subset L(p,q)$. For any pair of $\xv, \yv \in \mathcal{S}(g)$, we can consider the set $\mathrm{Rect} (\xv,\yv)$ of rectangles from $\xv$ to $\yv$, which is either empty or consists of two elements. Ranging over all such pairs, we form a union \[\mathrm{Rect}(g)=\bigcup_{\xv,\yv\in \mathcal{S}(g)} \mathrm{Rect} (\xv,\yv).\] This will be the domain of our sign assignment.
	
	\begin{defi}\label{sign assignment}
		On a grid diagram $g$, a \emph{sign assignment} is a function \[\boldsymbol{S}: \mathrm{Rect}(g)\to \{\pm 1\}\] that satisfies \begin{enumerate}
			\item If $(r_1,r_2)$ and $(r'_1,r'_2)$ form an alternative pair, i.e. $r_1\in \mathrm{Rect}(\xv,\wv)$, $r_2\in Rect (\wv,\zv)$, $r'_1\in \mathrm{Rect}(\xv,\yv)$, $r'_2\in Rect (\yv,\zv)$ with $\yv\ne \wv$ and $r_1\ast r_2= r'_1 \ast r'_2$, then $\Sbold(r_1)\Sbold(r_2)=-\Sbold(r'_1)\Sbold(r'_2)$;
			\item If $r_1\ast r_2$ is a horizontal annulus, then $\Sbold(r_1)\Sbold(r_2)=1$;
			\item If $r_1\ast r_2$ is a vertical annulus, then $\Sbold(r_1)\Sbold(r_2)=-1$.
		\end{enumerate}
	\end{defi}
	In \cite{Celoria_2023}, the author constructed sign assignments for grid diagrams of links in lens spaces using Spin central extension of the symmetric group (following \cite{Gallais2007SignRF} and \cite{MR3381987}) and the fact that the set of generators $\mathcal{S}(g)$ can be identified with $S_n\times \ZZ/p\ZZ$ in a natural way.
	
	By definition, a sign assignment on a grid diagram has nothing to do with the choices of $O$, $X$ base points, so the existence and uniqueness shown in \cite{Celoria_2023} extend directly to our case. Moreover, when fixing a grid diagram $g$, an explicit way of constructing an isomorphism between chain complexes with different sign assignments is also given in \cite{Celoria_2023}. It is again base point-independent, so we have the following: 
	
	\begin{thm}\label{existence of oriented theory}(\cite[Theorem 1.1]{Celoria_2023})
		A sign assignment exists for any grid diagram representing a possibly singular link $L$ in $L(p,q)$ (actually unique up to a certain kind of gauge transformations). Moreover, for a fixed grid diagram, the sign-refined grid homology is independent of the choice of a sign assignment.
	\end{thm}
	
	With a chosen sign assignment $\Sbold$ on $g$, we now upgrade our chain group by replacing $\FF$ with $\ZZ$ in \text{Definition \ref{definition of chain complex}} and redefining the differential to be \[\partial_{\Sbold}^-\xv=\sum_{\yv\in \boldsymbol{S}} \sum_{\substack{r\in \mathrm{Rect}^\circ(\xv,\yv)\\n_{\XXV}(r)=0}} \Sbold(r) \prod_{i=1}^{2n+k}  U_{i}^{n_{O_i}(r)} \yv;\]
	
	\[\widetilde{\partial}_{\Sbold} \xv=\sum_{\yv\in \boldsymbol{S}} \sum_{\substack{r\in \mathrm{Rect}^\circ(\xv,\yv)\\n_{\XXV}(r)=0, n_{\OV}(r)=0}} \Sbold(r) \yv.\]
	
	Using axioms in \text{Definition \ref{sign assignment}} one sees directly that \[(\partial_{\Sbold}^-)^2=0\text{ and } (\widetilde{\partial}_{\Sbold})^2=0.\]
	
	Indeed, when $\xv\ne \zv$, we have seen that each domain in $\pi_2(\xv,\zv)$ has either zero or two ways to decompose into empty rectangles, and property (1) of $\Sbold$ shows that when there are two ways, they cancel with each other. When $\xv=\zv$, only annuli connect $\xv$ to itself. However, these do not contribute to the differential since there is an $X$ or $XX$ base point in each horizontal or vertical annulus.
	
	Now, we have shown that there are well-defined homology theories $HFK^-(g;\ZZ)$ along with its two counterparts. Using the remark in \text{Theorem \ref{existence of oriented theory}}, we shall omit the sign assignment from our notation. 
	
	\subsection{Invariance of oriented theory}
	Discussion in the previous section shows that oriented homology theories are well-defined for each individual grid diagram. The aim of this section is to modify the proof in \text{Subsection \ref{Proof of invariance}} to show that these oriented theories won't change under grid moves, thus $HFK^\circ(L;\ZZ)$ are well-defined link invariants.
	
	Invariance of oriented theories for regular links in $S^3$ and lens spaces has been verified in \cite{MR3381987} and \cite{tripp2021gridhomologylensspace}, respectively. Note that the appearance of $XX$ singular points only affects the grading of maps but has no influence on identities appearing in the proof to hold. We have proved in \text{Subsection \ref{Proof of invariance}} that the gradings in our case behave correctly during grid moves. So, it is straight forward to adapt their proof to our situation. 
	
	The basic idea behind their proof is to replace the various tools used in the proof with corresponding signed versions. The main tasks in the following are \begin{itemize}
		\item relating sign assignments on different grid diagrams;
		\item defining signed version of all chain maps used in the proof.
	\end{itemize}
	
	\subsubsection{Commutation invariance revisit}
	To adapt the previous proof of commutation invariance to the newly signed version, it is necessary to extend the definition of $\Sbold$ to pentagons and hexagons.  
	Recall that we have relative grading formulas given by counting base points in rectangles: 
	\[M(\xv,\yv)=1+2n_{\XXV\XXV}(r)-2n_{\OV}(r)+ n_{\xv}(\mathrm{int}(r))+ n_{\yv}(\mathrm{int}(r)),\]
	\[A(\xv,\yv)= 2n_{\XXV\XXV}(r)+n_{\XXV_r}(r)-n_{\OV}(r).\] 
	
	On a grid diagram, the $\mathrm{Spin^c}$ class of a generator is characterized by the criterion that $\xv$ and $\yv$ belong to the same class if and only if they can be connected by a sequence of rectangles. Then the formula above implies that we can lift the Maslov grading in each $\mathrm{Spin^c}$ class to a function $M_{\sfrak}: \mathcal{S}_{\sfrak} \to \ZZ$. Here we use $\mathcal{S}_{\sfrak} $ to denote the set of generators in the $\mathrm{Spin^c}$ structure $\sfrak$. This allows us to use $M(\xv)$ as exponent to control the sign. This is also necessary because examples computed in \cite{tripp2021gridhomologylensspace} show that the absolute grading which is canonical in the sense that it can be identified with classical Heegaard Floer homology in the case of regular links is not integral valued in general. 
	
	Unfortunately, in \cite{tripp2021gridhomologylensspace}, they directly used $M(\xv)$, defined in \text{Equation \ref{canonical lift of M}} as an exponent of $(-1)$ to characterize the signs of pentagons, which is a little weird. However, since the only property needed in their proof is when initial states of two pentagons are connected by a rectangle without $X$ mark and of the same side (See the definition of $\Sbold_{\mathrm{Pent}}$ below), then a sign (-1) appears in $\Sbold$. So things get settled when we fix an integral lift of $M$ for each $\mathrm{Spin^c}$ class. We still denote by $M$ the integral lift chosen class by class.
	
	Assume $g$ and $g'$ are as described in \text{Subsubsection \ref{Commutation invariance}}. Since a sign assignment has nothing to do with markings, we may fix a $\Sbold$ that works for both diagrams. 
	
	Recall that we have a closest point map $\xv\mapsto\xv'$ which is a bijection $\mathcal{S}(g)\to\mathcal{S}(g')$, and on the combined diagram, we have a canonical small triangle $t_{\xv}$ connecting each pair $(\xv,\xv')$. (See \text{Figure \ref{fig:commutation}}.) When $p\in \mathrm{Pent}(\xv,\yv')$, we have $t_{\xv}\ast p \in \mathrm{Rect}^\circ(\xv',\yv')$. Let $R(p)$ denote the rectangle associated to $p$ in this way. A pentagon is said to be left (resp. right) if it lies to the left (resp. right) of the intersection point a on $\beta_i\cap \gamma_i$. Further, let $\mathrm{Pent}(g,g')$ be the set of all pentagons from a generator in $\mathcal{S}(g)$ to some generator in $\mathcal{S}(g')$ in the combined diagram. With all these terminologies in hand, we define \[\Sbold_{\mathrm{Pent}}: \mathrm{Pent}(g,g')\to \{\pm 1\},\] for $p\in Pent (\xv,\yv')$ \[\Sbold_{Pent}(p)=\begin{cases}
		(-1)^{M(\xv)+1}\Sbold(R(p)), & \text{if p is a left pentagon;} \\
		(-1)^{M(\xv)}\Sbold(R(p)), & \text{if  p is a right pentagon.}
	\end{cases}\]
	
	Recall that Maslov gradings on $\mathcal{S}(g)$ and $\mathcal{S}(g')$ are identified so the fixed lift works for both. We can introduce sign assignment $\Sbold_{\mathrm{Pent}}:\mathrm{Pent}(g',g)\to \{\pm 1\}$ in exactly the same way. In what follows, we shall omit the subscript ``Pent''.
	
	Now, we renew the pentagon counting maps to
	\[P_{\Sbold}(\xv)=\sum_{\yv'\in \boldsymbol{S}(g')}\sum_{\substack{p\in \mathrm{Pent}^\circ(\xv,\yv')\\p\cap \XXV=\emptyset}} \Sbold(p) \prod_{i=1}^{2n+k}U_i^{n_{O_i}(p)}\yv'.\] 
	\[P'_{\Sbold}(\xv')=\sum_{\yv\in \boldsymbol{S}(g)}\sum_{\substack{p\in \mathrm{Pent}^\circ(\xv',\yv)\\p\cap \XXV=\emptyset}} \Sbold(p) \prod_{i=1}^{2n+k}U_i^{n_{O_i}(p)}\yv.\] 
	
	 \cite[Lemma 7.2]{tripp2021gridhomologylensspace} shows that they are both chain maps using the definition of $\Sbold$. 
	
	Next, we further extend $\Sbold$ to $\mathrm{Hex}(g)$, the set of all hexagons connecting generators in $\mathcal{S}(g)$ on the combined grid diagram. Recall that each $h\in \mathrm{Hex}^\circ (\xv,\yv)$ has $ h \ast B \in \mathrm{Rect}^\circ(\xv,\yv)$, where $B$ is one of the bigons shown in \text{Figure \ref{fig:commutation}}. Define $R(h)=h\ast B$ and $\Sbold(h)=\Sbold(R(h))$. Renew $H$ to \[H_{\Sbold}(\xv)= \sum_{\yv\in \boldsymbol{S}(g)} \sum_{\substack{h\in \mathrm{Hex}^\circ(\xv,\yv)\\h\cap \XXV=\emptyset}} \Sbold(h) \prod_{i=1}^{2n+k} U_{i}^{n_{O_i}(h)} \yv.\]
	
	Now, \cite[Lemma 7.5]{tripp2021gridhomologylensspace} as well as  \cite[Lemma 15.3.4]{MR3381987} show that \[\partial^- \circ H_{\Sbold} +H_{\Sbold} \circ \partial^- +P'_{\Sbold} \circ P_{\Sbold} +\mathrm{Id}_{CFK^-(g;\ZZ)}=0,\] that is, $P_{\Sbold}$ and $-P'_{\Sbold}$ are homotopy inverses to each other. 
	
	With the help of this, we know that $CFK^-(g;\ZZ)$ and $CFK^-(g';\ZZ)$ are quasi-isomorphic, which is exactly what we want.
	
	\subsubsection{Stabilization invariance revisit}
	
	Let $g$ and $g'$ be the pair considered in \text{Subsubsection \ref{Stabilization invariance}}. Recall that all the maps used there count rectangles in $g'$, so things become transparent when we relate signs in $g$ with those in $g'$. This can be done using $e: \BI \to CFK^-(g)[U_{\text{new}}][1,1]$, which is defined via bijection not only on the set of generators but also on the set of empty rectangles. Via this correspondence, a sign assignment $\Sbold:\mathrm{Rect} (g')\to \{\pm 1\}$ gives rise to one on $\mathrm{Rect}(g)$. Using this fact, a lift $M_{\sfrak}$ for each $\mathrm{Spin^c}$ class of generators on $g'$ gives rise to one on $g$ as well. 
	
	Using the chosen $\Sbold$, we can enhance each of $H_{O_{\text{new}}}$, $H_{X_{\text{new}}}$, $H_{X_{\text{new}},O_{\text{new}}}$ by adding $\Sbold(r)$ to each term in the sum. 
	
	As noted in \cite[Section 15.3.2]{MR3381987}, we still have a square
	\[\xymatrix{(\BI^-,\partial_{\BI}^{\BI}) \ar[d]^{e} \ar[r]^{\partial^{\BN}_{\BI}} &(\BN^-,\partial_{\BN}^{\BN}) \ar[d]^{e\circ H_{X_{\text{new}}}}\\
		(CFK^-(g)[U_{\text{new}}][1,1],-\partial^-_{g}) \ar[r]^{U_{\text{new}}-U_j} &(CFK^-(g)[U_{\text{new}}][1,1],\partial^-_{g})\\} .\]
	
	Over $\ZZ$, it is no longer commutative. Instead, each edge map anti-commutes with the differentials, and maps going right-down and down-right sum to zero. Alternatively, consider the map $\mathrm{stab}:CFK^-(g';\ZZ)\to \mathrm{Cone} (U_{\text{new}}-U_j)$ defined by 
	\[\mathrm{stab}(\xv)= \begin{cases}
		(-1)^{M(\xv)}e(\xv), & \xv \in \BI\\
		(-1)^{M(\xv)} e\circ H_{X_{\text{new}}}(\xv), & \xv \in \BN 
	\end{cases}\] on the set of generators and extend it linearly to the whole complex. The property above can be reinterpreted as $stab$ being a chain map. This follows from (1) in \text{Definition \ref{sign assignment}}. 
	
	Using (1)(2)(3) in \text{Definition \ref{sign assignment}}, one sees that \begin{itemize}
		\item \text{Proposition \ref{homotopic U_i action}} still holds over $\ZZ$, so the choice of $U_j$ on the same thin edge or $S^1$ component as $U_{\text{new}}$ is still irrelevant.
		\item 	$H_{X_{\text{new}}}\circ H_{O_{\text{new}}}+\mathrm{Id}_{\BI^-}=0$.
		\item  $H_{O_{\text{new}}}\circ H_{X_{\text{new}}}+ \partial^{\BN}_{\BN}\circ H_{X_{\text{new}},O_{\text{new}}}+ H_{X_{\text{new}},O_{\text{new}}}\circ \partial^{\BN}_{\BN}+\mathrm{Id}_{\BN^-}=0$
	\end{itemize} 
	
	Note that two vertical maps are quasi-isomorphisms of chain complexes thanks to the identities above. Now, the argument about mapping cone in \text{Subsubsection \ref{Stabilization invariance}} and \cite[Lemma 5.2.12]{MR3381987} finishes the proof.
	
	\subsection{Oriented resolution cube}
	
	In this subsection, we are aiming to extend results in \text{Section \ref{Resolution cube}} to an oriented version. Note that the resolution cube constructed in \cite{MR2574747} and \cite{MR3229041} are both over $\FF=\ZZ/2\ZZ$, but some indication of signed convention appeared in an arxiv version of \cite{MR2574747}. Recently, oriented resolution cubes for knots in $S^3$ appeared in \cite{Beliakova2022APO} and \cite{MR4777638}. We now try their methods in our situation.
	
	As in \text{Subsection \ref{Resolution cube and spectral sequence}}, we fix a special grid diagram $g$ for some link $L$ in $L(p,q)$ with standard picture as in \text{Figure \ref{fig:standard picture at each crossing}} at each crossing. We assume the size of $g$ is $n+k$ as usual, where $n$ is the number of $XX$ base points while $k$ is the number of $X$ base points. All the notations below will be the same as in \text{Section \ref{Resolution cube}}. By assumption on $g$, we have for each $c\in C(g)$, diagrams of $g_{c,s}$ and $g_{c,r}$ obtained from local modification. For a given complete resolution $I$, by applying the procedure shown in \text{Figure \ref{fig:resolution of crossings}} to each crossing, we obtain a compatible diagram $g_I$.
	
	As we did when proving the invariance of the oriented theory, once we choose a sign assignment for $g$, it gives rise to ones on the diagrams $g_{c,s}$ and $g_{c,r}$ for any $c\in C(g)$, as well as one on $g_I$ for any complete resolution of $g$. Since in a standard picture we only modify the picture locally at each crossing, it suffices to show the existence of induced sign assignment on $g_{c,s}$ and $g_{c,r}$ for an arbitrary $c\in C(g)$. Recall that $g_{c,r}$ is obtained from $g$ by changing the position of a pair of $X$ base points; in particular, they share the same grid. Using the fact that a sign assignment has nothing to do with the set of $O$, $X$ base points, one on $g$ naturally gives rise to one on $g_{c,r}$. On the other hand, $g_{c,s}$ has grid states and rectangles in canonical one-to-one correspondence with a subcomplex of $CFK^-(g)$ (or $CFK^-(g_{c,r})$, depending on the sign of $c$) (See the proof of \text{Lemma \ref{g_s isom to Z}}). So there is a unique sign assignment on $g_{c,s}$ compatible with the given one on $g$.  
	
	As in the previous section, we fix an integral lift of Maslov grading in each $\mathrm{Spin^c}$ structure on the diagram $g$. This gives rise to a lift of $M$ for generators on all $g_{c,r}$, $g_{c,s}$, as well as $g_I$s. 
	
	With these preparations in hand, we can now upgrade maps and complexes in \text{Section \ref{Resolution cube}} to their counterparts over $\ZZ$. Since we used $Q$ from the proof of \text{Lemma \ref{g_s isom to Z}} to induce the sign assignment on $g_{c,s}$, it is now an isomorphism of chain complexes over $\ZZ[U_1,\ldots,U_{2n+k}]$. Now
	redefine maps $\Phi_A:Y\to Z$ and $\Phi_B: Z\to Y$ by 
	\[\Phi_A(\xv)=\sum_{\yv\in \mathcal{P}} \sum_{\substack{r\in \mathrm{Rect}^\circ(\xv,\yv)\\n_{\XXV_0}(r)=0, n_{\mathbb{A}}(r)=1}} \Sbold(r) \prod_{O_i\in \OV} U_i^{n_{O_i}(r)}\yv;\]

	\[\Phi_B(\xv)=\sum_{\yv\in \mathcal{C}} \sum_{\substack{r\in \mathrm{Rect}^\circ(\xv,\yv)\\n_{\XXV_0}(r)=0, n_{\mathbb{B}}(r)=1}} \Sbold(r) \prod_{O_i\in \OV} U_i^{n_{O_i}(r)}\yv.\]
	
	These are still chain maps, which is guaranteed by (1) in the definition of $\Sbold$.
	
	In the proof of \text{Lemma \ref{composition to U multiplication}}, we saw that only two thin horizontal annuli and two thin vertical annuli through the pairs of $A_c$, $B_c$ base points contribute to $\Phi_A\circ \Phi_B$ and $\Phi_B\circ\Phi_A$. By (2)(3) in \text{Definition \ref{sign assignment}}, two horizontal annuli contribute multiplication by $U_c+U_d$, while vertical annuli contribute multiplication by $-U_a-U_b$. This shows that \text{Lemma \ref{composition to U multiplication}} still holds, that is, $\Phi_A\circ \Phi_B$ and $\Phi_B\circ\Phi_A$ are both equal to multiplication by $U_c+U_d-U_a-U_b$.
	
	From this, we can deduce that diagrams in \text{Equation \ref{filtration diagram}} still commute. After obtaining these diagrams, the proof of \text{Theorem \ref{skein exact sequence}} in \text{Subsection \ref{Skein exact sequence}} is purely algebraic, so it now shows that the skein exact sequences in \text{Equation \ref{positive skein} and \ref{negative skein}} still hold over $\ZZ$.

	Now we move on to consider \text{Theorem \ref{spectral sequence}}. Using diagrams in \text{Equation \ref{filtration diagram}}, the zip and unzip homomorphisms can be defined as before. To prove the theorem, we construct a signed version of $\Phi_{I,J,K}$ and $(C,D)$. For $\Phi_{I,J,K}$, we only need to enhance the rectangle counting to an oriented version. Thanks to (1) of \text{Definition \ref{sign assignment}}, \text{Lemma \ref{composition of Phi}} still hold. Together with the sign conventions in (2)(3), one sees that $\Phi_A\circ \Phi_B$ and $\Phi_B\circ \Phi_A$ remain equal to multiplication by $-u^{(c)}$. 
	
	The group $C$ is defined as in \text{Subsection \ref{Resolution cube and spectral sequence}} except each summand is now over the base ring $\ZZ$. \text{Definition \ref{higher differential}} needs more care; we modify it as follows: 
	\begin{defi}\label{signed higher differential}
		For any pair of partitions $(I_1,J_1,K_1,L_1)<(I_2,J_2,K_2,L_2)$ , define a map \[D_{(I_1,J_1,K_1,L_1)<(I_2,J_2,K_2,L_2),\Sbold}: Z_{I_1}\cap Y_{J_1}\cap Z_{K_1}\cap Z_{L_1}\to Z_{I_2}\cap Y_{J_2}\cap Z_{K_2}\cap Z_{L_2}\]
		\[\begin{cases}
			\Phi_{I_1\cap J_2, L_2\cap J_1, L_2\cap I_1}, & \text{if } K_1=K_2;\\
			\pm 1 ,& \text{if } K_2=K_1\cup\{n\}, J_1=J_2, L_2=L_1, \text{ $n$ is a negative crossing}; \\
			\pm u^{(n)} ,& \text{if } K_2=K_1-\{n\}, J_1=J_2, I_1=I_2, \text{ $n$ is a negative crossing};\\
			\pm u^{(p)} ,& \text{if } K_2=K_1\cup\{p\}, J_1=J_2, L_2=L_1, \text{ $p$ is a positive crossing}; \\
			\pm 1, & \text{if } K_2=K_1-\{p\}, J_1=J_2, I_1=I_2, \text{ $p$ is a positive crossing} .
		\end{cases}\]
		
	\end{defi}
	The sign depends on the Maslov grading of the generator $\xv$ as well as the crossing $n$ or $p$. After fixing an order $<$ on $C(g)$, we define the sign to be \[(-1)^{M(\xv)+\vert \{c\in K_1\cap K_2| c<n (c<p)\} \vert },\] as they did in an arxiv version of \cite{MR2574747} (See also \cite{MR4777638}). Let $D_{\Sbold}$ be the sum of these new maps, ranging over all such pairs. We now show that \text{Lemma \ref{identify C with CFK}} still holds.
	
	To show that $D_{\Sbold}^2=0$, we do as before: for each $(I_1,J_1,K_1,L_1)$, we consider the $(I_2,J_2,K_2,L_2)$ component of $D^2|_{(I_1,J_1,K_1,L_1)}$. Assume the intermediate step is the summand corresponding to $(I_3,J_3,K_3,L_3)$. If the partitions $(I_1,J_1,K_1,L_1)$ and $(I_2,J_2,K_2,L_2)$ are identically the same, then things follow from $\partial^2=0$ using \text{Remark \ref{rmk on phi ijk}}. That remark also shows that for the map to be nonzero, at most two crossings change its belonging in $I$, $J$, $K$, or $L$.

	Next, assume $K_1=K_2$, but $(I_1,J_1,K_1,L_1)$ and  $(I_2,J_2,K_2,L_2)$ are not identically the same. When $K_3=K_2=K_1$, again by \text{Remark \ref{rmk on phi ijk}}, for the maps to be nonzero, we can only have one crossing change its position in $I$, $J$, $K$, or $L$. We have two subcases to consider. In one case, we have a square with two groups on the left equal to $(I_1,J_1,K_1,L_1)$ and two on the right equal to $(I_2,J_2,K_2,L_2)=(I_1-\{c\},J_1\cup \{c\},K_1,L_1)$, whose contribution is \[\partial\circ \Phi_{c,\emptyset,\emptyset}-\Phi_{c,\emptyset,\emptyset}\circ \partial=0.\] In the other case, we have $(I_3,J_3,K_3,L_3)=(I_1-\{c\},J_1\cup \{c\},K_1,L_1)$ and $(I_2,J_2,K_2,L_2)=(I_1-\{c\},J_1,K_1,L_1\cup \{c\})$, so the composition is of form $\Phi_{\{c\},\emptyset,\emptyset} \circ\Phi_{\emptyset,\{c\},\emptyset}$ which is multiplication by $-u^{(c)}$ using \text{Lemma \ref{composition to U multiplication}}. This will cancel with a component in which the intermediate group has $K_3\ne K_1$. For intermediate groups with $K_3\ne K_1$, it can only be $K_3=K_1\cup \{c\}$ using the definition of the partial order (and in this case $c\in L_2$). Now the two maps must compose to $\pm u^{(c)}$. The sign is always + since $U$ multiplication has nothing to do with the parity of $M$ and $K_1\cap K_3=K_3\cap K_2$. For each such $K_3\ne K_1$, there is exactly one partition $K'_3=K_1$ such that from $(I_1,J_1,K_1,L_1)$, $c$ first moves from $I$ to $J$, then to $L$ in $(I_2,J_2,K_2,L_2)$. This forms a square with the remaining case of $K_3=K_1$, showing that $D^2=0$ when $K_1=K_2$.
	
	When $\vert K_1\Delta K_2\vert =1$, we show the case $K_2=K_1\cup \{n\}$ and $K_2=K_1- \{n\}$ as examples, positive case can be dealt with similarly. For the first one, the two possible intermediate steps are $(I_1-\{n\},J_1,K_1\cup\{n\}, L_1)=(I_2,J_2,K_2,L_2)$ or $(I_1,J_1,K_1,L_1)$ for the maps to be nonzero. This leads to a square with two vertical maps, all being $\pm 1$ and two horizontal maps being restriction of $\partial^-$ to $Z_I\cap Y_{J}\cap Z_{K}\cap Z_{L}$. The signed convention in \text{Definition \ref{signed higher differential}} shows that it anti-commutes, that is $D_{\Sbold}^2=0$. For the second one, the only two possible intermediate groups correspond to partitions $(I_1,J_1,K_1-\{n\}, L_1\cup \{n\})=(I_2,J_2,K_2,L_2)$ and $(I_1,J_1,K_1,L_1)$ for the maps to be nonzero. The square will take the same form as in the first case (with vertical maps replaced by $\pm u^{(n)}$), showing that $D_{\Sbold}^2=0$ in this case.
	
	When $\vert K_1\Delta K_2\vert=2$, we must have one of the three cases: $K_2=K_1\cup\{c_1,c_2\}$, $K_2=K_1-\{c_1,c_2\}$ or $K_2=(K_1-\{c_1\})\cup \{c_2\}$. In the first case, for the maps to be nonzero, the intermediate step can only belong to partitions $(I_1-\{c_1\},J_1,K_1\cup\{c_1\}, L_1)$ or $(I_1-\{c_2\},J_1,K_1\cup\{c_2\},L_1)$. We again obtain a square with the same maps, up to sign on horizontal and vertical edges, both of the form $\pm u^{(c)}$ or $1$ thus Maslov grading preserving. Note that exactly one of $c_1<c_2$, $c_2<c_1$ holds, without loss of generality assume the first. Then the composition with $(I_1-\{c_1\},J_1,K_1\cup\{c_1\}, L_1)$ as an intermediate step has a negative sign, while the one  that goes through $(I_1-\{c_2\},J_1,K_1\cup\{c_2\},L_1)$ has a positive sign, showing that the square anti-commutes, thus sum to zero. This argument also applies to the remaining cases.
	
	Combining all the discussion above, we can conclude that $D_{\Sbold}^2=0$. The part of the proof that $C\simeq CFK^-(g)$ is again purely algebraic. The added sign has no effect on the identification of $D$ in the sub-quotient with $\partial^-$, so we have the desired quasi-isomorphism in oriented theory.
	
	With all these lemmas prepared, we can prove \text{Theorem \ref{spectral sequence}} as in \text{Subsection \ref{Resolution cube and spectral sequence}}. 
	
	Now we can conclude that skein relations and resolution cubes remain true for $HFK^-(L;\ZZ)$.
	
	\subsection{Orientation system versus sign assignment}\label{Orientation system versus sign assignment}
	
	The concept of a coherent orientation system in Heegaard Floer homology was first introduced in \cite[Section 3.6]{MR2113019}. Specializing to knot Floer homology, Sarkar (\cite{Sarkar2010ANO}) defined two coherent orientation systems to be weakly equivalent if they coincide on all domains without any marked points. For grid homology, he defined two sign assignments $s_1$ and $s_2$, (satisfying only assumption (1) in \text{Definition \ref{sign assignment}}), to be weakly equivalent if, for each link component $L_i$, the product of signs assigned to those horizontal and vertical annuli containing marked points on this component is identical. For clarity, in this subsection, we refer to this weaker notion as a sign assignment and denote the one from \text{Definition \ref{sign assignment}} as a true sign assignment, following \cite{MR2372850}. He showed that, for grid diagrams of knots and links in $S^3$, two orientation systems are weakly equivalent if and only if they produce to weakly equivalent sign assignments.
	
	On the other hand, in \cite{MR3412088}, Alishahi and Eftekhary  proved the existence of coherent orientation systems with certain positive boundary degeneration. This was later regarded as the canonical orientation for knot Floer theory. In \cite{eftekhary2018signassignmentslinkfloer}, Eftekhary showed that orientation systems with negative boundary degeneration lead to true sign assignments, whereas the canonical systems from \cite{MR3412088} result in false sign assignments, where the $\pm 1$ values in (2) and (3) of \text{Definition \ref{sign assignment}} are exchanged.  Nevertheless, we shall show that the true and false sign assignments lead to isomorphic grid homology groups, so our oriented grid homology coincides with the oriented theory considered in \cite{MR3412088}. Our argument relies on the construction in \cite[Theorem 3.5]{Sarkar2010ANO}.
	
	\begin{thm}
		Various versions of grid homologies for knots and links in lens spaces defined using true and false assignments over $\ZZ[U_1,\ldots,U_{2n+m}]$ coincide.
	\end{thm}
	\begin{proof}
		To see this, we show that certain representatives of chain complexes computing these two oriented homology theories can be identified directly.
		
		It was shown in \cite{MR2372850} and \cite{Celoria_2023} that true assignment exists and is unique up to gauge equivalence on grid diagrams for knots and links in $S^3$ and lens spaces, respectively. In \cite{Celoria_2023}, Celoria also showed that gauge equivalent sign assignments lead to the same isomorphic class of chain complexes so in particular, the homology group is well-defined and independent of the specific sign assignment chosen.
		
		Let us denote the horizontal annuli and vertical aunnli by $H_1,\ldots, H_N$ and $V_1,\ldots, V_N$, respectively. Here $N$ is the size of our grid diagram. In \cite[Theorem 3.5]{Sarkar2010ANO}, it was shown that given any assignment of signs to these annuli satisfying that the number of $-1$ on $H_i$'s is same as the number of $1$ on $V_i$'s modulo 2, we can construct a sign assignment with this given ``boundary condition'' using a true sign assignment and a 2-cochain $m: C _2(g)\to \{\pm 1\}$. Here the notion of cochain relies on the cell structure on the torus induced by the grid, as we mentioned at the beginning of \text{Subsection \ref{Construction of chain complex}}. In particular, we can construct a false sign assignment as follows.
		
		Fix any true sign assignment $s_0$. Consider a cochain $m:C_2(g)\to \{\pm 1\}$ defined by \[m(c)=\begin{cases}
			-1,\text{there is an $X$ or $XX$ marked point in this cell;} \\
			-1,\text{otherwise};
		\end{cases}\] on cells and extended linearly. It follows directly that $s(r):=s_0(r)m(r)$ is a false assignment. However, on each rectangle $r$ we counted in $\partial^-$, $s_0(r)=s(r)$, since we block all the $X$ base points. Thus, we can conclude that $s$ and $s_0$ lead to canonically isomorphic chain complexes.
		
		Moreover, if $s'$ is another false sign assignment, then $s_0'(r):=m(r)s'(r)$ is a true sign assignment. By the uniqueness result, $s_0$ and $s_0'$ are gauge equivalent, i.e., there exists a function $f: \boldsymbol{S} \to \{\pm 1\}$ such that for any $\xv,\yv\in \boldsymbol{S}$, and any $r\in \mathrm{Rect}(\xv,\yv)$, we have $s_0(r)=s_0'(r)\cdot f(\xv)\cdot f(\yv)$. This quickly leads to \[s(r)=s'(r)\cdot f(\xv)\cdot f(\yv)\] for any triple $(\xv,\yv,r)$ as above. Thus, a false assignment is also unique up to gauge equivalence. This together with the paragraph above conclude the desired result.

	\end{proof}
	
	Now, we can conclude that the oriented grid homologies defined in previous sections coincide with the oriented HFK theory considered in \cite{MR3412088}.

	\bibliographystyle{plain}
	\bibliography{citation}

\end{document}